\documentclass[10pt,a4paper]{amsart}

\usepackage[utf8]{inputenc}
\usepackage{amsmath}
\usepackage{amsfonts}
\usepackage{amssymb}
\usepackage{amsthm}
\usepackage[left=3.3cm,right=3.3cm,top=4cm,bottom=4cm]{geometry}
\usepackage[all]{xy}
\usepackage{color}
\usepackage{mathrsfs}
\usepackage{xypic}
\usepackage{verbatim}
\usepackage{color}
\usepackage{mathtools}
\usepackage{hyperref}
\usepackage{tikz-cd}
\usepackage{url}

\newcounter{counter}
\newcounter{theorems}

\theoremstyle{plain}
 \newtheorem{theorem}[theorems]{Theorem}
 \newtheorem{proposition}[counter]{Proposition}
 \newtheorem{lemma}[counter]{Lemma}

\theoremstyle{definition}

 \newtheorem{construction}[counter]{Construction} 
 \newtheorem{example}[counter]{Example}

\DeclareMathOperator{\codim}{codim}

\DeclareMathOperator{\End}{End}
\DeclareMathOperator{\Ext}{\mathrm{Ext}}

\DeclareMathOperator{\Hom}{Hom}

\DeclareMathOperator{\Pic}{Pic}

\DeclareMathOperator{\Tors}{Tors}

\newcommand{\del}{\partial}
\newcommand{\delbar}{\overline{\partial}}

\newcommand{\id}{\mathrm{id}}

\newcommand{\IZ}{\mathbb{Z}}

\newcommand{\IQ}{\mathbb{Q}}

\newcommand{\IQbar}{\overline{\mathbb{Q}}}

\newcommand{\IR}{\mathbb{R}}
\newcommand{\IC}{\mathbb{C}}

\newcommand{\Gm}{\mathbb{G}_m}

\newcommand{\IP}{\mathbb{P}}
\newcommand{\im}{\mathrm{im}}

\newcommand{\IA}{\mathbb{A}}

\newcommand{\hhat}{\widehat{h}}
\newcommand{\supp}{\mathrm{supp}}
\newcommand{\Div}{\mathrm{div}}

\newcommand{\pr}{\mathrm{pr}}

\newcommand{\dprime}{{\prime\prime}}

\begin{document}

\title[BHC for Semiabelian Varieties]{The Bounded Height Conjecture for Semiabelian Varieties}

\author{Lars K\"uhne}
\thanks{This work was supported by an Ambizione Grant of the Swiss National Science Foundation.}
\email{lars.kuehne@unibas.ch}
\address{Departement Mathematik und Informatik \\
	Spiegelgasse 1 \c
	4051 Basel \\
	Switzerland}


\maketitle
\begin{abstract}
The Bounded Height Conjecture of Bombieri, Masser, and Zannier states that for any sufficiently generic algebraic subvariety of a semiabelian $\IQbar$-variety $G$ there is an upper bound on the Weil height of the points contained in its intersection with the union of all algebraic subgroups having (at most) complementary dimension in $G$. This conjecture has been shown by Habegger in the case where $G$ is either a multiplicative torus or an abelian variety. However, there are new obstructions to his approach if $G$ is a general semiabelian variety. In particular, the lack of Poincaré reducibility means that quotients of a given semiabelian variety are intricate to describe. To overcome this, we study directly certain families of line bundles on $G$. This allows us to demonstrate the conjecture for general semiabelian varieties.
\end{abstract}

A generalization of the classical Manin-Mumford conjecture is the following theorem, which was proven by Raynaud \cite{Raynaud1983,Raynaud1983a} for abelian varieties, by Laurent \cite{Laurent1984} for algebraic tori, and by Hindry \cite{Hindry1988} in general. We recall that a semiabelian variety $G$ over a field $k$ is a connected smooth algebraic $k$-group that is the extension of an abelian variety by a torus.

\begin{theorem} Let $G$ be a semiabelian variety over $\overline{\mathbb{Q}}$ with torsion points $\mathrm{Tor}(G) \subseteq G(\overline{\mathbb{Q}})$. For any algebraic subvariety $X$ of $G$, there are finitely many connected algebraic subgroups $G_i$ of $G$ and finitely many torsion points $x_i \in \mathrm{Tor}(G)$ such that $\bigcup_{i=1}^{n}(G_i + x_i)$ is the Zariski closure of $X \cap \mathrm{Tor}(G)$.
\end{theorem}

More recently, another type of intersections in semiabelian varieties has been widely studied. These intersections are with algebraic subgroups instead of torsion points. Of course, investigating the intersection of $X$ with a single such subgroup is a dreary task. However, very interesting phenomena appear when intersecting $X \subseteq G$ with the countable union $G^{[s]}$ of \textit{all} algebraic subgroups having codimension $\geq s$ for some fixed integer $s$.

Since the pioneering work of Bombieri, Masser, and Zannier \cite{Bombieri1999} in this direction, two choices of $s$ are of paramount importance. If $s=\dim(X)+1$, an algebraic subgroup $H \subset G$ of codimension $\geq s$ usually does not meet $X$ at all. The intersection $X \cap G^{[s]}$ may nevertheless be dense in the Zariski topology -- even in generic cases. If $X$ is not contained in a proper algebraic subgroup of $G$, conjectures of Pink \cite{Pink2005a} and Zilber \cite{Zilber2002} imply that this never happens. Such statements about ``unlikely intersections'' are still unsettled problems, on which the reader finds a comprehensive overview in \cite{Zannier2012}. This article treats the other important and related case where $s=\dim(X)$. In this case, a generic subgroup $H \subset G$ of codimension $\geq s$ intersects $X$ already in finitely many points, and $X \cap G^{[\dim(X)]}$ can be dense with respect to the Zariski topology of $X$. The gist of the Bounded Height Conjecture (BHC) stated below is that the Weil height of the $\IQbar$-points in $X \cap G^{[\dim(X)]}$ should be nevertheless bounded from above.

In order to state this conjecture, we have to introduce some additional notions to tackle also non-generic cases. A closed irreducible subvariety $Y \subseteq G$ is called $s$-anomalous if there exists a connected algebraic subgroup $H \subseteq G$ satisfying
\begin{equation} \label{equality::codim2}
\max \{0,s - \codim_G(H)\}<\dim(Y)
\end{equation}
and a point $y \in Y(\IQbar)$ such that $Y \subseteq H + y$ (i.e., $Y$ is contained in a translate of $H$). In this situation, we say that $Y$ is associated with $H$. By $X^{(s)}$ we mean the union of all positive dimensional closed irreducible $s$-anomalous subvarieties contained in $X$. It is a corollary of Kirby's work \cite{Kirby2009} that $X^{(s)}$ (resp.\ $X \setminus X^{(s)}$) is a Zariski closed (resp.\ Zariski open) subset of $X$ (cf.\ \cite[Proposition 2.6]{Chambert-Loir2012a}). In addition, a proof allowing to determine $X^{(s)}$ effectively was given by Bombieri, Masser, and Zannier \cite{Bombieri2007} for tori and carried over to abelian varieties by R\'emond \cite{Remond2009}. 

Let now $G$ be a semiabelian variety over $\IQbar$ and $X$ a closed irreducible subvariety of $G$. To be able to work with heights, we choose a compactification $\overline{G}$ of $G$ (i.e., an open immersion $G \hookrightarrow \overline{G}$ such that $\overline{G}$ is proper). Let $L$ be a line bundle on $\overline{G}$ of $G$. Finally, let $h_L\! : \overline{G}(\IQbar)\rightarrow \IR$ be a Weil height associated with $L$. We can now state the

\vspace{0.2cm}

\textbf{Bounded Height Conjecture (BHC).} \textit{The height $h_L$ is bounded from above on the set $(X \setminus X^{(\dim(X))})(\overline{\IQ}) \cap G^{[\dim(X)]}(\overline{\IQ})$.}

\vspace{0.2cm}

This conjecture was first proposed by Bombieri, Masser, and Zannier \cite{Bombieri2007} in the case where $G$ is an algebraic torus. Even before this, they had provided a proof if $G$ is an algebraic torus and $X$ is a curve \cite{Bombieri1999}. The extension of their conjecture from tori to semiabelian varieties is merely formal and can be found in Habegger's article \cite{Habegger2009a}, where a proof of the BHC for abelian varieties is given. It is also envisaged in \cite[Théorème 1.4]{Chambert-Loir2012a}. This extension is in fact natural as semiabelian varieties have proven to be the right object for many standard conjectures in diophantine geometry (e.g.\ Manin-Mumford, Mordell-Lang, Bogomolov). In addition, they appear naturally as Jacobians of semistable curves (cf.\ \cite[Example 9.2.8]{Bosch1990}) like abelian varieties do for smooth curves so that they still retain a connection with the original study of rational points on curves. We also mention intermediate results in the direction of the BHC given by Bombieri, Masser, and Zannier \cite{Bombieri2008a}, Maurin \cite{Maurin2008, Maurin2011}, Viada \cite{Viada2003}, and Zannier \cite{Zannier2000}. Finally, let us indicate that the conjecture becomes quite generally false if $\dim(X)$ is replaced with any $s < \dim(X)$.

In parallel to his work on the BHC for abelian varieties, Habegger \cite{Habegger2009} obtained a complete proof of the conjecture for tori. Regarding the general case of the BHC, no further progress was made since his two breakthrough articles \cite{Habegger2009a, Habegger2009}. In fact, several additional problems precluded further generalizations up to now. These problems originate from the ``mixed'' nature of semiabelian varieties (i.e., the additional structures induced by the non-triviality of the extension constituting the semiabelian variety). The aim of this article is to solve these problems. Its main result, Theorem \ref{theorem::main} below, yields the BHC in general. In line with \cite{Habegger2009a}, we actually prove a stronger version of the BHC here. To announce it, we introduce certain ``height cones''; for each subset $\Sigma \subset G(\IQbar)$ and each real number $\varepsilon>0$, we define such a height cone by setting
\begin{equation} \label{equation::heightcones}
C(\Sigma, h_L, \varepsilon)= \left\{ x \in G(\IQbar)  \ | \ \exists \ a\in \Sigma, b\in G(\IQbar)\!: x=a+b\text{ and } h_L(b)\leq \varepsilon \max \{ 1, h_L(a) \} \right\}.
\end{equation}

\begin{theorem} \label{theorem::main} Let $G$ be a semiabelian variety and $\overline{G}$ a compactification endowed with an ample line bundle $L$. Furthermore, let $X$ be a closed subvariety of $G$. In addition, assume that $G$, $\overline{G}$, $L$, and $X$ are defined over $\IQbar$. Let $h_L$ be a Weil height associated with $L$. For each integer $s$, there exists some $\varepsilon>0$ such that $h_L$ is bounded from above on $(X \setminus X^{(s)})(\overline{\IQ}) \cap C(G^{[s]}(\IQbar),h_L,\varepsilon)$.
\end{theorem}

The above Theorem \ref{theorem::main} is proven in the course of Section \ref{section::proof}. We sketch the proof to compare our approach with the one of Habegger from \cite{Habegger2009a, Habegger2009}. However, \textit{not} all new obstructions are yet present in this case (see Section \ref{section::structuralproperties} and the more involved Example \ref{example::nonrational} in particular).
Ignoring some preliminary reductions (Section \ref{subsection::reductions}), the proof consists of three major steps, which we outline successively in the following. 

In the first step (Section \ref{section::approximatinghomomorphisms}), we pass from algebraic subgroups $H$ to a family of $\IQ$-line bundles parameterized by bounded subsets $\mathcal{K}$ in a (subcone of a) finite-dimensional $\IQ$-vector space $V_\IQ$. This means that with each point $\phi \in V_\IQ$ is associated an element of $\mathrm{Pic}_\IQ(\overline{G}_\phi)=\mathrm{Pic}(\overline{G}_\phi) \otimes_\IZ \IQ$ for some compactification $\overline{G}_\phi$ of $G$, which usually depends itself on $\phi$. Subsequently, we use the boundedness of $\mathcal{K}$ to approximate its members by finitely many $\IQ$-line bundles on various compactifications of $G$ (Lemma \ref{lemma::finite_approximation}). In order to reduce book-keeping to a minimum while still presenting the main new ideas, we let $E$ be an elliptic curve and \textit{assume for now} that $G$ is the extension of $E^2$ by a $2$-dimensional torus $\mathbb{G}_m^2$. Even more, we consider only $2$-dimensional subgroups $H \subset G^{[2]}$ that are extensions of $E$ by a $1$-dimensional torus $\Gm$.

To start with, we replace the subgroups $H \subset G^{[2]}$ under consideration by their associated quotients $\pi_H: G\rightarrow G^\prime = G/H$. A point $x \in G(\IQbar)$ lies on a subgroup $H$ if and only if it is contained in the kernel of $\pi_H$. One can choose compactifications $\overline{G}$, $\overline{G}^\prime$ of $G$, $G^\prime$, an algebraic map $\overline{\pi}_H: \overline{G} \rightarrow \overline{G}^\prime$ extending $\pi_H$, and an ample line bundle $L^\prime$ on $\overline{G}^\prime$ such that $x \in \ker(\pi_H)$ implies
\begin{equation}
\label{equation::heightnaive}
\hhat_{\overline{\pi}_H^\ast L^\prime}(x)=0
\end{equation}
where $\hhat_{\overline{\pi}^\ast_H L^\prime}: G(\IQ) \rightarrow \IR$ is the Néron-Tate height associated with the line bundle $\overline{\pi}^\ast_H L^\prime$ (see Section \ref{section::heights} for this notion). 

Let us first consider exclusively the case where $G$ is the \textit{trivial} extension $\Gm^2 \times E^2$. All quotients of $G$ are then likewise trivial extensions; this allows us to restrict to surjective homomorphisms $\varphi: G \rightarrow G^\prime = \Gm \times E$. Any such homomorphism extends to an algebraic map $\overline{\varphi}: G_{\overline{\Gamma(\varphi_{\mathrm{tor}})}} \rightarrow \overline{G}^\prime = \IP^1 \times E$ for some ``graph'' compactification $G_{\overline{\Gamma(\varphi_{\mathrm{tor}})}}$ of $G$ (see Construction \ref{construction2}). Fixing once and for all an ample line bundle $L^\prime$ on $\overline{G}$, each $\varphi \in V = \Hom(G,G^\prime)$ yields a line bundle $\overline{\varphi}^\ast L^\prime \in \mathrm{Pic}(G_{\overline{\Gamma(\varphi_{\mathrm{tor}})}})$. Homogenity allows us to associate also a $\IQ$-line bundle $\overline{\phi}^\ast L^\prime$ with each quasi-homomorphism $$\phi \in V_\IQ = \Hom_\IQ(G,G^\prime) = \Hom(G,G^\prime) \otimes_\IZ \IQ;$$ to be precise, we have $\overline{\phi}^\ast L^\prime \in \mathrm{Pic}_\IQ(G_{\overline{\Gamma(n \cdot \phi_{\mathrm{tor}})}}) = \mathrm{Pic}(G_{\overline{\Gamma(n \cdot \phi_{\mathrm{tor}})}}) \otimes_\IZ \IQ$ with $n$ being a denominator of $\varphi$ (i.e., $n$ is an integer such that $n \cdot \phi \in \Hom(G,G^\prime)$). In this way, we obtain a $\IQ$-line bundle for each point of $V$. Let $V^\circ \subset V$ be the subset of ``surjective'' quasi-homomorphisms (i.e., those elements $\phi \in \Hom_\IQ(G,G^\prime)$ for which there exists an integer $n \geq 1$ such that $n \cdot \phi$ is a surjective homomorphism $G \rightarrow G^\prime$).

For our purposes, some manipulation of homomorphisms (cf.\ Lemma \ref{lemma::approximation}) allows to further restrict to a bounded subset $\mathcal{K} \subset V^\circ$ such that the distance between $\mathcal{K}$ and $V \setminus V^\circ$ (with respect to any linear norm on $V$) is strictly positive. This allows us to eventually arrange for the following assertion, which corresponds to our Lemmas \ref{lemma::finite_approximation} and \ref{lemma::height_approximation}: For each $\delta>0$, there exist finitely many ``surjective'' quasi-homomorphisms $\phi_1,\dots, \phi_{K} \in V^\circ$ and a constant $c(\delta)>0$ such that, for every $x \in H \subset G^{[2]}(\IQbar)$ with $H$ as above, we have
\begin{equation}
\label{equation::heightboundintro}
\hhat_{(\overline{\phi}_k)^\ast L^\prime}(x) \leq \delta \hhat_L(x) + c(\delta)
\end{equation}
for some $k \in \{ 1, \dots, K \}$. Comparing this inequality with (\ref{equation::heightnaive}), we notice that passing to a finite family of $\IQ$-line bundles worsens the bound but that the dependence on $\hhat_L(x)$ can be curbed by choosing $\delta$ sufficiently small. So far, this is just Habegger's argument as in \cite{Habegger2009a, Habegger2009}, although the focus in his work is more on $\Hom_\IQ(G,G^\prime)$ than on the associated $\IQ$-line bundles on compactifications of $G$.

Every \textit{split} semi-abelian variety can be essentially treated in this way, relying solely on quasi-homomorphisms. For general extensions $G \in \mathrm{Ext}^1(E^2,\Gm^2)$, however, a shift to $\IQ$-line bundles instead of quasi-homomorphism becomes essential. Indeed, the quotients of such a semiabelian variety $G$ regularly fall into infinitely many different isogeny classes; compare the footnote on p.\ \pageref{footnote} for the simpler case $\mathrm{Ext}^1(E, \Gm^2)$. This means that a basic premise of Habegger's approach is \textit{not} satisfied for general extensions. In fact, repeating the above procedure does not lead to finitely many line bundles. Consequently, it does not yield an inequality like (\ref{equation::heightboundintro}) with a uniform constant $c(\delta)$. To circumvent this problem, we define suitable $\IQ$-line bundles directly on $G$. These should generalize pullbacks of line bundles along quasi-homomorphisms. 

There are some indications on how to write down such line bundles. First, it is a well-known fact that a homomorphism $\varphi$ between semiabelian varieties is describable in terms of the induced homomorphism $\varphi_{\mathrm{tor}}$ between their maximal subtori and the induced homomorphism $\varphi_{\mathrm{ab}}$ between their underlying abelian varieties (see Lemma \ref{lemma::semiabelian1}). In the situation above, it is hence natural to consider a family of $\IQ$-line bundles parameterized by the $\mathbb{Q}$-vector space
\begin{equation*}
V_\IQ = \Hom_\IQ(\Gm^2,\Gm) \times \Hom_\IQ(E^2,E),
\end{equation*}
though not every pair $(\phi_{\mathrm{tor}},\phi_{\mathrm{ab}}) \in V$ comes from an actual quasi-homomorphism between semiabelian varieties. Second, a result of Knop and Lange \cite[Theorem 2.1]{Knop1985} states that linearized line bundles on compactifications of $G$ retain a ``product-like'' shape even if $G$ is a non-trivial extension and there is no section $G \rightarrow \Gm^2$ of the inclusion $\Gm^2 \hookrightarrow G$.

The already mentioned results of Section \ref{section::compactification} (especially Construction \ref{construction3}) allow us to define for each $(\varphi_{\mathrm{tor}},\varphi_{\mathrm{ab}})$ in
$$V = \Hom(\Gm^2,\Gm) \times \Hom(E^2,E)$$ 
a compactification $G_{\overline{\Gamma(\varphi_{\mathrm{tor}})}}$ of $G$, which only depends on (the graph of) $\varphi_{\mathrm{tor}}: \Gm^2 \rightarrow \Gm$, and a $\IQ$-line bundle $L_{(\varphi_{\mathrm{tor}},\varphi_{\mathrm{ab}})}$ on $G_{\overline{\Gamma(\varphi_{\mathrm{tor}})}}$. For each homomorphism $\varphi: G \rightarrow G^\prime$ with restriction $\varphi_{\mathrm{tor}}: \Gm^2 \rightarrow \Gm$ to maximal subtori, there furthermore exists an extension $\overline{\varphi}: G_{\overline{\Gamma(\varphi_{\mathrm{tor}})}} \rightarrow \overline{G}^\prime$ such that $\overline{\varphi}^\ast L^\prime = L_{(\varphi_{\mathrm{tor}},\varphi_{\mathrm{ab}})}$ for some ample line bundle $L^\prime$ on $\overline{G}^\prime$. The line bundles $L_{(\varphi_{\mathrm{tor}},\varphi_{\mathrm{ab}})}$ play a prominent role in our proof, generalizing the pullbacks $\overline{\varphi}^\ast L^\prime$ from the split case $G = \Gm^2 \times E^2$. Surprisingly, there is neither need for a homomorphism $\varphi: G \rightarrow G^\prime$ nor for a semiabelian variety $G^\prime$ to define them. Naturally, some checking is necessary to guarantee that they simulate pullbacks along homomorphisms sufficiently well (see e.g.\ Lemmas \ref{lemma::nonnegativeheight}, \ref{lemma::heightsclosehomomorphisms} and \ref{lemma::heightsgrouplaw}).

For the next two steps of the proof, we can revert to the general case of Theorem \ref{theorem::main}. In the second step of the proof (Section \ref{section::heightbound}), we establish two concurring height bounds similar to \cite{Habegger2009a, Habegger2009}. However, the non-homogenity of the canonical height on semiabelian varieties, which decomposes into a linear and a quadratic part, is yet another problem. A sensible choice of line bundles is needed to counterbalance this in the height estimates (cf.\ the proof of Lemma \ref{lemma::height_approximation}). The first of the two said height bounds is similar to (\ref{equation::heightboundintro}). The second opposing height bound is a consequence of Siu's numerical bigness criterion (\cite[Corollary 1.2]{Siu1993}). To apply Siu's criterion, we need to estimate two types of intersection numbers related to the line bundles $L_{(\varphi_{\mathrm{tor}},\varphi_{\mathrm{ab}})}$ and the Zariski closure of $X$ in $G_{\overline{\Gamma(\varphi_{\mathrm{tor}})}}$ (Lemma \ref{lemma::alpha}). Subject to sufficiently strong estimates on these intersections numbers (as stated in Lemmas \ref{lemma::intersection1} and \ref{lemma::intersection2}), we already finish the proof of Theorem \ref{theorem::main} at this point by combining the two opposite height bounds.

In the third and last step of our proof (Section \ref{section::intersections}), we estimate these intersection numbers. Homogeneity, or the lack hereof, is once again an issue. Serious difficulties seem to arise when trying to obtain \textit{lower} bounds on intersection numbers by counting torsion points as in \cite{Habegger2009a, Habegger2009}. Although some technical tools such as \cite[Proposition 3]{Habegger2009a} were already written up more generally than strictly necessary in order to foster future generalizations, it is not clear whether this can be done at all. Therefore, we provide an alternative to this argument (Lemma \ref{lemma::intersection1}) based on hermitian differential geometry (see also Sections \ref{section::hermitiandifferentialgeometry} and \ref{section::distribution} for details). In fact, this alternative is strikingly simple in the special case of abelian varieties treated in \cite{Habegger2009a}. We obtain the sought-after lower bounds on intersection numbers by integrating appropriately chosen $(1,1)$-forms. These $(1,1)$-forms are defined in Section \ref{section::chernforms} as real interpolations of Chern forms associated with specific hermitian metrics on the line bundles $L_{(\varphi_{\mathrm{tor}},\varphi_{\mathrm{ab}})}$. On the level of $(1,1)$-forms, balancing the different homogeneities of the ``toric'' and ``abelian'' contributions is an easy task (see e.g.\ our definition (\ref{equation::definition_omega})). 

Whereas the definition of the used $(1,1)$-forms and the verification of their basic properties is almost trivial for abelian varieties (Section \ref{subsection::abelian}), tori and hence general semiabelian varieties demand considerably more work (Section \ref{subsection::toric}). The reason for this is that any invariant hermitian metric on the line bundles under consideration is merely continuous and leads to a singular Chern current supported on the maximal compact subgroup $K_{G} \subseteq G(\IC)$ (see e.g.\ \cite[Lemme 6.3]{Chambert-Loir2000}). A singular Chern current being detrimental for the application of Ax's Theorem \cite{Ax1972} in Section \ref{section::distribution}, we have to work with a less natural non-invariant hermitian metric instead. For the Chern forms associated to such a metric, establishing some natural properties is a non-trivial task; the reader may compare the proof of Lemma \ref{lemma::chernforms_quasihomomorphisms} with the evident relation (\ref{equation::abelianhomogenity}). 

It should be mentioned that Chern forms were also used by Maurin \cite{Maurin2011}, R\'emond \cite{Remond2009}, and Vojta \cite{Vojta1996} to control intersections numbers appearing in diophantine geometry. In particular, both \cite{Maurin2011} -- in the case of tori -- and \cite{Vojta1996} -- in the case of semiabelian varieties -- endow line bundles with non-invariant hermitian metrics. Apart from this, it seems that the overlap of their work with our Sections \ref{section::chernforms} and \ref{section::distribution} is rather narrow. It is nevertheless noteworthy that Ax's Theorem plays an essential role here as it does in the work of Habegger \cite{Habegger2009a, Habegger2009} and Rémond \cite{Remond2009}. In contrast to Lemma \ref{lemma::intersection1}, our proof of the supplementary \textit{upper} bounds on intersection numbers in Lemma \ref{lemma::intersection2} uses algebraic intersection theory \cite{Fulton1998} to avoid problems steming from the non-compactness of $G$.

Finally, it should be mentioned that a previous announcement \cite{Kuehne2016b} stated a non-optimal version of the first step in the proof of Theorem \ref{theorem::main}. This included a non-effective compactness argument (\cite[Lemma 2]{Kuehne2016b}), which is replaced here by the simpler Lemma \ref{lemma::finite_approximation}. The improvement is due to the systematic avoidance of quasi-homomorphisms. Related to this is our Section \ref{section::structuralproperties}. Not being logically necessary for the main proof, it illustrates why a direct use of quasi-homomorphisms as in \cite{Habegger2009a,Habegger2009} proves difficult. Quintessentially, the surjective quasi-homomorphisms from a fixed semiabelian variety $G$ to other semiabelian varieties are more or less parameterized by the $\IQ$-points of certain algebraic varieties (Theorem \ref{theorem::realizable}). However, these varieties are generally rather complicated. This is in stark contrast to the special cases of both tori and abelian varieties, where they are just affine linear spaces. Since the set of quotients, or dually the set of algebraic subgroups, of a fixed semiabelian variety $G$ is interesting in various situations beyond the results of this article (e.g., in the Manin-Mumford conjecture or more generally in the Zilber-Pink conjectures), adding these findings here seemed beneficial to further investigations. To my knowledge, neither a statement like Theorem \ref{theorem::realizable} nor an explicit non-rational counterexample as in Example \ref{example::nonrational} is anywhere mentioned, or even hinted at, in the literature so far.

It may well be that the general framework of our method (i.e., the use of bounded families of $\IQ$-line bundles in combination with real interpolations of Chern forms) gives also some leeway in problems where no group structure is present.

\textbf{Notations and conventions.} \textit{Algebraic Geometry (General).} Denote by $k$ an arbitrary field. By a $k$-variety, we mean a reduced scheme of finite type over $k$. By a \textit{subvariety} of a $k$-variety we mean a closed reduced subscheme. Note that a subvariety is determined by its underlying topological space and we frequently identify both. The tangent bundle of a $k$-variety $X$ is written $TX$ and its fiber over a point $x \in X(k)$ is denoted by $T_x X$. Furthermore, $X^{\mathrm{sm}}$ denotes the smooth locus of $X$. 

\textit{Meromorphic functions.} For every $k$-variety $X$, we write $\mathcal{K}_X$ for the sheaf of its meromorphic functions (cf.\ \cite[Definition 7.1.13]{Liu2002}). With each meromorphic function $f \in \Gamma(X, \mathcal{K}_X)$, we associate the complement $D(f)$ of its zero set (i.e., those points $x \in X$ such that $f_x \notin \mathfrak{m}_x \mathcal{O}_{X,x}$).

\textit{Products and projections.} For any product $Y_1 \times_k \dots \times_k Y_m$ of algebraic varieties, we write $\pr_i$ ($i=1,\dots,m$) for the projection $Y_1 \times_k \dots \times_k Y_m \rightarrow Y_i$ without further specification of the varieties $Y_i$. This leads to multiple different usages of the same notation $\pr_i$, sometimes close to each other. However, this should nowhere cause confusion if context is taken into account.

\textit{Algebraic groups.} An algebraic $k$-group is a group scheme of finite type over $\mathrm{Spec}(k)$. We refer to \cite[Expos\'e VI$_\text{A}$]{SGA3I} for the basic properties of algebraic $k$-groups. An algebraic $k$-subgroup of an algebraic $k$-group $G$ is a $k$-subscheme $H$ such that the group law of $G$ induces a group law on $H$. Note that $H$ is necessarily Zariski closed in $G$ (\cite[Corollaire VI$_\text{A}$.0.5.2]{SGA3I}) and of finite type over $\mathrm{Spec}(k)$. Left-multiplication by an element $g \in G(k)$ is written $l_g: G \rightarrow G$. More generally, we use the same notation $l_g$ for the left-multiplication with respect to an action $G \times X \rightarrow X$.

A split $k$-torus is an algebraic $k$-group that is isomorphic to some direct product of copies of multiplicative groups $\mathbb{G}_{m}$. A $k$-torus is an algebraic group $G$ such that its base change $G_{k^{\mathrm{sep}}}$ to the separable closure $k^{\mathrm{sep}}$ of $k$ is a split torus. A linear $k$-algebraic group is an algebraic $k$-group whose underlying scheme is both affine and connected. 

For fixed algebraic $k$-groups $G_1$ and $G_3$, the isomorphism classes of Yoneda extensions 
\begin{equation} \label{equation::extension}
\begin{tikzcd} 
0 \ar[r] & G_1 \ar[r] & G_2 \ar[r] & G_3 \ar[r] & 0
\end{tikzcd}
\end{equation}
form an abelian group $\Ext^1_k(G_1,G_3)$ with respect to Baer summation (cf.\ \cite[Section I.3]{Oort1966}).

We write $[n]_G$ for the multiplication-by-$n$ map on any commutative algebraic group $G$. The notation $\cdot_G\! : G \times G \rightarrow G$ is used for the group law of $G$ and $e_G: k \rightarrow G$ denotes the identity of $G$. We omit the reference to $G$ in these notations when this group can be inferred from context.  
We write $A^\vee$ for the dual abelian variety associated with an abelian variety $A$. Pulling back line bundles along a homomorphism $\varphi\!: A \rightarrow B$ induces a homomorphism $\varphi^\vee \!: B^\vee \rightarrow A^\vee$ of the associated dual abelian varieties.

\textit{Line bundles and linearizations.} Line bundles are denoted by capital italic letters $L, M, \dots $ whereas the corresponding calligraphic letters $\mathcal{L}, \mathcal{M}, \dots$ are reserved for the invertible sheaves formed by their sections. The line bundle dual to $L$ is written $L^\vee$. In the situation where we have an algebraic group $G$ acting on a scheme $X$, we use Mumford's definition of $G$-linearization (\cite[Definition 1.6]{Mumford1994}) for general $\mathcal{O}_X$-modules. For an invertible sheaf $\mathcal{L}$ on $X$, a $G$-linearization corresponds to an action $\varrho\! : G \times L \rightarrow L$ such that the projection $L \rightarrow X$ is $G$-equivariant. We refer to \cite[Section 1.3]{Mumford1994} for details. Given a $G$-linearized line bundle $(L,\varrho)$ we write $(L,\varrho)^{\otimes n}$ for the line bundle $L^{\otimes n}$ with the $T$-linearization induced by $\varrho$. If $\varphi\!: H \rightarrow G$ is a homomorphism from another algebraic group $H$, $Y$ a scheme with $H$-action and $f\!: Y \rightarrow X$ a $\varphi$-equivariant algebraic map, we write $f^\ast(L,\varrho)$ for the line bundle $f^\ast L$ with the induced $H$-linearization. For a $G$-equivariant closed immersion $\iota\! : Y \hookrightarrow X$, we also write $(L,\varrho)|_Y$ instead of $\iota^\ast(L,\varrho)$.

\textit{Chern classes.} With a line bundle $L$ on a projective variety, we associate a first Chern class $c_1(L)$ in the sense of \cite{Fulton1998}; we refer the reader to there for an exposition on the basic properties of Chern classes and the basic intersection theory we are using. We denote by $[X]$ the $k$-cycle class associated with an irreducible algebraic variety $X$ of dimension $k$ (in some ambient projective variety). 

\textit{Complex points and analytifications.} Throughout this article, we choose once and for all an embedding $\IQbar \hookrightarrow \IC$. For every $\IQbar$-variety $X$, we consider its complex points $X(\IC)$ as a complex (analytic) space (see \cite{Grauert1994} for this notion), the analytification of $X$. By our above convention on varieties, $X(\IC)$ is in fact reduced.


\textit{Complex spaces, differential forms, and currents.} Let $S$ be a reduced complex (analytic) space. Recall that this means that $S$ is locally biholomorphic to a closed analytic subvariety $V$ in a complex domain $U \subset \IC^n$. A function $f$ on $S$ is smooth (resp.\ holomorphic, meromorphic) if, for each such sufficiently small local chart, it is the restriction of a smooth (resp.\ holomorphic, meromorphic) function on $U$. 
In the same way, we use local charts to define plurisubharmonic functions on $S$ as restrictions. 

Similarly, a smooth differential form $\omega$ on $S$ is a differential form on the smooth locus $S^\mathrm{sm}$ of $S$ with the following additional property: $S$ can be covered by local charts $V \subset U \subset \IC^n$ as above such that  for each such chart the differential form $\omega|_{V^{\mathrm{sm}}}$ is the restriction of a smooth differential form on $U$. There are also well-defined linear operators $d$ and $d^c = i/2\pi(\delbar - \del)$ on the smooth differential forms on $S$. For each local chart $V \subset U \subset \IC^n$, these are simply the restrictions of the operators of the same name on $\IC^n$. A differential form $\omega$ on $S$ is called closed (resp.\ exact) if $d\omega = 0$ (resp.\ there exists a differential form $\omega^\prime$ on $S$ such that $d\omega^\prime = \omega$).

A line bundle $L$ over $S$ is a complex analytic space over $S$ such that $S$ can be covered by open subsets $U$ with $L|_U = U \times \IC$. A smooth hermitian metric on $L$ is a smooth (in the above sense) function $\Vert \cdot \Vert: L \rightarrow \IR$ whose restriction to each fiber over $S$ is a hermitian metric. With such a metric, we can define a Chern form $c_1(L,\Vert \cdot \Vert)$ in the usual way; if $\mathbf{s}: U \rightarrow L$ is a non-zero holomorphic section over some open subset $U \subset S$, we set $c_1(L,\Vert \cdot \Vert)|_U= dd^c (-\log \Vert \mathbf{s} \Vert )$. This construction yields a smooth differential form $c_1(L,\Vert \cdot \Vert)$ on $S$.

\section{Preliminaries on Semiabelian Varieties}
\label{section::semiabelian}

\subsection{Basics}
\label{section::basics}

Recall that a semiabelian variety $G$ over $k$ is a connected smooth algebraic $k$-group that is the extension 
\begin{equation} \label{eqn::extension}
\begin{tikzcd} 
0 \ar[r] & T \ar[r] & G \ar[r] & A \ar[r] & 0
\end{tikzcd}
\end{equation}
of an abelian variety $A$ by a torus $T$. 
Any homomorphism from a smooth linear algebraic group to an abelian variety is the zero homomorphism (see e.g.\ \cite[Lemma 2.3]{Conrad2002}). Therefore, any smooth linear algebraic subgroup of $G$ must be contained in $T$. It follows that $T$ is the maximal smooth linear algebraic subgroup of $G$. We hence call $T$ \textit{the} toric part of $G$ and $G \rightarrow G/T=A$ (or just $A$) \textit{the} abelian quotient of $G$. For a semiabelian variety $G$ over $k$, we write $\eta_G$ for the Yoneda extension class in $\Ext^1_k(A,T)$ described by (\ref{eqn::extension}). Each homomorphism $\varphi: A \rightarrow B$ (resp.\ $\varphi: T \rightarrow S$) of abelian varieties (resp.\ tori) induces a pullback $\varphi^\ast: \Ext^1_k(B,T) \rightarrow \Ext^1_k(A,T)$ (resp.\ a pushforward $\varphi_\ast: \Ext^1_k(A,T) \rightarrow \Ext^1_k(A,S)$). 

The Weil-Barsotti formula (see \cite[Section III.18]{Oort1966} or the appendix to \cite{Moret-Bailly1981}) gives a canonical identification $\Ext^1_k(A,\mathbb{G}_m) = A^\vee(k)$. If $T$ is split (i.e., $T = \mathbb{G}_m^t$) we make frequent use of the identify $\Ext^1_k(A,\mathbb{G}_m^t) = \Ext^1_k(A,\mathbb{G}_m)^t = (A^\vee)^t(k)$. 
The pullback 
\begin{equation}
\label{equation::ext_pullback}
\varphi^\ast: \Ext^1_k(B,\mathbb{G}_m^t) \longrightarrow \Ext^1_k(A,\mathbb{G}_m^t)
\end{equation} 
along a homomorphism $\varphi: A \rightarrow B$ corresponds to the $t$-fold product $\varphi^\vee \times \cdots \times \varphi^\vee$ of the dual morphism $\varphi^\vee: B^\vee \rightarrow A^\vee$. Pushforwards also allow a simple description. Indeed, let $\varphi: \mathbb{G}_m^{t} \rightarrow \mathbb{G}_m^{t^\prime}$ be the homomorphism described by $\varphi^\ast(Y_v)= \prod_{u=1}^{t} X_u^{a_{uv}}$ in standard coordinates $X_1,\dots, X_{t}$ (resp.\ $Y_1,\dots, Y_{t_2}$) on $\mathbb{G}_m^{t}$ (resp.\ $\mathbb{G}_m^{t^\prime}$). Then, the pushforward
\begin{equation}
\label{equation::ext_pushforward}
\varphi_\ast: \Ext^1_k(A,\Gm^{t}) \longrightarrow \Ext^1_k(A,\Gm^{t^\prime})
\end{equation}
corresponds to the homomorphism $(A^\vee)^{t} \rightarrow (A^\vee)^{t^\prime}$ sending $(\eta_1, \dots, \eta_{t})$ to $(\sum_{u=1}^{t}a_{uv}\eta_u)_{1 \leq v\leq t^\prime}$.

As for abelian varieties, one calls two semiabelian varieties $G, G^\prime$ isogeneous if there exists an isogeny $G \rightarrow G^\prime$ (i.e., a surjective homomorphism $G \rightarrow G^\prime$ with finite kernel). Evidently, the multiplication-by-$n$ homomorphism $[n]$ of a semiabelian variety is an isogeny. As for abelian varieties, this yields an equivalence relation on semiabelian varieties. 

Finally, we note that quotients as well as smooth algebraic subgroups of a semiabelian variety are themselves semiabelian varieties (cf.\ \cite[Corollary 5.4.6]{Brion2017}). In particular, the algebraic subgroups appearing in Theorem \ref{theorem::main} are all semiabelian varieties because a well-known result of Cartier (\cite[Corollaire VI$_\text{B}$.1.6.1]{SGA3I}) states that all algebraic $k$-groups are smooth if $k$ has characteristic $0$.

\subsection{Homomorphisms and quasi-homomorphisms}
\label{section::homomorphisms}

We recall the fundamental result on homomorphisms between semiabelian varieties.

\begin{lemma} \label{lemma::semiabelian1} 
Let $G$ (resp.\ $G^\prime$) be a semiabelian variety over $k$ such that $A$ (resp.\ $A^\prime$) is the abelian quotient and $T$ (resp.\ $T^\prime$) is the toric part of $G$ (resp.\ $G^\prime$). For any homomorphism $\varphi: G\rightarrow G^\prime$ there exist unique homomorphisms $\varphi_{\mathrm{tor}}: T \rightarrow T^\prime$ and $\varphi_{\mathrm{ab}}: A \rightarrow A^\prime$ such that
\begin{equation} \label{equation::snakelemma}
\begin{tikzcd} 
0 \ar[r] & T \ar[r] \ar[d, "\varphi_{\mathrm{tor}}"] & G \ar[r] \ar[d, "\varphi"] & A \ar[r] \ar[d, "\varphi_{\mathrm{ab}}"] & 0 \\
0 \ar[r] & T^\prime \ar[r] & G^\prime \ar[r] & A^\prime \ar[r] & 0
\end{tikzcd}
\end{equation}
is a homomorphism of exact sequences. Furthermore, the induced map 
\begin{equation} \label{equation::injection}
\Hom(G,G^\prime) \longrightarrow \Hom(T,T^\prime) \times \Hom(A,A^\prime), \ \varphi \longmapsto (\varphi_{\mathrm{tor}}, \varphi_{\mathrm{ab}}),
\end{equation}
is an injective homomorphism with image
\begin{equation} \label{equation::image}
\{ (\varphi_{\mathrm{tor}}, \varphi_{\mathrm{ab}}) \in \Hom(T,T^\prime) \times \Hom(A,A^\prime) \ | \  (\varphi_{\mathrm{tor}})_\ast \eta_{G} = (\varphi_a)^\ast \eta_{G^\prime} \text{ in $\Ext^1_k(A,T^\prime)$} \}.
\end{equation}
\end{lemma}

This lemma is well-known in the literature (see e.g.\ \cite{Bertrand1983} or \cite{Villareal2008}). In fact, the existence of a pair $(\varphi_{\mathrm{ab}}, \varphi_{\mathrm{tor}})$ follows directly from the fact that any map from a smooth linear algebraic group to an abelian variety is zero (\cite[Lemma 2.3]{Conrad2002}) and its uniqueness is obvious. The remaining assertions can be shown by standard homological algebra in the category of commutative algebraic $k$-groups, which is abelian by a result of Grothendieck \cite[Th\'eor\`eme VI$_\text{A}$.5.4.2]{SGA3I}. By the snake lemma, the homomorphism $\varphi$ is surjective (resp.\ an isogeny) if and only if both $\varphi_{\mathrm{tor}}$ and $\varphi_{\mathrm{ab}}$ are surjective (resp.\ isogenies). All of this is contained in \cite[Chapter VII]{Serre1988}, described in a pre-schematic language.

In the situation of Lemma \ref{lemma::semiabelian1} we call $\varphi_{\mathrm{tor}}$ (resp.\ $\varphi_{\mathrm{ab}}$) the toric (resp.\ abelian) component of $\varphi$. In addition, we say that $\varphi$ is represented by the pair $(\varphi_{\mathrm{tor}}, \varphi_{\mathrm{ab}})$ and, conversely, that $(\varphi_{\mathrm{tor}},\varphi_{\mathrm{ab}})$ represents $\varphi$. We state an immediate consequence of Lemma \ref{lemma::semiabelian1} for later reference as a separate lemma.

\begin{lemma} \label{lemma::semiabelian3} Assume that $k$ is algebraically closed. Let $G$ be a semiabelian variety over $k$ with abelian quotient $A$ and toric part $\mathbb{G}_m^{t}$. For every homomorphism $\varphi_{\mathrm{tor}}\!: \mathbb{G}_m^{t} \rightarrow \mathbb{G}_m^{t^\prime}$ and every isogeny $\varphi_{\mathrm{ab}}\!: A \rightarrow B$ there exists a semiabelian variety $G^\prime$ over $k$ and a homomorphism $\varphi\!: G \rightarrow G^\prime$ represented by $(\varphi_{\mathrm{tor}},\varphi_{\mathrm{ab}})$.
\end{lemma}

\begin{proof} Write $(\varphi_{\mathrm{tor}})_\ast \eta_{G}=(\eta_1^\dprime,\cdots, \eta_{t^\prime}^\dprime) \in (A^\vee)^{t^\prime}(k)$. Since $\varphi_{\mathrm{ab}}^\vee: B^\vee \rightarrow A^\vee$ is an isogeny (cf.\ \cite[Remark (3) on p.\ 81]{Mumford1970}), there exist $\eta_i^\prime \in B^\vee(k)$ such that $\eta_i^\dprime = \varphi_{\mathrm{ab}}^\vee (\eta_i^\prime)$. Let $G^\prime$ be the semiabelian variety described by $\eta_{G^\prime}=(\eta_1^\prime,\dots,\eta_{t^\prime}^\prime) \in (B^\vee)^{t^\prime}(k)=\Ext^1_k(B,\Gm^{t^\prime})$. As $(\varphi_{\mathrm{tor}})_\ast \eta_G  = (\varphi_{\mathrm{ab}})^\ast \eta_{G^\prime}$, there exists a homomorphism $\varphi: G \rightarrow G^\prime$ representing $(\varphi_{\mathrm{tor}}, \varphi_{\mathrm{ab}})$ by Lemma \ref{lemma::semiabelian1}.
\end{proof}

We need to work also with quasi-homomorphisms of semiabelian varieties. First of all, note that for any semiabelian varieties $G$ and $G^\prime$ the $\IZ$-module $\Hom(G,G^\prime)$ of homomorphisms is torsion-free. Indeed, this is true for both tori and abelian varieties so that we may infer the general case from Lemma \ref{lemma::semiabelian1}. By quasi-homomorphisms we mean the elements of $\Hom_\IQ(G,G^\prime)=\Hom(G,G^\prime)\otimes_\IZ \IQ$.  In analogy to actual homomorphisms, each quasi-homomorphism is denoted in the form $\phi\!: G \rightarrow_\IQ G^\prime$. By tensoring (\ref{equation::injection}) with $\IQ$, we can also associate with each quasi-homomorphism $\phi\!: G \rightarrow_\IQ G^\prime$ uniquely a toric component $\phi_{\mathrm{tor}}\!: T \rightarrow_\IQ T^\prime$ and an abelian component $\phi_{\mathrm{ab}}\!: A \rightarrow_\IQ A^\prime$. With each quasi-homomorphism $\phi\!: G \rightarrow_\IQ G^\prime$ we can associate a ``kernel up to torsion'' $\ker(\phi)+\Tors(G)$ 
in the following way: Let $n$ be a denominator of $\phi$ (i.e., $n \cdot \phi \in \Hom(G,G^\prime)$) and set $\ker(\phi)+\Tors(G)=\ker(n \cdot \phi)+\Tors(G)$. 
Additionally, we say that $\phi$ is surjective if $n \cdot \phi$ is. These definitions are clearly independent of the chosen denominator $n$. Albeit a quasi-homomorphism $\phi_{\mathrm{ab}} \in \Hom_\IQ(A^\prime, A)$ does not induce a pullback as in (\ref{equation::ext_pullback}), it gives rise to a homomorphism
\begin{equation*}
(\phi_{\mathrm{ab}})^{\ast,\IQ}: \Ext^1_k(A,\Gm^t)_\IQ = (A^\vee(k) \otimes_\IZ \IQ)^t \longrightarrow ((A^\prime)^\vee(k) \otimes_\IZ \IQ)^{t} = \Ext^1_k(A^\prime,\Gm^t)_\IQ.
\end{equation*}
Similarly, a quasi-homomorphism $\phi_t \in \Hom_\IQ(\Gm^t,\Gm^{t^\prime})$ induces a homomorphism
\begin{equation*}
(\phi_{\mathrm{tor}})_{\ast,\IQ}: \Ext^1_k(A,\Gm^t)_\IQ = (A^\vee(k) \otimes_\IZ \IQ)^t \longrightarrow (A^\vee(k) \otimes_\IZ \IQ)^{t^\prime} = \Ext^1_k(A,\Gm^{t^\prime})_\IQ
\end{equation*}
in place of (\ref{equation::ext_pushforward}).
\section{Compactifications} \label{section::compactification}

To compactify semiabelian varieties we use a well-known construction proposed by Serre (cf.\ \cite[Section 3.2]{Serre2001} and Serre's appendix in \cite{Waldschmidt1979}). Let $G$ be a semiabelian variety over $k$ with split toric part $T = \Gm^t$ and abelian quotient $A$.\footnote{Using descent along a finite Galois extension of $k^\prime/k$ (compare \cite[Example 6.2.B]{Bosch1990}) such that $T_{k^\prime}$ splits, one can get rid of the splitting assumption a posteriori but we do not need this generality.} Furthermore, let $\overline{T}$ be a $T$-equivariant compactification of $T$. This means that we are given a dense open immersion $T \hookrightarrow \overline{T}$ with $\overline{T}$ a proper $k$-variety and that there is an extension $\cdot_{\overline{T}}\!: T \times \overline{T} \rightarrow \overline{T}$ of the group law $\cdot_T\!: T \times T \rightarrow T$. We endow $G \times_k \overline{T}$ with the $T$-action given by 
\begin{equation} \label{equation::taction}
t \cdot (g, \overline{t}) = (t \cdot_G g, t^{-1} \cdot_{\overline{T}} \overline{t}), \ t \in T(S), \overline{t} \in \overline{T}(S), g \in G(S),
\end{equation}
on $S$-points. It is well-known that the (categorical) quotient $G_{\overline{T}}=G \times_k \overline{T} / T$ in the category of $k$-schemes exists and is a proper $k$-variety (see e.g.\ \cite{Faltings1984, Knop1984}). In fact, there exists a (finite) Zariski covering $\{U_i\}$ of $A$ together with compatible $T$-equivariant trivializations $\phi_i: U_i \times_A G \rightarrow U_i \times_k T$ over each $U_i$. The isomorphisms
\begin{equation*}
\phi_j \circ \phi_i^{-1}|_{(U_i \cap U_j) \times_k T}: (U_i \cap U_j) \times_k T \longrightarrow (U_i \cap U_j) \times_k T
\end{equation*}
determine sections $t_{ij} \in \Gamma(U_i \cap U_j, T)$. The variety $G$ can be described as a gluing of these trivial $T$-torsors by means of the \v{C}ech cocycle $\{t_{ij}\} \in \check{\mathrm{H}}^1(\{U_i\}, T)$. (In fact, $\{t_{ij}\}$ is also the cocycle describing $\eta_G \in (A^\vee)^t(k)=\Ext^1_k(A,T)$ in the Barsotti-Weil formula.) Via the extension $\cdot_{\overline{T}}$ of the group law $\cdot_T$, the same \v{C}ech cocycle $\{ t_{ij}\}$ determines also a gluing of the $k$-varieties $U_i \times_k \overline{T}$, yielding a proper $k$-variety $X$ and a projection $\overline{\pi}\!: G_{\overline{T}} \rightarrow A$. There is a canonical map $p\!: G \times_k \overline{T}\rightarrow G_{\overline{T}}$ over $A$ such that its base change
\begin{equation*}
p \times_{G_{\overline{T}}} (U_i \times_k \overline{T}) \!: (U_i \times_A G) \times_k \overline{T} \longrightarrow U_i \times_k \overline{T}
\end{equation*}
coincides with the action
\begin{equation*}
U_i \times_k \cdot_{\overline{T}}\!: U_i \times_k (T \times_k \overline{T}) \longrightarrow U_i \times_k \overline{T}\end{equation*}
under the identification $U_i \times_A G = U_i \times_k T$ described by $\phi_i$. Neither $G_{\overline{T}}$ nor $p$ depends on the Zariski covering $\{ U_i \}$ as the above is compatible with any further refinement. 
In addition, the $G$-action given by the group law $+_G\!: G \times G \rightarrow G$ extends uniquely to an action $G \times G_{\overline{T}} \rightarrow G_{\overline{T}}$.

If $(M, \varrho\!: T \times_k M \rightarrow M)$ is a $T$-linearized line bundle on $\overline{T}$, we endow $G \times_{k} M$ with a $T$-action in a way similar to (\ref{equation::taction}) and form the quotient $G(M, \varrho)= G \times_k M /T$. Repeating the above procedure, it is easy to infer that $G(M,\varrho)$ is a line bundle over $G_{\overline{T}}$. One checks also a compatibility $G(M \otimes M^\prime, \varrho \otimes \varrho^\prime) \approx G(M,\varrho) \otimes G(M^\prime, \varrho^\prime)$ with tensor products.

\begin{lemma} \label{lemma::ampleness} Let $(M,\varrho)$ be an ample $T$-linearized line bundle on $\overline{T}$ and $N$ an ample line bundle on $A$. Then, the line bundle $G(M,\varrho)\otimes \overline{\pi}^\ast N$ (resp.\ $G(M,\varrho)$) is ample (resp.\ nef). 
\end{lemma}

\begin{proof} By \cite[Example 1.2.22]{Lazarsfeld2004}, the line bundle $M^{\otimes 3k}$ is normally generated for sufficiently large integers $k$. This allows us to apply \cite[Theorem 3.5]{Knop1984}, which yields that $G(M,\varrho)^{\otimes 3k} \otimes \overline{\pi}^\ast N^{\otimes 3k} = (G(M,\varrho) \otimes \overline{\pi}^\ast N)^{\otimes 3k}$ is normally generated\footnote{We use this notion as in \cite{Knop1984,Mumford1970a}. In particular, it is not required that $G_{\overline{T}}$ is normal.} and hence very ample (cf.\ \cite[p.\ 38]{Mumford1970a} for this final implication).\footnote{The author thanks Friedrich Knop for acknowledging a gap in the proof of \cite[Lemma 1.7]{Knop1984} and for pointing out this argument.}
	
For nefness, let $C$ be a proper curve in $G_{\overline{T}}$. We already know that $G(M,\varrho)^{\otimes k} \otimes (\overline{\pi}^\ast N) = G(M^{\otimes k},\varrho^{\otimes k}) \otimes (\overline{\pi}^\ast N)$ is ample for any integer $k\geq 1$. Hence, the degree of the $0$-cycle class
\begin{equation*}
k c_1(G(M,\varrho)) \cap [C] + c_1(\overline{\pi}^\ast N) \cap [C]
\end{equation*}
is positive for any $k$ (see e.g.\ \cite[Lemma 12.1]{Fulton1998}). Dividing by $k$ and taking the limit $k\rightarrow \infty$, we obtain
\begin{equation*}
\deg(c_1(G(M,\varrho)) \cap [C]) \geq 0,
\end{equation*} 
which means that $G(M,\varrho)$ is nef.
\end{proof}

We are interested in the behavior of the above constructions with regard to homomorphisms. For this, let $\varphi\!: G \rightarrow G^\prime$ be a homomorphism of semiabelian varieties with toric component $\varphi_{\mathrm{tor}}\!: T \rightarrow T^\prime$ as in (\ref{equation::snakelemma}). In addition, let $\overline{T}$ (resp.\ $\overline{T}^\prime$) be a $T$-equivariant (resp.\ $T^\prime$-equivariant) compactification of $T$ (resp.\ $T^\prime$) so that $\varphi_{\mathrm{tor}}$ extends to a $\varphi_{\mathrm{tor}}$-equivariant map $\overline{\varphi}_{\mathrm{tor}}\!: \overline{T} \rightarrow \overline{T}^\prime$. Endowing $G \times_k \overline{T}$ (resp.\ $G^\prime \times_k \overline{T}^\prime$) with a $T$-action (resp.\ $T^\prime$-action) as in (\ref{equation::taction}), the $\varphi_{\mathrm{tor}}$-equivariant map $\varphi \times_k \overline{\varphi}_{\mathrm{tor}}\!:G \times_k \overline{T} \rightarrow G^\prime \times_k \overline{T}^\prime$ induces a map $\overline{\varphi}\!: G_{\overline{T}} \rightarrow G_{\overline{T}^\prime}^\prime$. Let now $(M,\varrho)$ be a $T^\prime$-linearized line bundle on $\overline{T}^\prime$. We have $
\overline{\varphi}^\ast G^\prime(M, \varrho) \approx G(\overline{\varphi}_{\mathrm{tor}}^\ast (M,\varrho))$; for the line bundle $G \times_k \overline{\varphi}_{\mathrm{tor}}^\ast M$ is the pullback of $G^\prime \times_k M$ along $\varphi \times_k \overline{\varphi}_{\mathrm{tor}}$ and the induced map $G \times_k \overline{\varphi}_{\mathrm{tor}}^\ast M \rightarrow G^\prime \times_k M$ is $\varphi_{\mathrm{tor}}$-equivariant.

In these considerations, the case where $\varphi$ is the multiplication-by-$n$ homomorphism $[n]_G$ for a semiabelian variety $G$ with toric part $T$ is of particular importance. \textit{To avoid pathologies, some further technical requirements on both the $T$-equivariant compactification $\overline{T}$ and the $T$-linearizable line bundle $M$ should be met.} First, an extension of $[n]_T$ to a map $[n]_{\overline{T}}\!: \overline{T} \rightarrow \overline{T}$ should exist for each integer $n$. (Such an extension is unique by separatedness.) Under this condition, there is an extension $\overline{\varphi}=[n]_{\overline{G}}\!: G_{\overline{T}}\rightarrow G_{\overline{T}}$ of $[n]_G$ by the last paragraph. Second, there should be a $T$-equivariant isomorphism $[n]_{\overline{T}}^\ast M \approx M^{\otimes |n|}$. If this is satisfied, the last assertion of the preceding paragraph specializes to $[n]_{\overline{G}}^\ast G(M,\varrho) \approx G([n]^{\ast}_{\overline{T}}(M,\varrho))\approx G((M,\varrho)^{\otimes |n|})\approx G(M, \varrho)^{\otimes |n|}$.

 Before introducing the two types of compactifications to be employed in our proof of Theorem \ref{theorem::main}, we recall a further notion. Let $T$ be a torus with $T$-equivariant compactification $\overline{T}$. Pulling meromorphic functions back fabricates a $T$-linearization of $\mathcal{K}_{\overline{T}}$. Denote by $\pr_2\!: T \times \overline{T} \rightarrow \overline{T}$ the projection and by $\sigma\!: T \times \overline{T} \rightarrow \overline{T}$ the $T$-action on $\overline{T}$. A Cartier divisor $D$ on $\overline{T}$ is called $T$-invariant if the pullbacks $\pr_2^\ast D$ and $\sigma^\ast D$ are equal. In this case, $D$ gives rise to a $T$-invariant invertible subsheaf $\mathcal{O}(D)$ of $\mathcal{K}_{\overline{T}}$. Hence, there is an induced $T$-linearization on $\mathcal{O}(D)$. We always mean this linearization when associating a $T$-linearized line bundle $(L(D),\varrho_D)$ with a $T$-invariant Cartier divisor $D$. Note that this $T$-linearization on $\mathcal{O}(D)$ is uniquely characterized by the fact that its rational section $1 \in \mathcal{K}_{\overline{T}}(\overline{T})$ is $T$-invariant. 
 
Any $T$-invariant Cartier divisor $D$ on $\overline{T}$ yields naturally a Cartier divisor on $G_{\overline{T}}$. Indeed, assume that $D$ is represented by $(V_j,f_j)$ with Zariski opens $V_j \subset \overline{T}$. For each Zariski open $U_i \subset A$ this gives a Cartier divisor on $U_i \times_A G_{\overline{T}} = U_i \times_k \overline{T}$ that is represented by $(U_i \times_k V_j, f_j \circ \pr_2)$. By $T$-invariance, these Cartier divisors glue together to a Cartier divisor $G(D)$ on $G_{\overline{T}}$. Furthermore, it is easy to see that $L(G(D))$ is isomorphic to $G(L(D),\varrho_D)$.

\begin{construction}[$D_t$, $(M_t, \varrho_t)$] \label{construction0} The torus $\Gm=\mathrm{Spec}(k[X,X^{-1}])$ has a $\Gm$-equivariant compactification $\iota_1: \Gm \hookrightarrow \IP^1 = \mathrm{Proj}(k[Z_1,Z_2])$ with $\iota_1^\ast (Z_2/Z_1)= X$. There is an extension $[n]_{\IP^1}: \IP^1  \rightarrow \IP^1 $ of $[n]_{\Gm}: \Gm \rightarrow \Gm$. Let $E_0$ (resp.\ $E_\infty$) be the $\Gm$-invariant Cartier divisor on $\IP^1$ represented by 
\begin{equation*}
(D(Z_1), Z_2/Z_1) \text{ and } (D(Z_2), 1) \text{ (resp.\ } (D(Z_1), 1) \text{ and } (D(Z_2), Z_1/Z_2)\text{).}
\end{equation*}
For the torus $T=\Gm^t$, the map $\iota_t=\iota_1 \times \cdots \times \iota_1\!: \Gm^t \hookrightarrow \overline{T}=(\IP^1)^t$ gives a $T$-equivariant compactification. Denoting by $\pr_i\!:\overline{T}=(\IP^1)^t \rightarrow \IP^1$ the projection to the $i$-th component, we set $D_t = \sum_{1 \leq i \leq t} \pr_i^\ast(E_0+E_\infty)$ and $\mathcal{M}_{t}=\mathcal{O}(D_t)$. By the above, there is a natural $T$-linearization $\varrho_{t}=\varrho_{D_t}$ on the associated line bundle $M_{t}$ that acts trivially on its global section $1 \in \mathcal{M}_t(\overline{T})$. Furthermore, from the evident identity $[n]^\ast_{\overline{T}} D_t = |n| \cdot D_t$ of Cartier divisors we obtain an identity $[n]^\ast_{\overline{T}} \mathcal{M}_{t} = \mathcal{M}^{\otimes |n|}_{t}$ of $\mathcal{O}_X$-submodules of $\mathcal{K}_{\overline{T}}$ so that $[n]_{\overline{T}}^\ast (M_t, \varrho_t) = (M_t, \varrho_t)^{\otimes |n|}$.
\end{construction}

\begin{construction}[$\overline{G}$, $M_{\overline{G}}$, $G(D_t)$] \label{construction1}
Given a semiabelian variety $G$ having split toric part $T=\Gm^t$ and abelian quotient $\pi\!: G \rightarrow A$, we use the $T$-equivariant compactification $\iota_t: \Gm^t \hookrightarrow (\IP^1)^t=\overline{T}$ constructed above to obtain a smooth compactification $\overline{G}=G_{\overline{T}}$ with abelian quotient $\overline{\pi}: \overline{G} \rightarrow A$. The $T$-invariant line bundle $(M_t,\varrho_t)$ yields further a line bundle $M_{\overline{G}}=G(M_{t},\varrho_{t})$ on $\overline{G}$, which satisfies $[n]_{\overline{G}}^\ast M_{\overline{G}} \approx M_{\overline{G}}^{\otimes |n|}$. In addition, the line bundle $M_{\overline{G}}$ is associated with the Cartier divisor $G(D_t)$.

We remark that this compactification also appears in \cite{Chambert-Loir1999,Chambert-Loir2000,Vojta1996,Vojta1999} and Serre's appendix to \cite{Waldschmidt1979}.
\end{construction}

\begin{construction}[$G_{\overline{\Gamma(\varphi_{\mathrm{tor}})}}$, $M_{\overline{\Gamma(\varphi_{\mathrm{tor}})}}$, $\pi_{\overline{\Gamma(\varphi_{\mathrm{tor}})}}$] \label{construction3} Assume given a semiabelian variety $G$ with split toric part $\Gm^{t}$ and abelian quotient $\pi: G \rightarrow A$ as well as a homomorphism $\varphi_{\mathrm{tor}} \in \Hom(\Gm^{t},\Gm^{t^\prime})$. Let $\Gamma(\varphi_{\mathrm{tor}}) \subset \Gm^{t} \times \Gm^{t^\prime}$ be the graph of $\varphi_{\mathrm{tor}}$ and $\overline{\Gamma(\varphi_{\mathrm{tor}})}$ its Zariski closure in the $(\Gm^{t} \times \Gm^{t^\prime})$-equivariant compactification $(\mathbb{P}^1)^{t} \times (\mathbb{P}^1)^{t^\prime}$.\footnote{The reader is warned that the Zariski closure $\overline{\Gamma(\varphi_{\mathrm{tor}})}$ is not normal, but that we also have no use for its normality.} The projection to $\Gm^{t}$ induces an identification $\Gamma(\varphi_{\mathrm{tor}}) = \Gm^{t}$. In this way, $\overline{\Gamma(\varphi_{\mathrm{tor}})}$ can be considered as a $\Gm^{t}$-equivariant compactification of $\Gm^{t}$. As $[n]_{\Gamma(\varphi_{\mathrm{tor}})}$ is just the restriction of $[n]_{\Gm^{t} \times \Gm^{t^\prime}}$, it clearly extends to $\overline{\Gamma(\varphi_{\mathrm{tor}})}$ because $[n]_{\Gm^{t} \times \Gm^{t^\prime}}$ extends to $(\mathbb{P}^1)^{t} \times (\mathbb{P}^1)^{t^\prime}$. Therefore, there is an extension of $[n]_{G}$ to the ``graph compactification'' $G_{\overline{\Gamma(\varphi_{\mathrm{tor}})}}$. To fix notations, we record a self-explanatory commutative diagram 
\begin{equation} \label{equation::graphdiagram2}
\begin{tikzcd} 
\arrow[d, hook] (\IP^1)^{t} & \arrow[d, hook] \arrow[l, hook'] \arrow[r,hook] \Gm^{t} & \arrow[r, "\mathrm{pr}_2"] \arrow[ll, bend right, "\mathrm{pr}_1"'] \arrow[d, hook] \overline{\Gamma(\varphi_{\mathrm{tor}})} &  (\IP^1)^{t^\prime}
\\
\arrow[d] \overline{G} & \arrow[d, "\pi"] \arrow[l, hook'] \arrow[r, "\iota_{\overline{\Gamma(\varphi_{\mathrm{tor}})}}", hook] G & \arrow[ll, bend right, "q_{\overline{\Gamma(\varphi_{\mathrm{tor}})}}"' near start] \arrow[d, "\pi_{\overline{\Gamma(\varphi_{\mathrm{tor}})}}"] G_{\overline{\Gamma(\varphi_{\mathrm{tor}})}}
\\
A & \arrow[l, "\mathrm{id}_{A}"'] \arrow[r, "\mathrm{id}_{A}"] A & A.
\end{tikzcd}
\end{equation}
Construction \ref{construction0} gives a $\Gm^{t^\prime}$-linearized line bundle $(M_{t^\prime},\varrho_{t^\prime})$ on $(\IP^1)^{t^\prime}$. Its $(\Gm^{t} \times \Gm^{t^\prime})$-linearized pullback $\pr_2^\ast (M_{t^\prime},\varrho_{t^\prime})$ yields a line bundle $M_{\overline{\Gamma(\varphi_{\mathrm{tor}})}} = G_{\overline{\Gamma(\varphi_{\mathrm{tor}})}}(\pr_2^\ast(M_{t^\prime},\varrho_{t^\prime})|_{\overline{\Gamma(\varphi_{\mathrm{tor}})}})$ on $G_{\overline{\Gamma(\varphi_{\mathrm{tor}})}}$. Setting $\varphi_{\mathrm{tor}}=\id_{\Gm^{t}}$, this construction specializes to Construction \ref{construction1} above (i.e., $\overline{G} \approx G_{\overline{\Gamma(\id_{\Gm^t})}}$  with compatible $M_{\overline{G}} \approx M_{\overline{\Gamma(\id_{\Gm^t})}}$).

For any non-zero integer $n$, we can relate $(G_{\Gamma(\varphi_{\mathrm{tor}})},M_{\Gamma(\varphi_{\mathrm{tor}})})$ with $(G_{\Gamma(n \cdot \varphi_{\mathrm{tor}})},M_{\Gamma(n\cdot \varphi_{\mathrm{tor}})})$. For this, we define $G^\prime$ and $G^\dprime$ to be the semiabelian varieties such that $\eta_{G^\prime} = (\eta_G, (\varphi_{\mathrm{tor}})_\ast \eta_G)$ and $\eta_{G^\dprime} =(\eta_G, (n \cdot \varphi_{\mathrm{tor}})_\ast \eta_G)$ in $\Ext^1(A, \Gm^{t}\times \Gm^{t^\prime})$. The equivariant closed immersions $\overline{\Gamma(\varphi_{\mathrm{tor}})}, \overline{\Gamma(n\cdot \varphi_{\mathrm{tor}})} \subset (\IP^1)^{t} \times (\IP^1)^{t^\prime}$ yield closed immersions $G_{\overline{\Gamma(\varphi_{\mathrm{tor}})}} \subset \overline{G}^\prime$ and $G_{\overline{\Gamma(n \cdot \varphi_{\mathrm{tor}})}} \subset \overline{G}^\dprime$. In addition, the finite morphism $[1]_{(\IP^1)^{t}}\times [n]_{(\IP^1)^{t^\prime}}$ yields a finite map $\vartheta_n: \overline{G}^\prime \rightarrow \overline{G}^\dprime$. As $[1]_{(\IP^1)^{t}}\times [n]_{(\IP^1)^{t^\prime}}$ restricts to a $\Gm^{t}$-equivariant birational map $\overline{\Gamma(\varphi_{\mathrm{tor}})} \rightarrow \overline{\Gamma(n \cdot \varphi_{\mathrm{tor}})}$, $\vartheta_{n}$ restricts to a birational map $\vartheta_{\varphi_{\mathrm{tor}},n}: G_{\overline{\Gamma(\varphi_{\mathrm{tor}})}} \rightarrow G_{\overline{\Gamma(n\cdot \varphi_{\mathrm{tor}})}}$. Furthermore,
\begin{equation*}
\vartheta_{\varphi_{\mathrm{tor}},n}^\ast M_{\overline{\Gamma(n\cdot \varphi_{\mathrm{tor}})}} 
\approx
\vartheta_{n}^\ast
\overline{G}^\dprime(\pr_2^\ast(M_{t^\prime},\varrho_{t^\prime}))|_{\overline{\Gamma(n \cdot \varphi_{\mathrm{tor}})}}
\approx
\overline{G}^\prime(\pr_2^\ast(M_{t},\varrho_{t})^{\otimes |n|})|_{\overline{\Gamma(\varphi_{\mathrm{tor}})}}
\approx
M_{\overline{\Gamma(\varphi_{\mathrm{tor}})}}^{\otimes |n|}.
\end{equation*}
In addition, there are the evident relations $\pi_{\overline{\Gamma(\varphi_{\mathrm{tor}})}} = \pi_{\overline{\Gamma(n \cdot \varphi_{\mathrm{tor}})}} \circ \vartheta_{\varphi_{\mathrm{tor}},n}$, $q_{\overline{\Gamma(\varphi_{\mathrm{tor}})}} = q_{\overline{\Gamma(n \cdot \varphi_{\mathrm{tor}})}} \circ \vartheta_{\varphi_{\mathrm{tor}},n}$ and $\iota_{\overline{\Gamma(n\cdot \varphi_{\mathrm{tor}})}} = \vartheta_{\varphi_{\mathrm{tor}},n} \circ \iota_{\overline{\Gamma(\varphi_{\mathrm{tor}})}}$.
\end{construction}

\begin{construction}[$\overline{\varphi}: G_{\overline{\varphi_{\mathrm{tor}}}} \rightarrow \overline{G}^\prime$] 
\label{construction2} We describe a subcase of Construction \ref{construction3} for later reference, enlarging also the commutative diagram \eqref{equation::graphdiagram2}. In this case, we start with a homomorphism $\varphi\!: G \rightarrow G^\prime$ of semiabelian varieties with split toric parts $T=\Gm^{t}$ and $T^\prime=\Gm^{t^\prime}$. We obtain a compactification $G_{\overline{\Gamma(\varphi_{\mathrm{tor}})}}$ from Construction \ref{construction3}. Furthermore, the homomorphism $\varphi$ induces now an even larger commutative diagram
\begin{equation} \label{equation::graphdiagram}
\begin{tikzcd} 
\arrow[d, hook] (\IP^1)^{t} & \arrow[d, hook] \arrow[l, hook'] \arrow[r, hook] T & \arrow[r, "\mathrm{pr}_2"] \arrow[ll, bend right, "\mathrm{pr}_1"'] \arrow[d, hook] \overline{\Gamma(\varphi_{\mathrm{tor}})} & \arrow[d, hook] (\IP^1)^{t^\prime}
\\
\arrow[d, "\overline{\pi}_0"] \overline{G} & \arrow[d,"\pi_0"] \arrow[l, hook'] \arrow[r, "\iota_\varphi", hook] G & \arrow[ll, bend right, "q_{\varphi}"' near start] \arrow[d, "\overline{\pi}_\varphi"] \arrow[r, "\overline{\varphi}"] G_{\overline{\Gamma(\varphi_{\mathrm{tor}})}} & \arrow[d, "\overline{\pi}_1"] \overline{G}^\prime
\\
A & \arrow[l, "\mathrm{id}_{A}"'] \arrow[r, "\mathrm{id}_{A}"] A & \arrow[r, "\varphi_{\mathrm{ab}}"] A & A^\prime
\end{tikzcd}
\end{equation}
such that there is a decomposition $\varphi= \overline{\varphi} \circ \iota_{\varphi}$; the map $\overline{\varphi}: G_{\overline{\Gamma(\varphi_{\mathrm{tor}})}} \rightarrow \overline{G}^\prime$ here arises naturally as follows: the toric part $\varphi_{\mathrm{tor}}: \Gm^t \rightarrow \Gm^{t^\prime}$ of $\varphi$ extends to a $\varphi_{\mathrm{tor}}$-equivariant map $\overline{\varphi}_{\mathrm{tor}}: \overline{\Gamma(\varphi_{\mathrm{tor}})} \rightarrow (\IP^1)^{t^\prime}$, which is just a restriction of $\mathrm{pr}_2: (\IP^1)^t \times (\IP^1)^{t^\prime} \rightarrow (\IP^1)^{t^\prime}$. As described above, this induces a corresponding extension $\overline{\varphi}: G_{\overline{\Gamma(\varphi_{\mathrm{tor}})}} \rightarrow \overline{G}^\prime$ of $\varphi: G \rightarrow G^\prime$.
 
In addition, each line bundle $M_{\overline{\Gamma(\varphi_{\mathrm{tor}})}}$ is a pullback $\overline{\varphi}^\ast M_{\overline{G}^\prime}$ for some homomorphism $\varphi: G \rightarrow G^\prime$ of semiabelian varieties. In fact, we can take $\eta_{G^\prime} = (\varphi_{\mathrm{tor}})_\ast \eta_G \in \Ext^1(A,\Gm^{t^\prime})$ and the homomorphism $\varphi: G \rightarrow G^\prime$ represented by $(\varphi_{\mathrm{tor}},\id_A)$.
\end{construction}

\section{Heights}
\label{section::heights}

We consistently work with (logarithmic) Weil heights and refer to \cite[Theorem B.3.6]{Hindry2000} for the main features of Weil's height machinery. In short, it provides for each line bundle $L$ on a projective $\IQbar$-variety $X$ a class of height functions $h_L\!: X(\IQbar) \rightarrow \IR$ such that any two height functions attached to $(X,L)$ differ by a globally bounded function on $X(\IQbar)$.

Let $G$ be a semiabelian variety over $\IQbar$ with toric part $T$ and abelian quotient $\pi\!: G \rightarrow A$. \textit{Assume also given a $T$-equivariant compactification $\overline{T}$ of the torus $T$ and a $T$-linearized line bundle $(M,\varrho)$ on $\overline{T}$ such that $[n]_T$ extends to $[n]_{\overline{T}}\!:\overline{T} \rightarrow \overline{T}$ and that there is an isomorphism $[n]^\ast_{\overline{T}}(M,\varrho) \approx (M,\varrho)^{\otimes n}$.} For our purposes, these conditions on $\overline{T}$ and $(M,\varrho)$ are always satisfied. Additionally, we choose a symmetric line bundle $N$ on $A$. We furnish $G_{\overline{T}}$ with the line bundle $L=G(M,\varrho) \otimes \overline{\pi}^\ast N$, which is ample if both $M$ and $N$ are ample (Lemma \ref{lemma::ampleness}). Weil's height machinery supplies us with some height function $h_L\!: G_{\overline{T}}(\IQbar) \rightarrow \IR$. The function $h_L$ is neither unique nor does it enjoy homogeneity properties like the Néron-Tate height of a symmetric line bundle on an abelian variety. However, the following lemma remedies this partially. We call a $T$-linearized line bundle $T$-effective if it has a $T$-invariant non-zero global section.

\begin{lemma} \label{lemma::canonicalheight} For any $(M,\varrho)$ (resp.\ $N$) as above, there exists a function $\widehat{h}_{G(M,\varrho)}: G_{\overline{T}}(\IQbar) \rightarrow \IR$ (resp.\ $\widehat{h}_{\overline{\pi}^\ast N}: G_{\overline{T}}(\IQbar) \rightarrow \IR$) such that
\begin{enumerate}
\item[(a)] $|h_{G(M,\varrho)}-\widehat{h}_{G(M,\varrho)}|$ (resp.\ $|h_{\overline{\pi}^\ast N}-\widehat{h}_{\overline{\pi}^\ast N}|$) is globally bounded on $G(\IQbar)$,
\item[(b)] $\widehat{h}_{G(M,\varrho)}([n]x)=|n|\widehat{h}_{G(M,\varrho)}(x)$ (resp.\ $\widehat{h}_{\overline{\pi}^\ast N}([n]x)=n^2\widehat{h}_{\overline{\pi}^\ast N}(x)$) for any $x\in G(\IQbar)$ and any integer $n$.
\item[(c)] Given a second $T$-linearized line bundle $(M^\prime,\varrho^\prime)$ (resp.\ a symmetric line bundle $N^\prime$ on $A$) as above, we have the additivity relations
\begin{equation*}
\widehat{h}_{G(M \otimes M^\prime, \varrho \otimes \varrho^\prime)} = \widehat{h}_{G(M,\varrho)}+\widehat{h}_{G(M^\prime, \varrho^\prime)} \text{ and } \widehat{h}_{\overline{\pi}^\ast(N\otimes N^\prime)} = \widehat{h}_{\overline{\pi}^\ast N}+\widehat{h}_{\overline{\pi}^\ast N^\prime}.
\end{equation*}
\item[(d)] If $(M,\varrho)$ (resp.\ $N$) is $T$-effective (resp.\ ample), then $\widehat{h}_{G(M,\varrho)}|_{{G}(\IQbar)}$ (resp.\ $\widehat{h}_{\overline{\pi}^\ast N}$) is non-negative.
\end{enumerate}
Furthermore, $\widehat{h}_{G(M,\varrho)}$ (resp.\ $\widehat{h}_{\overline{\pi}^\ast N}$) is uniquely characterized by (a) and (b).
\end{lemma}

It is natural to work with the unique $\widehat{h}_{L}=\widehat{h}_{G(M,\varrho)}+\widehat{h}_{\overline{\pi}^\ast N}$ instead of a non-canonical Weil height $h_{L}$. By (a) of the above theorem, their difference is globally bounded on $G(\IQbar)$. As for abelian varieties, the zero set of $\widehat{h}_{L}$ coincides with the torsion points of $G$ if both $M$ and $N$ are ample and $(M,\varrho)$ is $T$-effective. 

The assumption of $T$-effectivity in \textit{(d)} cannot be relaxed to mere effectivity. In fact, assume that $T=\Gm^t$ and let $Q_i$, $1 \leq i \leq t$, be the line bundles on $A$ such that $\eta_G = (Q_1,\dots,Q_t) \in A^\vee(k)^t = \Ext^1_k(A,T)$. For $\varrho$ running through all possible $T$-linearizations of the trivial line bundle $\IA^1_{\overline{T}}$, the line bundle $G(\IA^1_{\overline{T}},\varrho)$ runs through $\overline{\pi}^\ast(Q_1^{k_i} \otimes \cdots \otimes Q_t^{k_t})$ for arbitrary integers $k_i$, as a comparison of \v{C}ech cocycles shows. Except for this caveat, we do not need this and leave the verification to the interested reader.

\begin{proof} \textit{(a), (b)}: The first two assertions of the lemma as well as uniqueness can be inferred directly from \cite[Theorem B.4.1]{Hindry2000} applied to $G(M,\varrho)$ (resp.\ $N$). Indeed, $[n]_{\overline{G}}^\ast G(M,\varrho) \approx G(M,\varrho)^{\otimes n}$ by our assumption (compare Section \ref{section::compactification}) and $[n]_{\overline{G}}^\ast \overline{\pi}^\ast N \approx \overline{\pi}^\ast [n]_A^\ast N \approx \overline{\pi}^\ast N^{\otimes n^2}$ since $N$ is symmetric. (The result in \cite{Hindry2000} is stated for divisor classes on smooth varieties but it is also true for line bundles on general varieties with exactly the same proof. The reader may compare also \cite[Lemma 9.2.4]{Bombieri2006}.)

\textit{(c)}: As $G(M \otimes M^\prime,\varrho \otimes \varrho^\prime) \approx G(M,\varrho) \otimes G(M^\prime,\varrho^\prime)$, the global boundedness of $|h_{G(M\otimes M^\prime,\varrho \otimes \varrho^\prime)} - h_{G(M,\varrho)} - h_{G(M^\prime,\varrho^\prime)}|$ is a standard property of the Weil height. As (a) and (b) already characterize $\hhat_{G(M \otimes M^\prime,\varrho \otimes \varrho^\prime)}$ uniquely, we infer the first equality in (c). The second one follows similarly.

\textit{(d)}: Similarly, one observes that $\widehat{h}_{G(M,\varrho)}$ (resp.\ $\widehat{h}_{\overline{\pi}^\ast N}$) is non-negative if $h_{G(M,\varrho)}$ (resp.\ $h_{\overline{\pi}^\ast N}$) is bounded from below on $G(\IQbar)$ (resp.\ $G_{\overline{T}}(\IQbar)$). For the height $h_{\overline{\pi}^\ast N}$, this is true because the ampleness of $N$ implies that $N$ and hence $\overline{\pi}^\ast N$ has empty base locus. By assumption, we have a $T$-invariant non-zero global section $s\!: \overline{T} \rightarrow M$. This gives rise to local sections $s_i^\prime = U_i \times_k s\!: U_i \times_A G_{\overline{T}} = U_i \times_k \overline{T} \rightarrow U_i \times_k M = U_i \times_A G(M,\varrho)$. Due to the $T$-invariance of $s$, the sections $s_i^\prime$ glue together to a non-zero global section $s^\prime$ of $G(M,\varrho)$. Furthermore, $T$-invariance guarantees that $s_x$ generates $M_x$ for every $x \in T(\IQbar)$. We infer that $s^\prime_x$ generates $G(M,\varrho)$ for every $x \in G(\IQbar)$. Therefore the base locus of $G(M,\varrho)$ is contained in $G_{\overline{T}} \setminus G$ and $h_{G(M,\varrho)}|_{G(\IQbar)}$ is bounded from below.
\end{proof}

In addition, we have a good functorial behavior of the heights $\widehat{h}_{G(M,\varrho)}$ and $\widehat{h}_{\overline{\pi}^\ast N}$. To state precisely what this means, let $G$ (resp.\ $G^\prime$) be a semiabelian variety over $\IQbar$ with toric part $T$ (resp.\ $T^\prime$) and abelian quotient $A$ (resp.\ $A^\prime$). Take furthermore equivariant compactifications $\overline{T}$ and $\overline{T}^\prime$ so that $\varphi_{\mathrm{tor}}: T \rightarrow T^\prime$ extends to a $\varphi_{\mathrm{tor}}$-equivariant map $\overline{\varphi}_{\mathrm{tor}}: \overline{T} \rightarrow \overline{T}^\prime$. In this situation, we consider a $T^\prime$-linearized line bundle $(M,\varrho)$ on $\overline{T}^\prime$ such that there is a $T^\prime$-equivariant isomorphism $[n]_{\overline{T}^\prime}^\ast (M,\varrho) \approx (M,\varrho)^{\otimes n}$. We also take a symmetric ample line bundle $N$ on $A^\prime$. 

\begin{lemma} 
\label{lemma::functoriality}
In the situation described in the above paragraph, let $\varphi: G \rightarrow G^\prime$ be a homomorphism with toric (resp.\ abelian) component $\varphi_{\mathrm{tor}}$ (resp.\ $\varphi_{\mathrm{ab}}$). For every $x \in G(\IQbar)$, we have then
\begin{equation*}
\widehat{h}_{G^\prime(M,\varrho)}(\varphi(x))=\widehat{h}_{G(\varphi_{\mathrm{tor}}^\ast (M,\varrho))}(x) \text{ and } \widehat{h}_{(\overline{\pi}^\prime)^\ast N}(\varphi(x))=\widehat{h}_{\overline{\pi}^\ast (\varphi_{\mathrm{ab}}^\ast N)}(x).
\end{equation*}
\end{lemma}
\begin{proof}
This follows directly from the functorial behavior of the Weil height under pullback and the uniqueness assertion of Lemma \ref{lemma::canonicalheight}. 
\end{proof}

We note a further addendum to Lemma \ref{lemma::canonicalheight}, which is specifically related to the line bundles $M_{\overline{G}}$ and $M_{\overline{\Gamma({\varphi_{\mathrm{tor}}})}}$. 

\begin{lemma} \label{lemma::nonnegativeheight} Let $G$ a semiabelian variety with split toric part $\Gm^{t}$ and abelian quotient $\pi: G \rightarrow A$. For any $\varphi_{\mathrm{tor}} \in \Hom(\Gm^{t},\Gm^{t^\prime})$, the restriction of $\hhat_{M_{\overline{\Gamma(\varphi_{\mathrm{tor}})}}}: G_{\overline{\Gamma(\varphi_{\mathrm{tor}})}}(\IQbar) \rightarrow \IR$ to $G(\IQbar)$ is non-negative. In particular, the restriction of $\hhat_{M_{\overline{G}}}$ to $G(\IQbar)$ is non-negative.
\end{lemma}

\begin{proof} In Construction \ref{construction2}, it is shown that each $M_{\overline{\Gamma(\varphi_{\mathrm{tor}})}}$ is isomorphic to the pullback of a line bundle of the form $M_{\overline{G}^\prime}$. Using Lemma \ref{lemma::functoriality}, it hence suffices to prove the non-negativity of $\hhat_{M_{\overline{G}}}$. This is already in the literature (cf.\ \cite[Lemme 3.9]{Chambert-Loir2000}), but we give the argument here for completeness because it is a direct consequence of Construction \ref{construction0}. The Cartier divisor $D_{t}$ on $\Gm^{t}$ is effective and $\Gm^{t}$-invariant so that the constant function $1 \in \mathcal{K}_{(\IP^1)^{t}}((\IP^1)^{t})$ gives rise to a $\Gm^{t}$-invariant global section of its associated line bundle $(M_{t},\varrho_{t})$. In other words, $(M_{t},\varrho_{t})$ is $\Gm^{t}$-effective and we can use Lemma \ref{lemma::canonicalheight} (d).
\end{proof}

Fix again a semiabelian variety $G$ over $\IQbar$ with toric part $\Gm^{t}$ and abelian quotient $\pi: G \rightarrow A$. Furthermore, let $\overline{G}$ be a compactification of $G$ and $\overline{\pi}\!: \overline{G} \rightarrow A$ its abelian quotient as in Construction \ref{construction1}. We want to estimate the difference between $\hhat_{M_{\overline{\Gamma(\varphi_{\mathrm{tor}})}}}$ and $\hhat_{M_{\overline{\Gamma(\varphi^\prime_{\mathrm{tor}})}}}$ for two ``close'' homomorphisms $\varphi_{\mathrm{tor}}, \varphi_{\mathrm{tor}}^\prime\in \Hom(\Gm^{t},\Gm^{t^\prime})$. Simultaneously, we examine the corresponding ``abelian'' analogue. For this purpose, let $A^\prime$ be a second abelian variety and $N$ (resp.\ $N^\prime$) an ample symmetric line bundle on $A$ (resp.\ $A^\prime$). We choose some linear norms $\vert\cdot\vert$ on $\Hom(\Gm^{t}, \Gm^{t^\prime})$ and $\Hom(A, A^\prime)$ for quantification.\footnote{A natural choice of norm for $\Hom(A, A^\prime)$, using the Rosati involution on $A \times A^\prime$, is introduced in \cite[Section 4]{Habegger2009a}. As there, however, we have no need to choose any particular norm.}

\begin{lemma} \label{lemma::heightsclosehomomorphisms} In the above situation, there exist constants $c_1$ and $c_2$ depending only on $G$, $N$, $t^\prime$, $A^\prime$, $N^\prime$ and the linear norms $\vert\cdot\vert$ on $\Hom(\Gm^{t},\Gm^{t^\prime})$ and $\Hom(A,A^\prime)$ such that the following assertions are true: For any pair $(\varphi_{\mathrm{tor}}, \varphi_{\mathrm{tor}}^\prime) \in \Hom(\Gm^{t},\Gm^{t^\prime})^2$ and any $x \in G(\IQbar)$, we have
\begin{equation} \label{equation::heighttorus}
\vert
\hhat_{M_{\overline{\Gamma(\varphi_{\mathrm{tor}})}}} (x)- \hhat_{M_{\overline{\Gamma(\varphi^\prime_{\mathrm{tor}})}}}(x) \vert \leq c_1 \vert\varphi_{\mathrm{tor}}-\varphi_{\mathrm{tor}}^\prime\vert \cdot \hhat_{M_{\overline{G}}}(x).
\end{equation}
Similarly, we have
\begin{equation} \label{equation::heightabelian}
\vert \widehat{h}_{\overline{\pi}^\ast \varphi_{\mathrm{ab}}^\ast N^\prime}(x)- \widehat{h}_{\overline{\pi}^\ast (\varphi_{\mathrm{ab}}^\prime)^\ast N^\prime}(x)\vert \leq c_2 | \varphi_{\mathrm{ab}}- \varphi_{\mathrm{ab}}^\prime |^2 \cdot \widehat{h}_{\overline{\pi}^\ast N}(x)
\end{equation}
 for any pair $(\varphi_{\mathrm{ab}},\varphi_{\mathrm{ab}}^\prime)\in \Hom(A,A^\prime)^2$.
\end{lemma}

\begin{proof} We prove first the inequality (\ref{equation::heighttorus}). The proof takes place on the ``graph compactification'' $G_{\overline{\Gamma}}$ of $G$ where $\overline{\Gamma} = \overline{\Gamma(\varphi_{\mathrm{tor}} \times \varphi_{\mathrm{tor}}^\prime)} \subset (\IP^1)^{t} \times (\IP^1)^{t^\prime} \times (\IP^1)^{t^\prime}$. We denote the projections corresponding to these three factors by $\mathrm{pr}_i$ ($i = 1,2,3$). The projections $(\pr_1 \times \pr_2)|_{\overline{\Gamma}}: \overline{\Gamma} \rightarrow \overline{\Gamma(\varphi_{\mathrm{tor}})}$ and $(\pr_1 \times \pr_3)|_{\overline{\Gamma}}: \overline{\Gamma} \rightarrow \overline{\Gamma(\varphi_{\mathrm{tor}}^\prime)}$ are $\Gm^{t}$-equivariant and hence induce maps $G_{\overline{\Gamma}} \rightarrow G_{\overline{\Gamma(\varphi_{\mathrm{tor}})}}$ and $G_{\overline{\Gamma}} \rightarrow G_{\overline{\Gamma(\varphi_{\mathrm{tor}}^\prime)}}$, which both restrict to the identity on $G$. By Lemma \ref{lemma::functoriality}, we obtain
\begin{equation} \label{equation::pullback}
\hhat_{M_{\overline{\Gamma(\varphi_{\mathrm{tor}})}}}(x)= \hhat_{G(\pr_2^\ast(M_{t^\prime},\varrho_{t^\prime})|_{\overline{\Gamma}})}(x) \text{ and } \hhat_{M_{\overline{\Gamma(\varphi_{\mathrm{tor}}^\prime)}}}(x)= \hhat_{G(\pr_3^\ast(M_{t^\prime},\varrho_{t^\prime})|_{\overline{\Gamma}})}(x)
\end{equation}
for any $x \in G(\IQbar)$. Similarly, we have 
\begin{equation} \label{equation::pullback2}
\hhat_{M_{\overline{G}}}(x)= \hhat_{G(\pr_1^\ast(M_{t},\varrho_{t})|_{\overline{\Gamma}})}(x)
\end{equation}
for every $x \in G(\IQbar)$.

Using standard coordinates $X_u$, $1 \leq u \leq t$, (resp.\ $Y_v$, $1 \leq v\leq t^\prime$,) on $\Gm^{t}$ (resp.\ $\Gm^{t^\prime}$), we write
\begin{equation*}
\varphi_{\mathrm{tor}}^\ast(Y_v) = X_1^{a_{1v}}X_2^{a_{2v}}\cdots X_{t}^{a_{tv}}\text{ (resp.\ } (\varphi_{\mathrm{tor}}^\prime)^\ast(Y_v) = X_1^{a_{1v}^\prime}X_2^{a_{2v}^\prime}\cdots X_{t}^{a_{tv}^\prime} \text{), } 1 \leq v\leq t^\prime,
\end{equation*} 
with integers $a_{uv}$ (resp.\ $a_{uv}^\prime$). Our strategy is to compare the restriction of the line bundles $\mathrm{pr}_1^\ast (M_{t},\varrho_{t})^{\otimes l}$, $l$ sufficiently large, and $\mathrm{pr}_2^\ast (M_{t^\prime},\varrho_{t^\prime}) \otimes \mathrm{pr}_3^\ast (M_{t^\prime},\varrho_{t^\prime})^{\otimes -1}$ on $\overline{\Gamma}$. In fact, we claim that both
\begin{equation} \label{equation::effectivity}
(\mathrm{pr}_1^\ast M_{t}^{\otimes l} \otimes \mathrm{pr}_2^\ast M_{t^\prime} \otimes \mathrm{pr}_3^\ast M_{t^\prime}^{\otimes -1})|_{\overline{\Gamma}}\text{ and } (\mathrm{pr}_1^\ast M_{t}^{\otimes l} \otimes \mathrm{pr}_2^\ast M_{t^\prime}^{\otimes -1} \otimes \mathrm{pr}_3^\ast M_{t^\prime})|_{\overline{\Gamma}}
\end{equation}
are $\Gm^{t}$-effective with respect to the induced linearizations. In this case, Lemma \ref{lemma::canonicalheight} (c,\ d) implies that
\begin{equation} \label{equation::lestimate2}
\vert \hhat_{G(\pr_2^\ast(M_{t^\prime},\varrho_{t^\prime})|_{\overline{\Gamma}})}(x) - \hhat_{G(\pr_3^\ast(M_{t^\prime},\varrho_{t^\prime})|_{\overline{\Gamma}})}(x) \vert \leq l \cdot \hhat_{G(\pr_1^\ast(M_{t},\varrho_{t})|_{\overline{\Gamma}})}(x)
\end{equation}
for each $x \in G(\IQbar) \subset G_{\overline{\Gamma}}(\IQbar)$.
Using the equalities (\ref{equation::pullback}) and (\ref{equation::pullback2}), the inequality (\ref{equation::heighttorus}) can be derived from (\ref{equation::lestimate2}) if we have adequate control on $l$. For this, we note that $\pr_1^\ast(L_{t},\varrho_{t})|_{\overline{\Gamma}}$ (resp.\ $\pr_{2}^\ast(L_{t^\prime},\varrho_{t^\prime})|_{\overline{\Gamma}}$, $\pr_3^\ast(L_{t^\prime},\varrho_{t^\prime})|_{\overline{\Gamma}}$)) can be defined by means of the $\Gm^{t}$-invariant Cartier divisor $\pr_1^\ast D_{t}|_{\overline{\Gamma}}$ (resp.\ $\pr_2^\ast D_{t^\prime}|_{\overline{\Gamma}}$, $\pr_3^\ast D_{t^\prime}|_{\overline{\Gamma}}$). We next describe these divisors explicitly and start with giving a covering of $\overline{\Gamma}$. With each $(t+2t^\prime)$-tuple $\underline{m}$ of numbers $m_r \in \{ -1, 1\}$, $1 \leq r \leq t+2t^\prime$, we associate a Zariski open
\begin{equation*}
U_{\underline{m}}= \overline{\Gamma} \cap \bigcap_{1 \leq u \leq t} D(\pr_1^\ast X_u^{m_{u}}) \cap \bigcap_{1 \leq v \leq t^\prime} D(\pr_2^\ast Y_{v}^{m_{v+t}})
\cap \bigcap_{1 \leq v \leq t^\prime} D(\pr_3^\ast Y_{v}^{m_{v+t+t^\prime}}).
\end{equation*}
Evidently, $\pr_1^\ast D_{t}|_{\overline{\Gamma}}$ is represented by $(U_{\underline{m}},f_{\underline{m}})$ with
\begin{equation*}
f_{\underline{m}}=\pr_1^\ast (X_1^{-m_{1}}X_2^{-m_{2}}\cdots X_{t}^{-m_{t}})
\end{equation*}
and $\pr_2^\ast D_{t^\prime}|_{\overline{\Gamma}}$ (resp.\ $\pr_3^\ast D_{t^\prime}|_{\overline{\Gamma}}$) is represented by $(U_{\underline{m}},g_{\underline{m}})$ (resp.\ $(U_{\underline{m}},g_{\underline{m}}^{\prime})$) with
\begin{equation*}
g_{\underline{m}} = \prod_{1 \leq v \leq t^\prime}\pr_2^\ast(Y_v^{-m_{v+t}}) \, \text{(resp.\ } g_{\underline{m}}^\prime = \prod_{1 \leq v \leq t^\prime}\pr_3^\ast(Y_v^{-m_{v+t+t^\prime}})\text{).}
\end{equation*}
The meromorphic function $1 \in \mathcal{K}_{\overline{\Gamma}}(\overline{\Gamma})$ gives a $\Gm^{t}$-invariant rational section of $\mathcal{O}(\pr_1^\ast D_{t}|_{\overline{\Gamma}})$, $\mathcal{O}(\pr_2^\ast D_{t^\prime}|_{\overline{\Gamma}})$ and $\mathcal{O}(\pr_3^\ast D_{t^\prime}|_{\overline{\Gamma}})$ by our choice of linearizations. It thus also gives a $\Gm^{t}$-invariant rational section of $\mathcal{O}(l\cdot \pr_1^\ast D_{t}|_{\overline{\Gamma}}+\pr_2^\ast D_{t^\prime}|_{\overline{\Gamma}}-\pr_3^\ast D_{t^\prime}|_{\overline{\Gamma}})$ and $\mathcal{O}(l\cdot \pr_1^\ast D_{t}|_{\overline{\Gamma}}-\pr_2^\ast D_{t^\prime}|_{\overline{\Gamma}}+\pr_3^\ast D_{t^\prime}|_{\overline{\Gamma}})$, to which the line bundles in (\ref{equation::effectivity}) are associated. For $\Gm^{t}$-effectivity, we may hence prove that it is actually a global section. In other words, we have to prove that both $f_{\underline{m}}^l\cdot g_{\underline{m}}\cdot (g_{\underline{m}}^{\prime})^{-1}$ and $f_{\underline{m}}^l\cdot (g_{\underline{m}})^{-1}\cdot g_{\underline{m}}^{\prime}$ are regular on $U_{\underline{m}}$. Let us remark first that for $l_v=\max_{1\leq u\leq t} \{ |a_{uv}-a_{uv}^{\prime}| \}$ the meromorphic function
\begin{equation*}
f_{\underline{m}}^{l_v} \cdot \pr_2^\ast(Y_v) \pr_3^\ast(Y_v)^{-1} = \prod_{1 \leq u \leq t} \pr_1^\ast (X_u^{s_u}), \, s_u = -m_ul_v + (a_{uv} - a_{uv}^{\prime}),
\end{equation*}
is regular on $U_{\underline{m}}$; for $\pr_1^\ast (X_u^{-m_u})$ is regular on $U_{\underline{m}} \subset D(\pr_1^\ast (X_u^{m_u}))$. Similarly, the meromorphic function $f_{\underline{m}}^{l_v} \cdot \pr_2^\ast(Y_v)^{-1} \pr_3^\ast(Y_v)$ is regular on $U_{\underline{m}}$. We write $g_{\underline{m}} \cdot (g_{\underline{m}}^\prime)^{-1} = \prod_{v=1}^{t^\prime} h_{\underline{m}}$ with
\begin{equation*}
h_{\underline{m}}= \pr_2^\ast(Y_v^{-m_{v+t}})\pr_3^\ast (Y_v^{m_{v+t+t^\prime}})
\end{equation*}
and claim that $f_{\underline{m}}^{l_v} \cdot h_{\underline{m}}$ is regular on $U_{\underline{m}}$. If $m_{v+t}=m_{v+t+t^\prime}=1$ or $m_{v+t}=m_{v+t+t^\prime}=-1$, this follows directly from our previous remark. In case $m_{v+t}=-1$ and $m_{v+t+t^\prime}= 1$, the function
\begin{equation*}
f_{\underline{m}}^{l_v} \cdot h_{\underline{m}} = \pr_2^\ast(Y_v)^2 \cdot \left(f_{\underline{m}}^{l_v} \cdot \pr_2^\ast(Y_v)^{-1} \pr_3^\ast(Y_v) \right)
\end{equation*} 
is regular by our remark and the fact that $\pr_2^\ast(Y_v)$ is regular on $U_{\underline{m}} \subset D(\pr_2^\ast(Y_v)^{-1})$. The case $m_{v+t}=1$ and $m_{v+t+t^\prime}=-1$ can be handled in the same way, establishing our claim. In conclusion, the condition
\begin{equation} \label{equation::lestimate}
l \geq \sum_{1 \leq v \leq t^\prime} l_v = \sum_{1 \leq v\leq t^\prime} \left( \max_{1 \leq u\leq t} \{ |a_{uv}-a_{uv}^\prime|\} \right)
\end{equation}
suffices to ensure the regularity of $f_{\underline{m}}^{l}\cdot g_{\underline{m}}\cdot (g_{\underline{m}}^{\prime})^{-1}$. The same argument shows that each $f_{\underline{m}}^{l}\cdot (g_{\underline{m}})^{-1}\cdot g_{\underline{m}}^{\prime}$ is regular on $U_{\underline{m}}$. Combining (\ref{equation::lestimate2}) and (\ref{equation::lestimate}), we obtain (\ref{equation::heighttorus}). 

The inequality (\ref{equation::heightabelian}) boils down to
\begin{equation*}
\vert \widehat{h}_{N^\prime}(\varphi_{\mathrm{ab}}(y)) - \widehat{h}_{N^\prime}(\varphi_{\mathrm{ab}}^\prime(y))|\leq c_2 | \varphi_{\mathrm{ab}}-\varphi_{\mathrm{ab}}^\prime|^2 \cdot \widehat{h}_{N}(y), y=\pi(x),
\end{equation*}
where $\widehat{h}_{N}$ and $\widehat{h}_{N^\prime}$ are now just the N\'eron-Tate heights on the abelian varieties $A$ and $A^\prime$. This follows straightforwardly from the fact that the map 
\begin{equation*}
\Hom(A,A^\prime) \longrightarrow \Pic(A), \varphi_{\mathrm{ab}} \longmapsto \varphi^\ast_{\mathrm{ab}} N^\prime,
\end{equation*}
is quadratic, which is a direct consequence of the Theorem of the Cube (\cite[Corollary II.6.2]{Mumford1970}). The reader may refer to \cite[p.\ 417]{Habegger2009a} for details.
\end{proof}

Finally, we state a lemma on the behavior of the heights $\hhat_{\overline{\Gamma(\varphi_{\mathrm{tor}})}}$ with respect to the group law. Again, there is an ``abelian'' analogue and we mention this also for later reference.

\begin{lemma} \label{lemma::heightsgrouplaw} For any $\varphi_{\mathrm{tor}} \in \Hom(\Gm^{t},\Gm^{t^\prime})$ and any points $x,y \in G(\IQbar)$, we have $\widehat{h}_{M_{\overline{\Gamma(\varphi_{\mathrm{tor}})}}}(xy) \leq \widehat{h}_{M_{\overline{\Gamma(\varphi_{\mathrm{tor}})}}}(x) + \widehat{h}_{M_{\overline{\Gamma(\varphi_{\mathrm{tor}})}}}(y)$. Similarly, we have $\widehat{h}_{\overline{\pi}^\ast N}(xy) \leq 2\widehat{h}_{\overline{\pi}^\ast N}(x) + 2\widehat{h}_{\overline{\pi}^\ast N}(y)$.
\end{lemma}

Note that this statement includes the fact that 
\begin{equation} \label{equation::hMG}
\hhat_{M_{\overline{G}}}(xy)\leq \hhat_{M_{\overline{G}}}(x) + \hhat_{M_{\overline{G}}}(y)
\end{equation}
for all $x,y \in G(\IQbar)$ (set $\varphi_{\mathrm{tor}} = \id_{\Gm^t}$). Most of our proof is actually about establishing this inequality, which has been already provided in the literature (see e.g.\ \cite[Corollaire 3.1]{Remond2003}). Nevertheless, we give a proof here both for completeness and because it is very close to the proof of Lemma \ref{lemma::heightsclosehomomorphisms} above.

\begin{proof} For the first assertion, it suffices to prove (\ref{equation::hMG}). In fact, each $M_{\overline{\Gamma(\varphi_{\mathrm{tor}})}}$ is isomorphic to some pullback $\overline{\varphi}^\ast M_{\overline{G}^\prime}$ along a homomorphism $\varphi: G \rightarrow G^\prime$ to another semiabelian variety $G^\prime$ (see Construction \ref{construction2}). In order to prove (\ref{equation::hMG}), we use the same strategy as for Lemma \ref{lemma::heightsclosehomomorphisms}. This means we consider the Zariski closure $\overline{\Gamma} \subset ((\IP^1)^t \times (\IP^1)^t) \times (\IP^1)^t$ of the graph of the group law $\cdot_T : \Gm^t \times \Gm^t \rightarrow \Gm^t$. Again, we denote the projection to the $i$-th component by $\pr_i$ ($i=1,2,3$). For this, we use standard coordinates $X_u$, $1 \leq u \leq t$, on $\Gm^t$ (and on its extension to $(\IP^1)^t$). Note that $(\pr_3^\ast X_u)=(\pr_1^\ast X_u)(\pr_2^\ast X_u)$ on $\overline{\Gamma}$.

With each $(3t)$-tuple $\underline{m} \in \{ -1, 1\}^{3t}$ of numbers $m_r \in \{ -1, 1\}$, $1 \leq r \leq 3t$, we associate a Zariski open. To wit, we define
\begin{equation*}
U_{\underline{m}}= \overline{\Gamma} \cap \bigcap_{1 \leq u \leq t} D(\pr_1^\ast X_u^{m_{u}}) \cap \bigcap_{1 \leq u \leq t} D(\pr_2^\ast X_u^{m_{u+t}}) \cap \bigcap_{1 \leq u \leq t} D(\pr_3^\ast X_u^{m_{u+2t}}).
\end{equation*}
It is easy to see that each $\pr_i^\ast D_t|_{\overline{\Gamma}}$ ($i=1,2,3$) is represented by $(U_{\underline{m}}, f^{(i)}_{\underline{m}})$, where
\begin{equation*}
f_{\underline{m}}^{(i)}=\pr_i^\ast (X_1^{-m_{1+(i-1)t}}X_2^{-m_{2+(i-1)t}}\cdots X_{t}^{-m_{t+(i-1)t}}).
\end{equation*}
Consequently, the restriction of $\pr_1^\ast D_t + \pr_2^\ast D_t - \pr_3^\ast D_t$ to $\overline{\Gamma}$ is represented by $(U_{\underline{m}}, f_{\underline{m}}^{(1)} \cdot f_{\underline{m}}^{(2)}\cdot (f_{\underline{m}}^{(3)})^{-1})$. The meromorphic function $f_{\underline{m}}^{(1)} \cdot f_{\underline{m}}^{(2)}\cdot (f_{\underline{m}}^{(3)})^{-1}$ equals
\begin{equation*}
\pr_1^\ast(X_1)^{-m_1+m_{2t+1}}\cdots
\pr_1^\ast(X_t)^{-m_t+m_{3t}}
\pr_2^\ast(X_{1})^{-m_{t+1}+m_{2t+1}}\cdots
\pr_2^\ast(X_{t})^{-m_{2t}+m_{3t}}.
\end{equation*}
By definition, each $\pr_i^\ast(X_u)^{-m_{u+(i-1)t}}$, $i \in \{ 1,2,3\}$, $1\leq u \leq t$, is regular on $U_{\underline{m}}$. Since $-m_u+m_{2t+u} \in \{ 0, -2m_u \}$ and $-m_{t+u}+m_{2t+u} \in \{ 0, -2m_{t+u} \}$, we infer the regularity of $f_{\underline{m}}^{(1)}\cdot f_{\underline{m}}^{(2)} \cdot (f_{\underline{m}}^{(3)})^{-1}$ on $U_{\underline{m}}$. As in the proof of Lemma \ref{lemma::heightsclosehomomorphisms}, we see that this implies that $1 \in \mathcal{K}_{\overline{\Gamma}}(\overline{\Gamma})$ is a $T$-invariant global section of $\mathcal{O}(\pr_1^\ast D_t + \pr_2^\ast D_t - \pr_3^\ast D_t)$. Thus, the first assertion follows from Lemma \ref{lemma::canonicalheight} (d).
For the second assertion, it suffices to note the equivalence of the assertion with
\begin{equation*}
\widehat{h}_{\overline{\pi}^\ast N}(\pi(x) + \pi(y))\leq 2\widehat{h}_{\overline{\pi}^\ast N}(\pi(x))+ 2\widehat{h}_{\overline{\pi}^\ast N}(\pi(y)).
\end{equation*}
Indeed, this inequality follows directly from the parallelogram law for the Néron-Tate height \cite[Theorem B.5.1 (c)]{Hindry2000} and its non-negativity for symmetric line bundles.
\end{proof}

\section{Hermitian Differential Geometry}
\label{section::hermitiandifferentialgeometry}

In the next two sections, we make extensive use of hermitian differential geometry at the level of rather explicit computations on semiabelian varieties. To avoid permanent interruptions in these, we recall here the necessary abstract framework separately. The reader is referred to \cite[Section 3.1]{Voisin2007} as well as \cite[Section 0.2]{Griffiths1994} and \cite[Section 1.2]{Huybrechts2005} for details. 

Let $Y$ be a complex manifold (e.g., $X^{\mathrm{sm}}(\IC)$ for a complex algebraic variety $X$). To $Y$ is associated its real tangent bundle $T_{\IR} Y$, its holomorphic tangent bundle $T^{1,0}_{\IC} Y$ (e.g., $T_x X(\IC)$ for a smooth complex algebraic variety $X$) and its anti-holomorphic tangent bundle $T^{0,1}_{\IC} Y$. As real vector bundles, all three can be canonically identified (cf.\ \cite[p.\ 17]{Griffiths1994}) and we do so from now on. In this way, we obtain an almost complex structure $I:T_{\IR} Y \rightarrow T_{\IR} Y$ (i.e., a linear map $I:T_{\IR}Y \rightarrow T_{\IR}Y$ such that $I^2=-\id_{T_{\IR}Y}$) from the multiplication-by-$i$ (resp.\ multiplication-by-$(-i)$) homomorphism on the complex vector bundle $T_\IC^{1,0} Y$ (resp.\ $T_\IC^{0,1} Y$). A $(1,1)$-form of real type on $Y$ is an alternating $\IR$-bilinear pairing
\begin{equation*}
\omega:  T_\IR Y \times_Y  T_\IR Y \longrightarrow \IR \times Y
\end{equation*}
such that $\omega(I (\cdot), I (\cdot)) = \omega(\cdot,\cdot)$. Under the identification $T_\IC^{1,0}Y = T_\IR Y$, this corresponds to an alternating $\IR$-bilinear pairing
\begin{equation*}
\omega: T_\IC^{1,0} Y \times_Y  T_\IC^{0,1} Y \longrightarrow \IR \times Y
\end{equation*}
such that $\omega(i (\cdot), i (\cdot))=-\omega(\cdot,\cdot)$.\footnote{One frequently identifies a $(1,1)$-form $\omega$ of real type with its scalar extension $\omega_\IC: T_\IC Y \times_Y T_\IC Y \rightarrow \IC \times Y$, $T_\IC Y = T_\IR Y \otimes_\IR \IC$. Since the restriction to $T_\IR Y \times_Y T_\IR Y$ or $T_\IC^{1,0} Y \times_Y T_\IC^{0,1} Y$ retains all information, we allow ourselves to switch tacitly between $\omega$ and $\omega_\IC$.} The Chern forms of hermitian line bundles are the basic examples of such $(1,1)$-forms. More generally, for any smooth function $\lambda: Y \rightarrow \IR$ the $(1,1)$-form $dd^c \lambda$ ($d^c=i/2\pi (\delbar - \del)$) is always of real type. To such a $(1,1)$-form $\omega$ is associated a symmetric $\IR$-bilinear pairing
\begin{equation*}
g_\omega: T_\IR Y \times_Y T_\IR Y \longrightarrow \IR \times Y, \, (v,w) \longmapsto \omega(v,Iw).
\end{equation*}
In fact, this establishes a one-to-one correspondence between $(1,1)$-forms of real type and symmetric $\IR$-bilinear forms on $T_\IR Y$.  Using our identification of $T_\IR Y$ with $T_\IC^{1,0} Y$ and $T_\IC^{0,1} Y$, the $(1,1)$-form $\omega$ is positive (resp.\ semipositive) in the ordinary sense (e.g.\ \cite[Definition 4.3.14]{Huybrechts2005}) if and only if $g_\omega$ is positive definitive (resp.\ positive semidefinite). We note that for a smooth function $\lambda: Y \rightarrow \IR$, the $(1,1)$-form $dd^c \lambda$ is semipositive if and only if $\lambda$ is plurisubharmonic (cf.\ \cite[Theorem K.8]{Gunning1990}).

For later reference, we remark that for any smooth function $f: Y \rightarrow \IR$ the $(1,1)$-form $\omega = i (\del f \wedge \delbar f)$ is of real type and
\begin{equation} \label{equation::explicitg}
g_\omega = \frac{1}{2}\left( \del f \otimes \delbar f + \delbar f \otimes \del f \right);
\end{equation}
for this is a local assertion that reduces by linearity to the fact that the $(1,1)$-form
\begin{equation*}
\omega = i \cdot (\alpha dz_i \wedge d\overline{z}_j + \overline{\alpha} dz_j \wedge d\overline{z}_i), \, \alpha \in \IC,
\end{equation*}
on $\IC^n$ is of real type and the fact that
\begin{equation*}
i dz_i \wedge d\overline{z}_j(v,Iw) = \frac{1}{2}\left( dz_i(v) d\overline{z}_j(w) + d\overline{z}_j(v) dz_i(w)\right), \, v, w \in T_{\IR,x} \IC^n, x \in \IC^n.
\end{equation*}

To a $(1,1)$-form $\omega$ of real type is also associated a hermitian form (with respect to $I$)
\begin{equation} \label{equation::hermitianform}
H_\omega: T_\IR Y \times_Y T_\IR Y \longrightarrow \IC \times Y, (v,w) \longmapsto g_\omega(v,w) - i \cdot \omega(v,w),
\end{equation}
and this can be also seen to be a one-to-one correspondence. Indeed, $\omega=-\mathrm{Im}(H_\omega)$.

Let $Z$ be a complex submanifold of a complex manifold $Y$ and $\omega$ a $(1,1)$-form of real type on $Y$. Restricting and taking exterior products, we obtain an alternating $\IR$-multilinear map
\begin{equation*}
(\omega|_{Z})^{ \wedge \dim(Z)}: (T_{\IR,x} Z)^{2\dim(Z)} \longrightarrow \IR
\end{equation*}
for each $x \in Z$. If the restriction of the $\IR$-bilinear form $g_{\omega,x}$ to $T_{\IR,x} Z$ is moreover positive definite, we have a non-zero Riemannian volume form (\cite[pp.\ 361-362]{Helgason2001})
\begin{equation*}
\mathrm{vol}(g_{\omega,x}): (T_{\IR,x} Z)^{ 2\dim(Z)} \longrightarrow \IR.
\end{equation*}
By \cite[Lemma 3.8]{Voisin2007}, $\mathrm{vol}(g_{\omega,x})$ agrees with $\dim(Z)!^{-1}(\omega|_{Z})^{\wedge \dim(Z)}$. If $\omega$ is continuous, this implies immediately that there is a euclidean neighborhood $U$ of $x$ in $Z$ such that $\int_U (\omega|_{Z})^{\wedge \dim(Z)}>0$. 

To use this argument effectively, we need a criterion to check whether the restriction of $g_{\omega,x}$ to $T_{\IR,x}Z$ is positive definite. For an arbitrary $\IR$-bilinear form $g$ on a real vector space $V$, we define its kernel by
\begin{equation*}
\mathrm{ker}(g) = \{ v \in V \ | \ \forall w\in V: g(v,w) = 0\}.
\end{equation*}
In our applications, $\omega$ is always semipositive so that $g_{\omega}$ is positive semidefinite. For a positive semidefinite bilinear form $g$, we have
\begin{equation} \label{equation::alternativekernel}
\ker(g) = \{ v \in V \ | \ g(v,v) = 0 \}
\end{equation}
and hence that $\ker(g|_W) = \ker(g) \cap W$ for any $\IR$-linear subspace $W \subset V$. Consequently, the restriction of $g_{\omega,x}$ to $T_{\IR,x}Z$ is positive definite if and only if $\ker(g_{\omega,x}) \cap T_{\IR,x} Z = \{ 0\}$. Finally, let us note that for any positive semidefinite $\IR$-bilinear forms $g_1, g_2$ on $V$ their sum $g_1+g_2$ is also a positive semidefinite $\IR$-bilinear form and (\ref{equation::alternativekernel}) implies that
\begin{equation} \label{equation::kernelintersection}
\ker(g_1+g_2) = \ker(g_1) \cap \ker(g_2).
\end{equation}
Finally, let us remark that $\ker(\omega_x)=\ker(g_{\omega,x})$ for each $(1,1)$-form of real type on $Y$ and every point $x \in Y$ -- under the condition that we consider $\omega$ as a bilinear form on $T_\IR Y$. We use this fact to simplify our notion in Section \ref{section::distribution}.

\section{Weil Functions, Hermitian Metrics and Chern Forms}
\label{section::chernforms}

We provide here the necessary tools for Section \ref{section::intersections}, in which bounds on certain intersections numbers are established. Our approach is to endow all line bundles under consideration with smooth hermitian metrics so that intersection numbers become integrals of the associated Chern forms. Throughout this section, we hence take $k=\IC$ as our base field. A major issue is to interpolate between the Chern forms of different line bundles. For this purpose, we introduce certain explicit smooth $(1,1)$-forms of real type, namely the ``toric'' $(1,1)$-forms $\omega(\phi_{\mathrm{tor}})$ in Subsection \ref{subsection::toric} and the ``abelian'' $(1,1)$-forms $\omega(N; \phi_{\mathrm{ab}})$ in Subsection \ref{subsection::abelian}.

\subsection{``Toric'' (1,1)-forms}
\label{subsection::toric}

Our first aim is to endow the line bundles $M_{\overline{G}}$ from Construction \ref{construction1} with a hermitian metric and to compute the associated Chern forms.  Functoriality allows us to endow additionally the line bundles $M_{\overline{\Gamma(\varphi_{\mathrm{tor}})}}$ from Construction \ref{construction3} with a hermitian metric. A closer look at the associated Chern forms leads us to introduce the ``toric'' $(1,1)$-forms $\omega(\phi_{\mathrm{tor}})$.

Our main instrument are Weil functions, on which the reader may find details in \cite[Chapter 10]{Lang1983} and \cite[Chapter I]{Lang1988}. Let $X$ be a complex algebraic variety and $D$ a Cartier divisor on $X$. In this situation, a function $\lambda\!: (X \setminus \supp(D))(\IC) \rightarrow \IR$ is called a Weil function for $D(\IC)$ if every point $x \in X(\IC)$ has an open neighborhood $U$ (in the euclidean topology) such that
\begin{equation} \label{equation::weil}
\lambda = - \log |f| + \alpha \text{ on $U \setminus \supp(D)(\IC)$}
\end{equation}
with $f$ a meromorphic function on $U$ such that $\Div(f)= D|_U$ (as formal sums of irreducible analytic varieties on $U$) and $\alpha$ a continuous function on $U$. Furthermore, $\lambda$ is called a smooth Weil function if $\alpha$ can be even assumed smooth on $U$. Every (smooth) Weil function $\lambda\!: (X \setminus \supp(D))(\IC) \rightarrow \IR$ associated with $D$ yields a (smooth) hermitian metric $g$ on the associated line bundle $L(D)$. In fact, its sections $\mathcal{O}(D)$ form a $\mathcal{O}_X$-submodule of $\mathcal{K}_X$ and we can just set $g_x(f)= e^{-\lambda(x)}|f(x)|$ for any meromorphic function $f$ on $U$ and any $x\in X(\IC)$ in its domain of definition. To a smooth Weil function $\lambda$ for $D$ is associated a smooth closed $(1,1)$-form of real type, namely the Chern form $c_1(L(D),g)$ of the associated smooth metric on $L(D)$. On an open euclidean neighborhood $U$ such that (\ref{equation::weil}) is true, we have $c_1(L,g)=dd^c \alpha$. Additionally, $dd^c \alpha = dd^c \lambda$ outside $\supp(D)(\IC)$.

We now record a standard result on Weil functions.  Let $D$ be a Cartier divisor on a complex projective variety $X$ and $\lambda$ be a Weil function for $D$. Assume that $D$ is the difference $D_1 - D_2$ of two effective Cartier divisors $D_1,D_2$ with \textit{disjoint} supports. From \cite[Propositions 10.2.1 and 10.3.2]{Lang1983}, we infer that $\sup\{\lambda, 0\}$ (resp.\ $-\inf\{\lambda, 0\}$) is a Weil function for $D_1$ (resp.\ $D_2$). The next lemma provides a smooth variant of this observation in the same situation.

\begin{lemma} \label{lemma::smoothweilfunctions} In the situation described above, assume additionally that $\lambda$ is a smooth Weil function. Then, $\log(1+e^{2\lambda})/2$ (resp.\ $\log(1+e^{-2\lambda})/2$) is a smooth Weil function for $D_1$ (resp.\ $D_2$).
\end{lemma}

\begin{proof} By assumption, we know that for each $x \in X(\IC)$ there exists an open euclidean neighborhood $U$, a meromorphic function $f$ representing $D$ on $U$ and a smooth function $\alpha$ satisfying (\ref{equation::weil}). Since $D_1$ and $D_2$ have disjoint supports, we may shrink $U$ to guarantee that it is relatively compact and that its topological closure $\overline{U}$ does not intersect $\supp(D_1)(\IC)$ or $\supp(D_2)(\IC)$. Suppose $\overline{U} \cap D_1 = \emptyset$ (resp.\ $\overline{U} \cap D_2 = \emptyset$). Then, $|f|\geq \varepsilon >0$ (resp.\ $|f|\leq \varepsilon^{-1}$) for some sufficiently small $\varepsilon>0$. Furthermore, $1$ (resp.\ $f$) is a local equation for $D_1$. Note that $\beta = \log(1 + |f|^{-2}e^{2\alpha})/2$ (resp.\ $\beta= \log(|f|^2+e^{2\alpha})/2$) is a smooth function on $U$.\footnote{Note that $z \mapsto |z|^2 = x^2 + y^2$ is smooth at $z=0$ in contrast to $z \mapsto |z|= \sqrt{x^2+y^2}$. This rules out the straightforward choice $\log(1+e^\lambda)$ (resp.\ $\log(1+e^{-\lambda})$).} In addition,
\begin{equation*}
\frac{1}{2}\log(1+e^{2\lambda}) = -\log|1| + \beta \text{ (resp.\ } \frac{1}{2}\log(1+e^{2\lambda}) = -\log|f| + \beta \text{),}
\end{equation*}
This demonstrates that $\frac{1}{2}\log(1+e^{2\lambda})$ is a smooth Weil function for $D_1$. Similarly, $\frac{1}{2}\log(1+e^{-2\lambda})$ can be shown to be a smooth Weil function for $D_2$.
\end{proof}
Let us next recollect a fundamental result of Vojta \cite{Vojta1996}. Let $G$ be a semiabelian variety with split toric part $T=\Gm^{t}$ and abelian quotient $A$. Recall from Construction \ref{construction1} its compactification $\overline{G}$ as well as the Cartier divisor $G(D_{t})$ on $\overline{G}$. With $\pr_u: (\IP^1)^{t} \rightarrow \IP^1$ being the projection to the $u$-th component as in Construction \ref{construction0}, we set $D_{u,0}=G(\pr_u^\ast E_0)$ and $D_{u,\infty}=G(\pr_u^\ast E_\infty)$ so that $G(D_{t})=\sum_{u=1}^{t} (D_{u,0} + D_{u,\infty})$.

\begin{lemma} \label{lemma::goodweilfunction}
For each divisor $D_{u,0}-D_{u,\infty}$, $1\leq u \leq t$, there exists a unique smooth Weil function 
\begin{equation*}
\lambda_u\!: \overline{G}(\IC) \setminus \supp(D_{u,0}-D_{u,\infty})(\IC) \longrightarrow \IR
\end{equation*}
that satisfies
\begin{equation*}
\lambda_u(x+y)= \lambda_u(x)+\lambda_u(y)
\end{equation*}
for all $x,y \in G(\IC)$. In addition, $e^{\lambda_u}$ is locally the absolute value of a meromorphic function.
\end{lemma}

Outside $\supp(D_{u,0}-D_{u,\infty})(\IC)$, we have locally $\lambda_u=\log |f|$ for some holomorphic function $f$. This implies $dd^c \lambda_u = 0$ on $G(\IC)$.

\begin{proof} This is stated in \cite[Proposition 2.6]{Vojta1996} except for the assertion about $e^{\lambda_u}$. Inspecting (2.6.3) in the proof of the said proposition, one sees that it suffices to prove the same assertion for the N\'eron function $\lambda_{(s)}$ (cf.\ \cite[Theorem 11.1.1]{Lang1983}) attached to the divisor $(s)$ on $A$. As $s$ is a rational section of an algebraically trivial line bundle by construction, the divisor $(s)$ is algebraically equivalent to the zero divisor. The explicit formula for $\lambda_{(s)}$ in terms of a normalized theta function (\cite[Theorem 13.1.1]{Lang1983}) directly yields the assertion in this case; the hermitian form $H$ in the formula is zero because $(s) \sim_{\mathrm{alg}} 0$ (cf.\ \cite[Proposition 2.2.2]{Birkenhake2004}).
\end{proof}

Using the Weil functions $\lambda_u$ we can define a subgroup
\begin{equation} \label{equation::maximalcompactsubgroup}
\{ x \in G(\IC) \ | \ \lambda_1(x)=\lambda_2(x)= \cdots = \lambda_{t}(x) = 0 \} \subset G(\IC).
\end{equation}
This coincides with the maximal compact subgroup $K_G$ of $G(\IC)$. Indeed, any homomorphism $K_G \rightarrow \IR$ vanishes by compactness so that $\lambda_u|_{K_G}=0$. By uniqueness, the restriction of $\lambda_u$ to the maximal torus $T(\IC)$ equals $-\log|X_u|$ (in standard coordinates $X_1,\dots,X_{t}$). Hence, the subgroup in (\ref{equation::maximalcompactsubgroup}) is topologically a fiber bundle with compact fiber $(S^1)^{t}$, $S^1=\{ z \in \IC \ | \ |z|=1 \}$, over the compact base $A(\IC)$. Therefore, it is compact itself and hence contained in $K_G$. As $G(\mathbb{C})$ is Hausdorff, its maximal compact subgroup $K_G$ is a closed subgroup. Using \cite[Theorem II.2.3]{Helgason2001} and counting dimensions, we see that $K_G$ is a real Lie subgroup of (real) dimension $2\dim(A)+t$.

Recall that $M_{\overline{G}}$ is the line bundle associated to the $T$-invariant Weil divisor $G(D_{t})$. By Lemmas \ref{lemma::smoothweilfunctions} and \ref{lemma::goodweilfunction}, the function
\begin{equation} \label{equation::weilfunction}
\lambda = \frac{1}{2} \sum_{u=1}^{t} \left( \log(1+e^{2\lambda_u})+ \log(1+e^{-2\lambda_u}) \right)
\end{equation} 
is a smooth Weil function for $G(D_{t})$. For the associated smooth hermitian line bundle, which is denoted $\overline{M}_{\overline{G}}$ in the sequel, we have
\begin{equation*}
c_1(\overline{M}_{\overline{G}}) = \frac{1}{2} \sum_{u=1}^{t} \left( dd^c\log(1+e^{2\lambda_u})+ dd^c\log(1+e^{-2\lambda_u}) \right) \, \text{on $G(\IC)$}.
\end{equation*}

The Weil functions of Lemma \ref{lemma::goodweilfunction} also satisfy some functoriality. To be precise, let $\varphi\!: G \rightarrow G^\prime$ be a homomorphism of semiabelian varieties with toric component $\varphi_{\mathrm{tor}}\!: T = \Gm^{t} \rightarrow T^\prime = \Gm^{t^\prime}$. Let $X_i$ (resp.\ $Y_j$) be the standard algebraic coordinates on $\Gm^{t}$ (resp.\ $\Gm^{t^\prime}$) and write $\varphi_{\mathrm{tor}}^\ast(Y_v)=X_1^{a_{1 v}}\cdots X_{t}^{a_{t v}}$ with integers $a_{uv}$. Lemma \ref{lemma::goodweilfunction} supplies Weil functions $\lambda_v^\prime$, $1 \leq u\leq t^\prime$, on $G^\prime$ and there is an identity
\begin{equation} \label{equation::lambda_functoriality}
\varphi^\ast \lambda^\prime_v = \lambda_v^\prime \circ \varphi =  a_{1 v} \lambda_1 + a_{2 v} \lambda_2  + \cdots + a_{t v} \lambda_{t} \, \text{on $G(\IC)$.}
\end{equation}
Indeed, the equality is valid on $T$ since the restriction of $\lambda_u$ (resp.\ $\lambda_v^\prime$) to the maximal torus $T(\IC) \approx (\IC^\times)^{t}$ (resp.\ $T^\prime(\IC) \approx (\IC^\times)^{t^\prime}$) is $(-\log |X_u|)$ (resp.\ $(-\log |Y_v|)$) as we noted above. It is also true on $K_{G}$ because $\varphi(K_{G})\subset K_{G^\prime}$. As $K_{G}$ and $T(\IC)$ generate $G(\IC)$ as a group, (\ref{equation::lambda_functoriality}) is true for all of $G(\IC)$. Note that $\varphi^\ast\lambda_v^\prime$ is independent of the abelian component of $\varphi$. Abusing notation, we therefore write $\varphi_{\mathrm{tor}}^\ast \lambda_v^\prime$ instead of $\varphi^\ast \lambda_v^\prime$. Even more, we can use (\ref{equation::lambda_functoriality}) to formally define $\phi_{\mathrm{tor}}^\ast \lambda_v^\prime: G(\IC) \rightarrow \IR$, $v=1,\dots, t^\prime$, for an arbitrary $\phi_{\mathrm{tor}} \in \Hom_\IR(\Gm^{t},\Gm^{t^\prime})$.

We now apply the results of the previous paragraph to endow the line bundles $M_{\overline{\Gamma(\varphi_{\mathrm{tor}})}}$, $\varphi_{\mathrm{tor}} \in \Hom(\Gm^{t},\Gm^{t^\prime})$, with hermitian metrics. For this, we use the homomorphism $\varphi: G \rightarrow G^\prime$ from Construction \ref{construction2} with $\overline{\varphi}^\ast M_{\overline{G}^\prime} \approx M_{\overline{\Gamma(\varphi_{\mathrm{tor}})}}$. We may endow $M_{\overline{\Gamma(\varphi_{\mathrm{tor}})}}$ with a hermitian metric such that $\overline{\varphi}^\ast \overline{M}_{\overline{G}^\prime} \approx \overline{M}_{\overline{\Gamma(\varphi_{\mathrm{tor}})}}$. Since the isomorphism between $\overline{\varphi}^\ast M_{\overline{G}^\prime}$ and $M_{\overline{\Gamma(\varphi_{\mathrm{tor}})}}$ is unique up to multiplication with a non-zero constant, this singles out a hermitian metric on $\overline{M}_{\overline{\Gamma(\varphi_{\mathrm{tor}})}}$ up to a non-zero constant scaling factor. Regardless of this indeterminate scaling factor, we have an identity of Chern forms $c_1(\overline{M}_{\overline{\Gamma(\varphi_{\mathrm{tor}})}})=\overline{\varphi}^\ast c_1(\overline{M}_{\overline{G}^\prime})$. Thus,
\begin{equation} \label{equation::chernformpullback}
c_1(\overline{M}_{\overline{\Gamma(\varphi_{\mathrm{tor}})}})
= \frac{1}{2}\sum_{v=1}^{t^\prime} \left( dd^c  \log(1+e^{2\varphi_{\mathrm{tor}}^\ast \lambda^\prime_v}) + dd^c \log(1+e^{-2\varphi_{\mathrm{tor}}^\ast \lambda^\prime_v}) \right) \text{ on $G(\IC)$.}
\end{equation}
Since the indeterminacy in the metric is negligible for our purposes, we suppress it in writing $\overline{M}_{\overline{\Gamma(\varphi_{\mathrm{tor}})}}$ for any hermitian line bundle as constructed above.
Again, the right hand side of (\ref{equation::chernformpullback}) depends only on $\varphi_{\mathrm{tor}}$ and is moreover well-defined for any $\phi_{\mathrm{tor}} \in \Hom_\IR(\Gm^{t},\Gm^{t^\prime})$. In other words, we can associate with each $\phi_{\mathrm{tor}}\in \Hom_\IR(\Gm^{t},\Gm^{t^\prime})$ a $(1,1)$-form
\begin{equation*}
\omega(\phi_{\mathrm{tor}}) = \frac{1}{2} \sum_{v=1}^{t^\prime} \left( dd^c  \log(1+e^{2\phi_{\mathrm{tor}}^\ast \lambda^\prime_v}) + dd^c \log(1+e^{-2\phi_{\mathrm{tor}}^\ast \lambda^\prime_v}) \right)
\end{equation*}
on $G(\IC)$. (Note that we do not claim that $\omega(\phi_{\mathrm{tor}})$ extends to any compactification of $G(\IC)$. In the proof of Lemma \ref{lemma::chernforms_quasihomomorphisms} we give such an extension in the case where $\phi_{\mathrm{tor}} \in \Hom(\Gm^t,\Gm^{t^\prime})$, but we neither can prove the existence of an extension in general nor do we need it.) In the remainder of this section, we establish basic properties of this $(1,1)$-form.

\begin{lemma} \label{lemma::positivity} Each $dd^c \log(1+e^{\pm 2\phi_{\mathrm{tor}}^\ast \lambda_v^\prime})$, $1\leq v \leq t^\prime$, is a semipositive $(1,1)$-form of real type on $G(\IC)$. Consequently, $\omega(\phi_{\mathrm{tor}})$ is a semipositive $(1,1)$-form of real type.
\end{lemma}

\begin{proof} It suffices to prove that $\log(1+e^{\pm 2\phi_{\mathrm{tor}}^\ast \lambda_v^\prime})$ is a plurisubharmonic function. This follows directly from $dd^c \lambda_u = 0$ (i.e., both $\lambda_u$ and $-\lambda_u$ are plurisubharmonic on $G(\IC)$) and the fact that $\log(1+e^x)$ is a convex monotonously increasing function (cf.\ \cite[Theorem K.5 (d)]{Gunning1990}).
\end{proof}

Furthermore, the map $\phi_{\mathrm{tor}} \mapsto \omega(\phi_{\mathrm{tor}})$ is continuous with respect to the euclidean topology on $\Hom_\IR(\Gm^{t},\Gm^{t^\prime}) \approx \IR^{t \times t^\prime}$ and the usual topology on smooth $(1,1)$-forms (cf.\ \cite[Section I.2]{Demailly2012} and \cite[Section 1.46]{Rudin1991}). By Lemma \ref{lemma::goodweilfunction}, there exists locally on $G(\IC)$ a non-zero holomorphic function $\kappa_u$ such that $e^{\lambda_u^\prime}=|\kappa_u|$. From $\lambda_u^\prime = \log (|\kappa_u|^2)/2$, we deduce
\begin{equation*}
\del \lambda_u^\prime = \frac{\overline{\kappa_u} \del \kappa_u}{2|\kappa_u|^2} = \frac{\del \kappa_u}{2 \kappa_u}
\text{ and } 
\delbar \lambda_u^\prime = \frac{\kappa_u \delbar \overline{\kappa_u}}{2|\kappa_u|^2} = \overline{\left(\frac{\del \kappa_u}{2 \kappa_u}\right)}.
\end{equation*}
Using Lemma \ref{lemma::ddc_computation} below, we hence obtain
\begin{equation} \label{equation::cherntoric}
dd^c \log(1+e^{2\phi_{\mathrm{tor}}^\ast \lambda^\prime_v}) +
dd^c \log(1+e^{-2\phi_{\mathrm{tor}}^\ast \lambda^\prime_v}) =
\frac{2i(\del \phi_{\mathrm{tor}}^\ast \lambda_v^\prime \wedge \delbar \phi_{\mathrm{tor}}^\ast \lambda_v^\prime)}{\pi(1+e^{2\phi_{\mathrm{tor}}^\ast \lambda^\prime_v})(1+e^{-2\phi_{\mathrm{tor}}^\ast \lambda^\prime_v})}.
\end{equation}

\begin{lemma} \label{lemma::ddc_computation} Let $\kappa_1,\dots,\kappa_m$ be zero-free holomorphic functions on an open subset $U \subset \IC^n$. Then,
\begin{equation*}
d d^c \log(1+|\kappa_1|^{a_1} \cdots |\kappa_m|^{a_m}) = \frac{i}{4\pi} \frac{|\kappa_1|^{a_1} \cdots |\kappa_m|^{a_m}}{(1+|\kappa_1|^{a_1}\cdots |\kappa_m|^{a_m})^2} \left(\sum_{u=1}^m \frac{ a_u \del \kappa_u}{\kappa_u}  \wedge \overline{\sum_{u=1}^m \frac{ a_u \del \kappa_u}{\kappa_u}} \right)
\end{equation*}
for any real numbers $a_1, \dots, a_m$.
\end{lemma}

\begin{proof} Without loss of generality, we assume that all $a_1, \dots, a_m$ are non-zero. To simplify our notation, we set $f(z)= |\kappa_1|^{a_1}\cdots |\kappa_m|^{a_m}$. First, we note that $\del |\kappa|^q = \del (\kappa\overline{\kappa})^{q/2} = \frac{q}{2} |\kappa|^{q-2} \overline{\kappa} \del \kappa 
= \frac{q}{2} |\kappa|^{q} \left( \frac{\del \kappa}{\kappa} \right)$ (resp.\ $\delbar |\kappa|^q= \frac{q}{2} |\kappa|^{q} \overline{\left( \frac{\del \kappa}{\kappa} \right)}$) implies that
\begin{equation*}
\del f(z) = f(z) \cdot \sum_{u=1}^m \left( \frac{a_u \del \kappa_u}{2\kappa_u} \right) \text{ (resp. } \delbar f(z) = f(z) \cdot \sum_{u=1}^m \overline{\left(\frac{a_u \del \kappa_u}{2 \kappa_u}\right)} \text{).}
\end{equation*}
Since both $\kappa_u$ and $\kappa_u^{-1}$ are holomorphic, we have
\begin{equation*}
\del \overline{\kappa_u^{-1}\del\kappa_u}
= \overline{\delbar (\kappa_u^{-1} \del \kappa_u)}
= \overline{\delbar (\kappa_u^{-1})\wedge \del \kappa_u+\kappa_u^{-1}\delbar \del (\kappa_u)}
=0
\end{equation*}
and hence
\begin{align*}
\del \delbar f(z) 
&= \del f(z) \wedge \left(\sum_{u=1}^m \overline{\left(\frac{a_u \del \kappa_u}{2 \kappa_u}\right)} \right) + f(z) \cdot \del \left(\sum_{u=1}^m \overline{\left(\frac{a_u \del \kappa_u}{2 \kappa_u}\right)} \right) \\
&= f(z) \left(\sum_{u=1}^m \frac{ a_u \del \kappa_u}{2\kappa_u} \right) \wedge \left(\sum_{u=1}^m \overline{\left(\frac{a_u \del \kappa_u}{2 \kappa_u}\right)} \right).
\end{align*}
We compute
\begin{align*}
\del \delbar \log(1+f(z)) 
&= \del \left( \frac{1}{1+f(z)}\delbar f(z)\right) \\
&= \frac{-1}{(1+f(z))^2} \del f(z) \wedge \delbar f(z)+ \frac{1}{1+f(z)}\del \delbar f(z) \\
&= \left( \frac{-f(z)^2}{(1+f(z))^2} + \frac{f(z)}{1+f(z)}\right) \left(\sum_{u=1}^m \frac{ a_u \del \kappa_u}{2\kappa_u} \wedge \sum_{u=1}^m \overline{\left(\frac{a_u \del \kappa_u}{2 \kappa_u}\right)} \right) \\
&= \frac{f(z)}{4(1+f(z))^2} \left(\sum_{u=1}^m \frac{ a_u \del \kappa_u}{\kappa_u} \wedge \sum_{u=1}^m \overline{\left(\frac{a_u \del \kappa_u}{\kappa_u}\right)} \right).
\end{align*}
The assertion follows directly as $dd^c = (i/2\pi)(\del + \delbar)(\delbar - \del) = (i/\pi) \del \delbar$.
\end{proof}

The next lemma establishes an essential homogeneity property for $\omega(\phi_{\mathrm{tor}})$.

\begin{lemma} \label{lemma::chernforms_quasihomomorphisms} 
Let $G$ be a semiabelian variety with abelian quotient $\pi: G \rightarrow A$ and toric part $\Gm^{t}$. In addition, let $t^\prime$ be a non-negative integer and $\varpi$ a smooth closed $(1,1)$-form on $A(\IC)$. For every $\phi_{\mathrm{tor}} \in \Hom_\IQ(\Gm^{t},\Gm^{t^\prime})$, every algebraic subvariety $X \subset G$, and every non-negative integers $s_1$, $s_2$ satisfying $s_1+s_2=\dim(X)$, the integral
\begin{equation*}
\int_{X(\IC)} \omega(\phi_{\mathrm{tor}})^{\wedge s_1} \wedge (\pi^\ast \varpi)^{\wedge s_2} 
\end{equation*}
is finite and
\begin{equation} \label{equation::integrals}
\int_{X(\IC)} \omega(n \cdot \phi_{\mathrm{tor}})^{\wedge s_1} \wedge (\pi^\ast \varpi)^{\wedge s_2} = n^{s_1} \cdot \int_{X(\IC)} \omega(\phi_{\mathrm{tor}})^{\wedge s_1} \wedge (\pi^\ast \varpi)^{\wedge s_2}
\end{equation}
for each non-negative integer $n$.
\end{lemma}


\begin{proof}
Let us first prove the lemma assuming that $\phi_{\mathrm{tor}} = \varphi_{\mathrm{tor}} \in \Hom(\Gm^{t},\Gm^{t^\prime})$. In this situation, $\omega(\varphi_{\mathrm{tor}})$ extends to a smooth closed $(1,1)$-form on the (complex) analytic space $G_{\overline{\Gamma(\varphi_{\mathrm{tor}})}}(\IC)$. Indeed, $\omega(\varphi_{\mathrm{tor}})$ is precisely defined to agree with the restriction of the smooth differential form $\widetilde{\omega}(\varphi_{\mathrm{tor}})=c_1(\overline{M}_{\overline{\Gamma(\varphi_{\mathrm{tor}})}})$ on $G_{\overline{\Gamma(\varphi_{\mathrm{tor}})}}(\IC)\supset G(\IC)$. Writing $\overline{X}$ for the Zariski closure of $X$ in $G_{\overline{\Gamma(\varphi_{\mathrm{tor}})}}$, we have hence
\begin{equation*}
\int_{X(\IC)} \omega(\phi_{\mathrm{tor}})^{\wedge s_1} \wedge (\overline{\pi}^\ast \varpi)^{\wedge s_2} = \int_{\overline{X}(\IC)} \widetilde{\omega}(\varphi_{\mathrm{tor}})^{\wedge s_1} \wedge (\overline{\pi}^\ast \varpi)^{\wedge s_2}
\end{equation*}
because $(\overline{X} \setminus X)(\IC)$ is of positive codimension in $X(\IC)$. As we are integrating a smooth differential form over a compact analytic space, the integral on the right-hand side is evidently finite.

To show the second part of the assertion, still assuming that $\phi_{\mathrm{tor}} = \varphi_{\mathrm{tor}} \in \Hom(\Gm^{t},\Gm^{t^\prime})$, we start note that also $c_1(\overline{M}_{\overline{\Gamma(n \cdot \varphi_{\mathrm{tor}})}})$ is an extension of $\omega(n \cdot \varphi_{\mathrm{tor}})$ on $G_{\overline{\Gamma(n \cdot \varphi_{\mathrm{tor}})}}(\IC) \supset G(\IC)$. In addition, Construction \ref{construction3} supplies us with a map $\vartheta_{\varphi_{\mathrm{tor}},n}: G_{\overline{\Gamma(\varphi_{\mathrm{tor}})}} \rightarrow G_{\overline{\Gamma(n\cdot \varphi_{\mathrm{tor}})}}$, which is the identity on $G$. Therefore, the smooth closed $(1,1)$-form $\widetilde{\omega}(n\cdot \varphi_{\mathrm{tor}}) =\vartheta_{\varphi_{\mathrm{tor}},n}^\ast c_1(\overline{M}_{\overline{\Gamma(n\cdot \varphi_{\mathrm{tor}})}})$ extends $\omega(n \cdot \varphi_{\mathrm{tor}})$ to $G_{\overline{\Gamma(\varphi_{\mathrm{tor}})}}(\IC) \supset G(\IC)$. (Note that both extensions $\widetilde{\omega}(\varphi_{\mathrm{tor}})$ and $\widetilde{\omega}(n \cdot \varphi_{\mathrm{tor}})$ are actually unique.)


Denote by $\overline{\pi}: G_{\overline{\Gamma(\varphi_{\mathrm{tor}})}} \rightarrow A$ the abelian quotient. Since the boundary $(\overline{X} \setminus X)(\IC)$ has measure zero, (\ref{equation::integrals}) would follow from the equality 
\begin{equation} \label{equation::integrals2}
\int_{\overline{X}(\IC)} \widetilde{\omega}(m \cdot \varphi_{\mathrm{tor}})^{\wedge s_1} \wedge (\overline{\pi}^\ast \varpi)^{\wedge s_2} = m^{s_1} \cdot \int_{\overline{X}(\IC)} \widetilde{\omega}(\varphi_{\mathrm{tor}})^{\wedge s_1} \wedge (\overline{\pi}^\ast \varpi)^{\wedge s_2}.
\end{equation}
For this, we claim that any function
\begin{equation*}
\rho_v^{\pm}: G(\IC)\rightarrow \IR^{>0}, \, x \mapsto (1+e^{\pm 2 m \varphi_{\mathrm{tor}}^\ast \lambda_v^\prime})/(1+e^{\pm 2 \varphi_{\mathrm{tor}}^\ast \lambda_v^\prime})^m, 1 \leq v \leq t^\prime,
\end{equation*}
extends smoothly to $G_{\overline{\Gamma(\varphi_{\mathrm{tor}})}}(\IC)$. It suffices to prove that each $x\in (G_{\overline{\Gamma(\varphi_{\mathrm{tor}})}}\setminus G)(\IC)$ has a (euclidean) neighborhood on which $\rho_v^\pm$ extends smoothly. For this, we let $\varphi: G \rightarrow G^\prime$ be again the homomorphism from Construction \ref{construction2} so that $M_{\overline{\Gamma(\varphi_{\mathrm{tor}})}} \approx \overline{\varphi}^\ast M_{\overline{G}^\prime}$. As before, Lemma \ref{lemma::goodweilfunction} affords a Weil function $\lambda_v^\prime$, $1 \leq v \leq t^\prime$, for each divisor $D_{v,0}^\prime - D_{v,\infty}^\prime$ on $\overline{G}^\prime$. Its pullback $\overline{\varphi}^\ast \lambda_v^\prime$ along $\overline{\varphi}: G_{\overline{\Gamma(\varphi_{\mathrm{tor}})}} \rightarrow \overline{G}^\prime$ restricts to the function $\varphi_{\mathrm{tor}}^\ast \lambda_v^\prime: G(\IC) \rightarrow \IR$ formally defined by (\ref{equation::lambda_functoriality}). There exists a euclidean neighborhood $U$ of $\overline{\varphi}(x)$ and a meromorphic function $f$ on $U$ with $\Div(f)=(D_{v,0}^\prime-D_{v,\infty}^\prime)|_{U}$ such that $\lambda_{v}^\prime + \log|f|$ extends to a smooth function $\alpha$ on $U$. On $G(\IC) \cap \overline{\varphi}^{-1}(U) \subset G_{\overline{\Gamma(\varphi_{\mathrm{tor}})}}(\IC)$, we have
\begin{equation*}
\varphi_{\mathrm{tor}}^\ast \lambda_v^\prime = \lambda_v^\prime \circ \overline{\varphi} = - \log |f \circ \overline{\varphi}| + (\alpha \circ \overline{\varphi}) 
\end{equation*}
and thus
\begin{equation} \label{equation::extension_rho}
\rho_v^{\pm} = \frac{1+ |f \circ \overline{\varphi}|^{\mp 2m}e^{\pm 2m(\alpha \circ \overline{\varphi})}}{(1+ |f \circ \overline{\varphi}|^{\mp 2}e^{\pm 2(\alpha \circ \overline{\varphi})})^m} = \frac{|f \circ \overline{\varphi}|^{\pm 2m}+e^{\pm 2m(\alpha \circ \overline{\varphi})}}{(|f \circ \overline{\varphi}|^{\pm 2}+e^{\pm 2(\alpha \circ \overline{\varphi})})^m}.
\end{equation}
Since $\supp(D_{v,0}^\prime) \cap \supp(D_{v,\infty}^\prime) = \emptyset$, we have $x \notin \supp(D_{v,0}^\prime)$ or $x \notin \supp(D_{v,\infty}^\prime)$. Shrinking $U$ if necessary, we may hence assume that $f$ or $f^{-1}$ is holomorphic on $U$. In either case, (\ref{equation::extension_rho}) yields a smooth extension of $\rho_v^{\pm}$ on $\overline{\varphi}^{-1}(U)$. By uniqueness, these extensions glue together to a smooth function $\widetilde{\rho}_v^\pm: G_{\overline{\Gamma(\varphi_{\mathrm{tor}})}}(\IC) \rightarrow \IR^{>0}$. (In fact, $\widetilde{\rho}_v^\pm(x)=1$ for all $x \in \supp(D_{v,0}^\prime)(\IC) \cup \supp(D_{v,\infty}^\prime)(\IC)$ because $(1+e^{m x})/(1+e^{x})^m \rightarrow 1$ if $x \rightarrow \pm \infty$.) In addition,
\begin{equation} \label{equation::exactness}
\widetilde{\omega}(m\cdot \varphi_{\mathrm{tor}}) - m\widetilde{\omega}(\varphi_{\mathrm{tor}}) = \frac{1}{2} \sum_{v=1}^{t^\prime} \left( dd^c \log(\widetilde{\rho}_v^+) + dd^c \log(\widetilde{\rho}_v^-) \right);
\end{equation}
indeed, this equality is obvious on $G(\IC)$ and any $(1,1)$-form on $G(\IC)$ has at most one smooth extension to the compactification $G_{\overline{\Gamma(\varphi_{\mathrm{tor}})}}(\IC)$. We deduce from (\ref{equation::exactness}) that $\widetilde{\omega}(m\cdot \varphi_{\mathrm{tor}}) - m\widetilde{\omega}(\varphi_{\mathrm{tor}})$ is exact and hence Stokes' theorem (\cite[p.\ 33]{Griffiths1994}) in combination with a partition of unity implies (\ref{equation::integrals2}).

Now, let us consider a general $\phi_{\mathrm{tor}} \in \Hom_\IQ(\Gm^{t},\Gm^{t^\prime})$ and a positive integer $n$ that is a denominator for $\phi_{\mathrm{tor}}$ (i.e., $n \cdot \phi_{\mathrm{tor}} \in \Hom(\Gm^{t},\Gm^{t^\prime})$). Setting $Y=[n]^{-1}(X)$, the restriction $[n]|_Y: Y \rightarrow X$ is finite \'etale of degree $n^{2\dim(A)+t}$. By functoriality (\ref{equation::lambda_functoriality}), we have $[n]^\ast \omega(\phi_{\mathrm{tor}}) = \omega(n \cdot \phi_{\mathrm{tor}})$ and $[n]^\ast \omega(m \cdot \phi_{\mathrm{tor}}) = \omega(m \cdot n \cdot \phi_{\mathrm{tor}})$. We infer that
\begin{equation*}
n^{2\dim(A)+t}\int_{X(\IC)} \omega(m\cdot \phi_{\mathrm{tor}})^{\wedge s_1} \wedge (\overline{\pi}^\ast \varpi)^{\wedge s_2} = \int_{Y(\IC)} \omega(m \cdot n \cdot \phi_{\mathrm{tor}})^{\wedge s_1} \wedge (\overline{\pi}^\ast [n]^\ast \varpi)^{\wedge s_2}
\end{equation*}
and 
\begin{equation*}
n^{2\dim(A)+t}\int_{X(\IC)} \omega(\phi_{\mathrm{tor}})^{\wedge s_1} \wedge (\overline{\pi}^\ast \varpi)^{\wedge s_2} = \int_{Y(\IC)} \omega(n \cdot \phi_{\mathrm{tor}})^{\wedge s_1} \wedge (\overline{\pi}^\ast [n]^\ast \varpi)^{\wedge s_2}.
\end{equation*}
Since $n \cdot \phi_{\mathrm{tor}} \in \Hom(\Gm^{t},\Gm^{t^\prime})$, this reduces the assertion of the lemma to the already proven special case.
\end{proof}

\subsection{``Abelian'' (1,1)-forms}
\label{subsection::abelian}

This subsection is the ``abelian'' equivalent of the last one and we introduce here $(1,1)$-forms $\omega(N;\phi_{\mathrm{ab}})$ analogous to the $(1,1)$-forms $\omega(\phi_{\mathrm{tor}})$. In fact, we construct a $(1,1)$-form $\omega(N;\phi_{\mathrm{ab}})$ on $A(\IC)$ for each $\phi_{\mathrm{ab}} \in \Hom_\IR(A,A^\prime)$, $A$ and $A^\prime$ abelian varieties, and each ample line bundle $N$ on $A^\prime$. Having pullbacks from abelian quotients at our disposal, it suffices here to work on abelian varieties and the definition is technically less demanding.

Let $\varphi_{\mathrm{ab}}: A \rightarrow A^\prime$ be a homomorphism of abelian varieties. We choose lattices $\Lambda \subseteq \mathbb{C}^{g}$, $g=\dim(A)$, and $\Lambda^\prime \subseteq \mathbb{C}^{g^\prime}$, $g^\prime=\dim(A^\prime)$, such that $\IC^{g} \twoheadrightarrow \IC^{g}/\Lambda=A(\IC)$ and $\IC^{g^\prime} \twoheadrightarrow \IC^{g^\prime}/\Lambda^\prime=A^\prime(\IC)$ are universal coverings. In the sequel, each holomorphic tangent space $T_{x} A(\IC)$, $x \in A(\IC)$, (resp.\ $T_x A^\prime(\IC)$, $x^\prime \in A^\prime(\IC)$,) is identified with $\IC^{g}$ (resp.\ $\IC^{g^\prime}$) by virtue of this quotient map. We write $\widetilde{\varphi}_{\mathrm{ab}}: \IC^{g} \rightarrow \IC^{g^\prime}$ for the lifting of $\varphi_{\mathrm{ab}}$ along the universal coverings.

Let $N$ be an ample line bundle on $A^\prime$. The Appell-Humbert Theorem (see e.g.\ \cite[Section 2.2]{Birkenhake2004}) allows us to describe $N$ in terms of a pair $(H,\chi)$ consisting of a hermitian form $H: \mathbb{C}^{g^\prime} \times \mathbb{C}^{g^\prime} \rightarrow \IC$ such that $\mathrm{Im}\,H(\Lambda^\prime,\Lambda^\prime) \subseteq \IZ$ and a semicharacter $\chi: \Lambda^\prime \rightarrow S^1$ for $H$. It is well-known (cf.\ \cite[Exercise 2.6.2]{Birkenhake2004} and \cite[Theorem 7.10]{Voisin2007}) that $N$ can be endowed with a metric $g$ such that the Chern form $c_1(\overline{N})$ of the hermitian line bundle $\overline{N}=(N,g)$ is given by
\begin{equation} \label{equation::chernM}
c_1(\overline{N})_x\!: T_{\IR,x} A^\prime(\IC) \times T_{\IR,x} A^\prime(\IC) = \IC^{g^\prime} \times \IC^{g^\prime} \longrightarrow \IC, \, (v,w) \longmapsto -\mathrm{Im}(H)(v,w), \, x \in A^\prime(\IC).
\end{equation}
Ampleness of $N$ is equivalent to $H$ being positive definite (\cite[Proposition 4.5.2]{Birkenhake2004}), which is equivalent to $c_1(\overline{N})$ being a positive $(1,1)$-form. The pullback of $c_1(\overline{N})$ along $\varphi_{\mathrm{ab}}$ is given by
\begin{equation}\label{equation::chernformpullback2}
\varphi_{\mathrm{ab}}^\ast c_1(\overline{N})_x\!: T_{\IR,x} A(\IC) \times T_{\IR,x} A(\IC) = \IC^{g} \times \IC^{g} \longrightarrow \IC, \, (v,w) \longmapsto -\mathrm{Im}(H)(\widetilde{\varphi}_{\mathrm{ab}}(v),\widetilde{\varphi}_{\mathrm{ab}}(w)),
\end{equation}
for each $x \in A(\IC)$. Lifting homomorphisms $A \rightarrow A^\prime$ to homomorphisms $\IC^{g} \rightarrow \IC^{g^\prime}$ of the universal coverings induces an injection $\Hom_\IR(A,A^\prime) \hookrightarrow \Hom_\IR(\IC^{g},\IC^{g^\prime})$, $\phi_{\mathrm{ab}} \mapsto \widetilde{\phi}_{\mathrm{ab}}$. As in (\ref{equation::chernformpullback}), the right hand side of (\ref{equation::chernformpullback2}) is well-defined for any $\widetilde{\phi}_{\mathrm{ab}} \in \Hom_\IR(\IC^{g},\IC^{g^\prime})$. For an element $\phi_{\mathrm{ab}} \in \Hom_\IR(A,A^\prime)$, we hence define the $(1,1)$-form $\omega(N;\phi_{\mathrm{ab}})$ on $A$ by demanding
\begin{equation*}
\omega(N;\phi_{\mathrm{ab}})_x: T_{\IR,x} A(\IC) \times T_{\IR,x} A(\IC) = \IC^{g} \times \IC^{g} \longrightarrow \IC, \, (v,w) \longmapsto -\mathrm{Im}(H)(\widetilde{\phi}_{\mathrm{ab}}(v),\widetilde{\phi}_{\mathrm{ab}}(w)),
\end{equation*}
for each $x \in A(\IC)$. Since $c_1(\overline{N})$ is positive and of real type, $\omega(N;\phi_{\mathrm{ab}})$ is semipositive and of real type as well. In addition, $\omega(N;\phi_{\mathrm{ab}})$ only depends on $\phi_{\mathrm{ab}}$ and the hermitian form $H$ associated with $N$ (i.e., the N\'eron-Severi class of $N$) but we have no use for this fact in the following.  Yet again, the assignment $\phi_{\mathrm{ab}} \mapsto \omega(N;\phi_{\mathrm{ab}})$ is continuous with respect to the usual topologies. Finally, there is the obvious homogeneity relation 
\begin{equation} \label{equation::abelianhomogenity}
\omega(N;n\cdot \phi_{\mathrm{ab}})= n^2 \cdot \omega(N;\phi_{\mathrm{ab}}).
\end{equation}

\section{Distributions, analytic subgroups, and Ax's Theorem}
\label{section::distribution}

In general, the $(1,1)$-forms $\omega(\phi_{\mathrm{tor}})$ and $\omega(N;\phi_{\mathrm{ab}})$ introduced in Section \ref{section::chernforms} have no realization as Chern forms of hermitian line bundles. As we show in this section, they nevertheless convey geometric information and are closely connected to the group structure of the semiabelian variety.

Once again, we consider a semiabelian variety $G$ with abelian quotient $\pi: G \rightarrow A$ and toric part $T=\Gm^{t}$. Let $t^\prime$ be a non-negative integer, $A^\prime$ an abelian variety and $N$ an ample line bundle on $A^\prime$. For each $\phi_{\mathrm{tor}} \in \Hom_\IR(\Gm^{t},\Gm^{t^\prime})$ (resp.\ $\phi_{\mathrm{ab}} \in \Hom_\IR(A,A^\prime)$), we have a semipositive $(1,1)$-form $\omega(\phi_{\mathrm{tor}})$ (resp.\ $\pi^\ast \omega(N; \phi_{\mathrm{ab}})$) of real type on $G(\IC)$. Set
\begin{equation*}
\omega= c \cdot \omega(\phi_{\mathrm{tor}}) + \pi^\ast \omega(N; \phi_{\mathrm{ab}})
\end{equation*}
for some arbitrary positive constant $c>0$. (The flexibility provided by $c$ is needed later in the proof of Lemma \ref{lemma::intersection1} in order to remedy the fact that there is no ample line bundle on a general semiabelian variety that is homogeneous with respect to the multiplication-by-$n$ map $[n]$.) Since $g_{\omega(\phi_{\mathrm{tor}})}$ and $g_{\pi^\ast \omega(N; \phi_{\mathrm{ab}})}$ are positive semidefinite, we infer from (\ref{equation::kernelintersection}) that
\begin{equation} \label{equation::kernelsplit}
\ker(\omega_x) = \ker(\omega(\phi_{\mathrm{tor}})_x) \cap \ker(\pi^\ast \omega(N; \phi_{\mathrm{ab}})_x) \subseteq T_{\IR,x} G(\IC)
\end{equation}
for each $x \in G(\IC)$. In addition, $\omega(I (\cdot), I (\cdot)) = \omega(\cdot,\cdot)$ implies that $\ker(\omega(\phi_{\mathrm{tor}})_x)$ is invariant under $I$. In fact, both $\ker(\omega(\phi_{\mathrm{tor}})_x)$ and $\ker(\omega(N; \phi_{\mathrm{ab}})_x)$ are $I$-invariant for the same reason. Under our standing identification of $T_{\IR} G(\IC)$ and $T^{1,0}_\IC G(\IC)$, this means that $\ker(\omega_x)$ is a $\IC$-linear subspace of $T^{1,0}_{\IC,x} G(\IC)$. Our next observation is that this yields a left-invariant holomorphic distribution (i.e., a holomorphic vector subbundle) $\ker(\omega) \subset T^{1,0}_\IC G(\IC)$, which is a straightforward consequence of the lemma below.

\begin{lemma} \label{lemma::translationinvariance} For every $x, y \in G(\IC)$, we have $(dl_y)_x \ker(\omega_x)= \ker(\omega_{y + x})$.
\end{lemma}

\begin{proof}
Because of (\ref{equation::kernelsplit}), it suffices to prove that
\begin{equation} \label{equation::invariance}
(dl_y)_x \ker(\omega(\phi_{\mathrm{tor}})_{x})=\ker(\omega(\phi_{\mathrm{tor}})_{y + x}) \text{ and } (dl_y)_x \ker(\pi^\ast \omega(N; \phi_{\mathrm{ab}})_{x})= \ker(\pi^\ast \omega(N; \phi_{\mathrm{ab}})_{y + x})
\end{equation}
for all $x, y \in G(\IC)$. The latter equality is a direct consequence of the fact that the Chern form $c_1(\overline{N})$ on $A^\prime(\IC)$ is translation-invariant, which can be read off from (\ref{equation::chernM}). Using (\ref{equation::explicitg}), we see that (\ref{equation::cherntoric}) implies
\begin{equation*}
g_{\omega(\phi_{\mathrm{tor}})}=\sum_{v=1}^{t^\prime}\frac{\del \phi_{\mathrm{tor}}^\ast \lambda_v^\prime \otimes \delbar \phi_{\mathrm{tor}}^\ast \lambda_v^\prime + \delbar \phi_{\mathrm{tor}}^\ast \lambda_v^\prime \otimes \del \phi_{\mathrm{tor}}^\ast \lambda_v^\prime}{2\pi(1+e^{2\phi_{\mathrm{tor}}^\ast \lambda^\prime_v})(1+e^{-2\phi_{\mathrm{tor}}^\ast \lambda^\prime_v})}.
\end{equation*}
By Lemma \ref{lemma::positivity}, each $
\del \phi_{\mathrm{tor}}^\ast \lambda_v^\prime \otimes \delbar \phi_{\mathrm{tor}}^\ast \lambda_v^\prime + \delbar \phi_{\mathrm{tor}}^\ast \lambda_v^\prime \otimes \del \phi_{\mathrm{tor}}^\ast \lambda_v^\prime$ is a positive semidefinite bilinear form on $T_x G(\IC)$ and it follows by (\ref{equation::kernelintersection}) that
\begin{equation*}
\ker(\omega(\phi_{\mathrm{tor}}))= \bigcap_{v=1}^{t^\prime} \ker (\del \phi_{\mathrm{tor}}^\ast \lambda_v^\prime \otimes \delbar \phi_{\mathrm{tor}}^\ast \lambda_v^\prime + \delbar \phi_{\mathrm{tor}}^\ast \lambda_v^\prime \otimes \del \phi_{\mathrm{tor}}^\ast \lambda_v^\prime).
\end{equation*}
In addition, (\ref{equation::alternativekernel}) implies that
\begin{equation*}
 \ker ((\del \phi_{\mathrm{tor}}^\ast \lambda_v^\prime \otimes \delbar \phi_{\mathrm{tor}}^\ast \lambda_v^\prime + \delbar \phi_{\mathrm{tor}}^\ast \lambda_v^\prime \otimes \del \phi_{\mathrm{tor}}^\ast \lambda_v^\prime)_x) = \{ w \in T_{\IR,x} G(\IC) \ | \ \del \phi_{\mathrm{tor}}^\ast \lambda_v^\prime(w) \cdot \delbar \phi_{\mathrm{tor}}^\ast \lambda_v^\prime(w) = 0 \}.
\end{equation*}
Since each $\phi_{\mathrm{tor}}^\ast \lambda_v^\prime$ is real-valued, we have $\overline{\del \phi_{\mathrm{tor}}^\ast \lambda_v^\prime (v)} = \delbar \phi_{\mathrm{tor}}^\ast \lambda_v^\prime (v)$ and thus
\begin{equation*}
\ker ((\del \phi_{\mathrm{tor}}^\ast \lambda_v^\prime \otimes \delbar \phi_{\mathrm{tor}}^\ast \lambda_v^\prime + \delbar \phi_{\mathrm{tor}}^\ast \lambda_v^\prime \otimes \del \phi_{\mathrm{tor}}^\ast \lambda_v^\prime)_x) = \{ w \in T_{\IR,x} G_0(\IC) \ | \ \del \phi_{\mathrm{tor}}^\ast \lambda_v^\prime(w) = 0 \} = \ker(\del \phi_{\mathrm{tor}}^\ast \lambda_v^\prime).
\end{equation*}
Note that $(\del \phi_{\mathrm{tor}}^\ast \lambda_v^\prime)_x$ is a $\IR$-linear map $T_{\IR,x}G(\IC) \rightarrow \IC$. 
As each $\lambda_v^\prime$ is a homomorphism, we have $(\phi_{\mathrm{tor}}^\ast \lambda_v^\prime \circ l_y)(\cdot) = \phi_{\mathrm{tor}}^\ast \lambda_v^\prime (\cdot)+ \phi_{\mathrm{tor}}^\ast \lambda_v^\prime (y)$ and therefore $\del (\phi_{\mathrm{tor}}^\ast \lambda_v^\prime \circ l_y) = \del \phi_{\mathrm{tor}}^\ast \lambda_v^\prime$. We infer 
\begin{equation*}
(dl_y) \ker (\del \phi_{\mathrm{tor}}^\ast \lambda_v^\prime)_x = \ker (\del (\phi_{\mathrm{tor}}^\ast \lambda_v^\prime \circ l_y))_{y + x} = \ker (\del \phi_{\mathrm{tor}}^\ast \lambda_v^\prime)_{y + x},
\end{equation*}
thus establishing the first equation of (\ref{equation::invariance}).
\end{proof}

Having proven translation-invariance, we can easily determine the rank of $\ker(\omega)$ by determining the dimension of $\ker(\omega_e)$. We do this next under some surjectivity assumption on $\phi_{\mathrm{tor}}$ and $\phi_{\mathrm{ab}}$. To describe this assumption,  we recall that the complex exponential map gives a universal covering $\IC \rightarrow \Gm(\IC)$. Taking products, we obtain universal coverings $\IC^{t} \rightarrow \Gm^{t}(\IC)$ and $\IC^{t^\prime} \rightarrow \Gm^{t^\prime}(\IC)$. Each homomorphism $\varphi_{\mathrm{tor}} \in \Hom(\Gm^{t},\Gm^{t^\prime})$ lifts to a linear map $\widetilde{\varphi}_{\mathrm{tor}}: \IC^{t} \rightarrow \IC^{t^\prime}$ (cf.\ \cite[Theorems 3.25 and 3.27]{Warner1983}). Tensoring with $\IR$, we obtain an injection $\Hom_\IR(\Gm^{t},\Gm^{t^\prime})\hookrightarrow \Hom_\IR(\IC^{t},\IC^{t^\prime})$, $\phi_{\mathrm{tor}} \mapsto \widetilde{\phi}_\mathrm{tor}$, and set 
\begin{equation*}
\Hom_\IR^{\circ}(\Gm^{t},\Gm^{t^\prime})= \{ \phi_{\mathrm{tor}} \in \Hom_\IR(\Gm^{t},\Gm^{t^\prime}) \ | \ \text{$\widetilde{\phi}_\mathrm{tor}$ is surjective} \}.
\end{equation*}
In Subsection \ref{subsection::abelian}, we have associated with each $\phi_{\mathrm{ab}} \in \Hom(A,A^\prime)$ a linear map $\widetilde{\phi}_\mathrm{ab}: \IC^{g} \rightarrow \IC^{g^\prime}$ and we define similarly
\begin{equation*}
\Hom_\IR^{\circ}(A,A^\prime)= \{ \phi_{\mathrm{ab}} \in \Hom_\IR(A,A^\prime) \ | \ \text{$\widetilde{\phi}_\mathrm{ab}$ is surjective} \}.
\end{equation*}

\begin{lemma} \label{lemma::distribution_rank} If $\phi_{\mathrm{tor}} \in \Hom^\circ_\IR(\Gm^{t},\Gm^{t^\prime})$ and $\phi_{\mathrm{ab}} \in \Hom^\circ_\IR(A,A^\prime)$, then $\ker(\omega)$ has rank $(t - t^\prime)+(\dim(A)-\dim(A^\prime))$ (as a complex vector bundle).
\end{lemma}

\begin{proof} First, we claim that there is a commutative exact diagram
\begin{equation} \label{equation::tangentspaces_exactdiagram}
\begin{tikzcd}
0 \arrow[r] & \ker(\omega(\phi_{\mathrm{tor}})|_{T_{\IR,e} T(\IC)}) \arrow[r] \arrow[d, hook] & \ker(\omega_e) \arrow[r] \arrow[d, hook] & \ker(\omega(N;\phi_{\mathrm{ab}})_e) \arrow[r] \arrow[d, hook] & 0 \\
0 \arrow[r] & T_{\IR,e} T(\IC) \arrow[r] & T_{\IR,e} G(\IC) \arrow[r, "(d\pi)_e"] & T_{\IR,e} A(\IC) \arrow[r] & 0.
\end{tikzcd}
\end{equation}
Except for the surjectivity of $\ker(\omega_e) \rightarrow \ker(\omega(N;\phi_{\mathrm{ab}})_e)$, this is a direct consequence of semipositivity and (\ref{equation::alternativekernel}). For surjectivity, it suffices to prove that there exists an $I$-invariant subspace $V \subset \ker(\omega(\phi_{\mathrm{tor}})_e)$ such that
\begin{equation} \label{equation::tangentspacedecomposition}
T_{\IR,e}T(\IC) \oplus V = T_{\IR,e}G(\IC).
\end{equation}
Given such a decomposition, we can find for any $v\in \ker(\omega(N;\phi_{\mathrm{ab}})_e)$ a $(d\pi)_e$-preimage $w\in V$. Furthermore, we have
\begin{equation*}
g_{\omega,e}(w,w)
= c \cdot g_{\omega(\phi_{\mathrm{tor}}),e}(w,w)+g_{\pi^\ast \omega(N;\phi_{\mathrm{ab}}),e}(w,w) 
=g_{\omega(N;\phi_{\mathrm{ab}}),e}(v,v) = 0
\end{equation*} 
since $w \in \ker(\omega(\phi_{\mathrm{tor}})_e)$ and $(d\pi)_e(w) = v\in \ker(\omega(N; \phi_{\mathrm{ab}})_e)$. Recall that the maximal compact subgroup $K_{G} \subset G(\IC)$ is a real Lie subgroup such that $\dim_\IR(T_{\IR,e} K_{G}) = 2\dim(A)+t$. 

We now claim that $V = T_{\IR,e} K_{G} \cap I (T_{\IR,e} K_{G}) \subset T_{\IR,e} G(\IC)$ is a suitable choice for (\ref{equation::tangentspacedecomposition}). Since each Weil function $\lambda_u$ ($1\leq u \leq t$) is constant zero on $K_{G}$, both the $(1,0)$-forms $\del \lambda_u$ ($1\leq u \leq t$) and the $(0,1)$-forms $\delbar \lambda_u$ ($1\leq u \leq t$) have to vanish on $V$. This immediately implies that $V \subset \ker(\omega(\phi_{\mathrm{tor}})_e)$. We already know that $\lambda_u|_{T(\IC)} = - \log |z_u|$ in standard coordinates $z_1,\dots, z_{t}$ on $T(\IC)= \Gm^{t}(\IC)$. We compute that
\begin{equation} \label{equation::dellambda}
\del \lambda_u|_{T_{\IR}T(\IC)} = -dz_u/2z_u \text{ and } 
\delbar \lambda_u|_{T_{\IR}T(\IC)}= -d\overline{z}_u/2\overline{z}_u \text{ ($1\leq u \leq t$),}
\end{equation}
which shows that the restrictions of $\del \lambda_1,\dots, \del \lambda_{t}, \delbar \lambda_1^\prime, \dots, \delbar \lambda_{t}^\prime$ to $T_{\IR,e}T(\IC)$ form a $\IC$-basis of $\Hom_{\IR}(T_{\IR,e}T(\IC),\IC)$. Since each of these forms vanishes on $V$ (see (\ref{equation::maximalcompactsubgroup})), we have $T_{\IR,e} T(\IC) \cap V = \{ 0 \}$. As $\dim_\IR(V) = \dim_\IR(T_{\IR,e} K_{G} \cap I (T_{\IR,e} K_{G})) \geq 2 \dim (A)$, we obtain the direct sum decomposition (\ref{equation::tangentspacedecomposition}).

Using (\ref{equation::tangentspaces_exactdiagram}), it remains to compute the dimensions of the $I$-invariant $\IR$-linear subspaces $\ker(\omega(\phi_{\mathrm{tor}})|_{T_{\IR,e} T(\IC)})$ and $\ker(\omega(M;\phi_{\mathrm{ab}})_e)$. For the former one, let us represent $\widetilde{\phi}_\mathrm{tor} \in \Hom_\IR(\IC^{t},\IC^{t^\prime})$ as a matrix $(a_{uv})_{1 \leq u \leq t, 1 \leq v \leq t^\prime} \in \IR^{t \times t^\prime}$. As in the proof of Lemma \ref{lemma::translationinvariance} above, we have
\begin{align*}
\ker(\omega(\phi_{\mathrm{tor}})|_{T_{\IR,e}T(\IC)})
&= \bigcap_{v=1}^{t^\prime}{\ker ((a_{1v}\del\lambda_1+\dots + a_{tv}\del\lambda_{t})|_{T_{\IR,e}T(\IC)}}) \\
&= \bigcap_{v=1}^{t^\prime}{\ker (a_{1v}dz_1 + \dots + a_{tv} dz_{t}: T_{\IR,e}T(\IC) \rightarrow \IC)}.
\end{align*}
Setting $dz_i = dx_i + i dy_i$ with $dx_i, dy_i: T_{\IR,e}T(\IC) \rightarrow \IR$, we can rewrite this as
\begin{equation*}
\ker(\omega(\phi_{\mathrm{tor}})|_{T_{\IR,e}T(\IC)}) = \bigcap_{v=1}^{t^\prime} \ker (a_{1v}dx_1 + \dots + a_{tv} dx_{t})\cap \bigcap_{i=1}^{t^\prime} \ker (a_{1v}dy_1 + \dots + a_{tv} dy_{t}).
\end{equation*}
The condition $\phi_{\mathrm{tor}} \in \Hom_\IR^\circ(\Gm^{t},\Gm^{t^\prime})$ is equivalent to the matrix $(a_{uv})$ having maximal rank $t^\prime$. This implies that the $2t^\prime$ real-valued functionals 
\begin{equation*}
a_{1v}dx_1 + \cdots + a_{tv} dx_{t}, a_{1v}dy_1 + \cdots + a_{tv}  dy_{t} \ (1 \leq v \leq t^\prime)
\end{equation*}
on $T_{\IR,e}T(\IC)$ are $\IR$-linearly independent. From this, we infer
\begin{equation*}
\dim_{\IR} \ker(\omega(\phi_{\mathrm{tor}})|_{T_{\IR,e} T(\IC)}) = 2(t - t^\prime).
\end{equation*}
For $\ker(\omega(N;\phi_{\mathrm{ab}})_e)$, it follows directly from (\ref{equation::chernformpullback2}) that
\begin{equation*}
g_{\omega(N;\phi_{\mathrm{ab}}),e}\!: T_{\IR,x} A(\IC) \times T_{\IR,x} A(\IC) = \IC^{g} \times \IC^{g} \longrightarrow \IC, \, (v,w) \longmapsto \mathrm{Re}(H)(\widetilde{\phi}_\mathrm{ab}(v),\widetilde{\phi}_\mathrm{ab}(w))
\end{equation*}
(compare with (\ref{equation::hermitianform})). Since $H$ is positive definite, so is its real part $\mathrm{Re}(H)$. Using (\ref{equation::alternativekernel}), we deduce that
\begin{equation*}
\ker(\omega(N;\phi_{\mathrm{ab}})_e)= \widetilde{\phi}_\mathrm{ab}^{-1}(\ker(\mathrm{Re}(H)))=\widetilde{\phi}_\mathrm{ab}^{-1}(\{ 0 \})= \ker(\widetilde{\phi}_\mathrm{ab}).
\end{equation*}
Finally, $\phi_{\mathrm{ab}} \in \Hom_\IR^\circ(A,A^\prime)$ implies that $\dim_{\IR} \ker(\omega(N;\phi_{\mathrm{ab}})_e) = 2(\dim(A)- \dim(A^\prime))$.
\end{proof}

In summary, we have proven that $\ker(\omega) \subset T_\IC^{1,0}G(\IC)$ is a left-invariant holomorphic distribution on $G(\IC)$. By the holomorphic Frobenius theorem (\cite[Theorem 2.26]{Voisin2007}), the distribution $\ker(\omega)$ is (holomorphically) integrable since the Lie bracket on $T_\IC^{1,0} G(\IC)$ vanishes. In fact, the integral manifold of $\ker(\omega)$ through a given point $x \in G(\IC)$ coincides with the analytic subgroup $x\cdot \exp_{G(\IC)}(\ker(\omega)_e)$ with $\exp_{G(\IC)}: T_\IC^{1,0}G(\IC) \rightarrow G(\IC)$ being the Lie group exponential (cf.\ \cite[Theorem II.1.7]{Helgason2001}).

As indicated in Section \ref{section::hermitiandifferentialgeometry}, we are interested in determining when a submanifold $Y \subset G(\IC)$ and a point $x \in Y(\IC)$ are such that $g_\omega|_{T_{\IR,x} Y}$ is positive definite. For this purpose, we introduce an elementary lemma about integrable (holomorphic) distributions.

\begin{lemma} \label{lemma::distributions} Let $M$ be a complex manifold of dimension $n$, $\mathcal{D} \subset T^{1,0}_{\IC} M$ an integrable holomorphic distribution of rank $m$ on $M$, $Z \subset M$ a $k$-dimensional analytic subvariety and $x$ a point on $Z$. Assume that there exists an open neighborhood $U \subset M$ of $x$ such that $\dim_\IC(\mathcal{D} \cap T_{\IC,y}^{1,0} Z) \geq l$ for any $y \in Z^{\mathrm{sm}} \cap U$. 
Then, the integral submanifold $L \subset U$ of $\mathcal{D}$ through $x$ satisfies $\dim_{x}(L \cap Z)\geq l$.
\end{lemma}


\begin{proof} By shrinking $U$ if necessary, we can assume that there exists a holomorphic flat chart $f: U \rightarrow \IC^{n-m}$ for $\mathcal{D}|_U$. Recall that this means that $f$ is a submersion and that each non-empty fiber of $f$ is an integral submanifold for $\mathcal{D}|_U$. By our assumption, the differential $d(f|_Z): T_{\IC,y}^{1,0}Z \rightarrow \IC^{n-m}$ has rank $\leq k-l$ for every $y \in Z^{\mathrm{sm}} \cap U$. By \cite[Lemma L.6]{Gunning1990a}, the local dimension of any fiber $f|_Z^{-1}(f(y)) = f^{-1}(f(y)) \cap Z$, $y \in Z^{\mathrm{sm}} \cap U$, is $\geq l$ everywhere. If $x$ is a smooth point of $Z$, this already implies $\dim_x(f^{-1}(f(x)) \cap Z) \geq l$. For $x$ in the singular locus, we use also the upper semi-continuity of the fiber dimension \cite[Lemma L.2]{Gunning1990a} to conclude the proof.
\end{proof}

Finally, we are ready to use Ax's Theorem to show non-degeneracy in all cases of interest.

\begin{lemma} \label{lemma::ax} Let $X \subset G$ be an algebraic subvariety such that $X^{(s)} \neq X$ for some non-negative integer $s$. Then $(\omega|_X)^{\wedge \dim(X)} \neq 0$ for every $(\phi_{\mathrm{tor}},\phi_{\mathrm{ab}}) \in \Hom^\circ_\IR(\Gm^{t},\Gm^{t^\prime}) \times \Hom^\circ_\IR(A,A^\prime)$ with $t^\prime + \dim(A^\prime) \geq s$.
\end{lemma}

\begin{proof} Assume $(\omega|_X)^{\wedge \dim(X)} = 0$, which means that $\dim_\IC(\ker(\omega_x)\cap T_x X) \geq 1$ for any $x\in X(\IC)$. For each $x \in X(\IC)$, let $L_x=x \cdot \exp_{G(\IC)}(\ker(\omega|_e))$ be the integral manifold of $\ker(\omega)$ through $x$. By Lemma \ref{lemma::distribution_rank}, the holomorphic distribution $\ker(\omega)$ has rank $\leq \dim(G)-s$ and this is also the dimension of $L_x$. From Lemma \ref{lemma::distributions}, we know that $\dim_x(L_x \cap X(\IC))\geq 1$. This is an intersection of an algebraic subvariety with an analytic subgroup in $G(\IC)$. Applying Ax's Theorem (\cite[Corollary 1]{Ax1972}), we obtain for each $x \in X(\IC)$ an algebraic subgroup $H \subset G$ such that $X \subset x H$ and
\begin{equation*}
\dim(H) \leq \dim(X) + \dim(L_x) - \dim_x(L_x \cap X(\IC)) <  \dim(X) + \dim(G) - s.
\end{equation*}
A comparison with (\ref{equality::codim2}) shows that this implies that $X$ is itself an $s$-anomalous variety, associated with $H$, and hence $X=X^{(s)}$.
\end{proof}

\section{Proof of Theorem \ref{theorem::main}}
\label{section::proof}

In this section, all algebraic groups are over $\mathrm{Spec}(\IQbar)$ without further mention. As usual, $T$ denotes the toric part of $G$ and $A$ the underlying abelian variety. Since our base field is $\IQbar$, the torus $T$ is split and we keep fixed a splitting throughout this section (i.e., assume $T = \Gm^{t}$). 

\subsection{Reductions}
\label{subsection::reductions}
We start with an elementary observation related to the ``height cones'' introduced in (\ref{equation::heightcones}). Let $h,h^\prime: G(\IQbar) \rightarrow \IR$ be functions satisfying
\begin{equation} \label{equation::heightinequalities}
c_3 h^\prime(x) - c_4 \leq h(x) \leq c_5 h^\prime(x) + c_6 \text{ and } h^\prime(x_1+ x_2) \leq c_7(h^\prime(x_1) + h^\prime(x_2))+c_8
\end{equation}
for all $x,x_1,x_2 \in G(\IQbar)$ with constants $c_i>0$ ($i \in \{3,\dots ,8\}$). For any subset $\Sigma \subseteq G(\IQbar)$ and any $\varepsilon>0$, there is an inclusion
\begin{equation*}
C(\Sigma,h^\prime,\varepsilon^\prime) \subseteq C(\Sigma,h,\varepsilon) \cup \{ x\in G(\IQbar) \ | \ h^\prime(x)< c_9 \}
\end{equation*}
with $\varepsilon^\prime = \varepsilon c_3c_5^{-1}/2$ and some constant $c_9=c_9(c_3,\dots,c_8,\varepsilon,\varepsilon^\prime)$. We leave this straightforward computation to the reader.

Let $\overline{G}$ be the compactification of $G$ and $M_{\overline{G}}$ the line bundle as in Construction \ref{construction1}. Furthermore, let $N$ be a ample symmetric line bundle on $A$. By Lemma \ref{lemma::ampleness}, $L = M_{\overline{G}} \otimes \overline{\pi}^\ast N$ is an ample line bundle on $\overline{G}$. For Theorem \ref{theorem::main}, it is sufficient to prove the boundedness of $h_{L}$ on $(X \setminus X^{(s)})(\overline{\IQ}) \cap C(G^{[s]}(\IQbar),h_{L},\varepsilon)$. In fact, let $L^\prime$ be an arbitrary ample line bundle on an arbitrary compactification $\overline{G}^\prime$ of $G$ and $h_{L^\prime}$ an associated Weil height. Applying \cite[Proposition 2.3]{Vojta1999} to the identity map $\id_G$, which gives a birational map $\overline{G} \dashrightarrow \overline{G}^\prime$, and the line bundles $L$ and $L^\prime$ we obtain the first two inequalities in (\ref{equation::heightinequalities}). The third inequality follows from applying the same proposition to the group law $+_G$, understood as a rational map $\overline{G}^\prime \times \overline{G}^\prime \dashrightarrow \overline{G}^\prime$. We may thus use our above observation to ensure the asserted reduction. Considering also Lemma \ref{lemma::canonicalheight} (a), we see that it even suffices to prove that $\widehat{h}_{L} = \widehat{h}_{M_{\overline{G}}}+\widehat{h}_{N}: G(\IQbar) \rightarrow \IR^{\geq 0}$ is bounded from above on $(X \setminus X^{(s)})(\overline{\IQ}) \cap C(G^{[s]}(\IQbar),\widehat{h}_{L},\varepsilon)$.

Our last reduction step is to note that Theorem \ref{theorem::main} is easily inferred from the following proposition, which is shown in the remaining parts of Section \ref{section::proof}.

\begin{proposition} \label{proposition::induction_step} Let $X \subseteq G$ be an irreducible Zariski closed subset of positive dimension such that $X^{(s)} \neq X$. Then, there exists a non-empty Zariski open subset $U \subseteq X$ and some $\varepsilon > 0$ such that $\widehat{h}_{L}$ is bounded on $U \cap C(G^{[s]}(\IQbar),\widehat{h}_{L},\varepsilon)$.
\end{proposition}

\begin{proof}[Proof of Theorem \ref{theorem::main} (using Proposition \ref{proposition::induction_step})] We perform an induction on $\dim(X)$. Theorem \ref{theorem::main} is clearly trivial if $X$ has dimension zero, which starts our induction. Assume now that $X$ is positive dimensional and that the assertion of the theorem, with $h_L$ replaced by $\hhat_L$, is already known for any $X^\prime$ with $\dim(X^\prime)<\dim(X)$. Without loss of generality, we can additionally assume that $X$ is irreducible and that $X^{(s)} \neq X$. Applying Proposition \ref{proposition::induction_step} to $X$, we obtain a non-empty Zariski open subset $U \subseteq X$ and a real number $\varepsilon>0$ such that $\widehat{h}_{L}$ is bounded on $U(\IQbar) \cap C(G^{[s]}(\IQbar),\widehat{h}_{L},\varepsilon)$. Now, $X^\prime = X \setminus U$ has dimension strictly less than $\dim(X)$ so that we may apply our inductive hypothesis to $X^\prime$. We obtain that $\widehat{h}_L$ is bounded on $(X^\prime \setminus (X^{\prime})^{(s)})(\IQbar) \cap C(G^{[s]}(\IQbar),\hhat_L,\varepsilon^\prime)$ for some $\varepsilon^\prime>0$. In conclusion, we know that $\widehat{h}_L$ is bounded on
\begin{equation*}
(X \setminus (X^{\prime})^{(s)})(\IQbar) \cap C(G^{[s]}(\IQbar),\hhat_L,\min\{ \varepsilon,\varepsilon^\prime \})
\end{equation*}
As $(X^\prime)^{(s)} \subseteq X^{(s)}$ by (\ref{equality::codim2}), this yields the assertion of Theorem \ref{theorem::main} for $X$.
\end{proof} 

\subsection{Approximating homomorphisms} \label{section::approximatinghomomorphisms}

The following lemma is useful for reducing the proof of the main theorem to a manageable situation. 

\begin{lemma} \label{lemma::homomorphisms} There exist finitely many abelian varieties $A_1^\prime,\dots, A_{j_0}^\prime$ (depending on $s$) such that each $x \in G^{[s]}(\IQbar)$ is contained in the kernel of some surjective homomorphism $\varphi\!:G \rightarrow G^\prime$, $\dim(G^\prime)\geq s$, that is represented (as in Lemma \ref{lemma::semiabelian1}) by a diagram
\begin{equation*}
\xymatrix{0 \ar[r] & \Gm^{t} \ar[r] \ar[d]^{{\varphi}_{\mathrm{tor}}} & G \ar[r] \ar[d]^{{\varphi}} & A \ar[r] \ar[d]^{{\varphi}_{\mathrm{ab}}} & 0  \\
		  0 \ar[r] & \mathbb{G}_m^{t^\prime} \ar[r] & G^\prime \ar[r] & A_j^\prime \ar[r] & 0.
}
\end{equation*}
\end{lemma}

As $\varphi_{\mathrm{tor}}$ is surjective, we clearly have $t^\prime \leq t$.

\begin{proof} Evidently, if $x \in H(\IQbar)$ with $\codim_G(H) \geq s$ then $x$ is in the kernel of the quotient $\pi\!: G \rightarrow G/H$. The toric part of $G/H$ can be identified with $\Gm^{t^\prime}$. The abelian component $\pi_{\mathrm{ab}}\!: A \rightarrow B$ of $\pi$ is surjective. By Poincaré's complete reducibility theorem (\cite[Theorem 1 on p.\ 173]{Mumford1970}), there exist only finitely many quotients $A \rightarrow A_j^\prime$, $1 \leq j \leq j_0$, up to isogeny. In particular, there exists an isogeny $\psi: B \rightarrow A_j^\prime$ for some $j^\prime \in \{1, \dots, j_0\}$. By Lemma \ref{lemma::semiabelian3}, there exists a semiabelian variety $G^\prime$ and a unique homomorphism $\psi^\prime: G/H \rightarrow G^\prime$ with toric part $\mathrm{id}_{\Gm^{t^\prime}}$ and abelian part $\psi$. We can take $\varphi = \psi^\prime \circ \pi$.
\end{proof}
If $G$ is an abelian variety, Poincaré's complete reducibility theorem yields immediately the existence of finitely many quotients $\varphi_i: G \rightarrow G_i$, $\dim(G_i)\geq s$, such that each $x \in G^{[s]}(\IQbar)$ is contained in the kernel of some $\varphi_i$. In addition, if $G$ is a torus a similar statement is true for more trivial reasons. Nevertheless, the analogous statement is false for general semiabelian varieties as simple examples show.\footnote{\label{footnote}In fact, consider the semiabelian variety $G$ that is the $\Gm^2$-extension of a non-CM elliptic curve $E$ represented by $(\eta_1,\eta_2)\in E^\vee(\overline{\IQ})^2 $. Assume also that $\IZ \eta_1 + \IZ \eta_2$ is a free $\IZ$-module of rank $2$. For each integer $n$, we consider the $\Gm$-extension $G^{(n)}$ of $E$ given by $n\eta_1 + \eta_2 \in E^\vee(\IQbar)$ and the homomorphism $\varphi^{(n)}\!: G \rightarrow G^{(n)}$ described by $(\varphi_{\mathrm{tor}}^{(n)})^\ast(Y_1)= X_1^n X_2$ and $\varphi_{\mathrm{ab}}^{(n)} = \id_E$. 
There exists a point $x \in \ker(\varphi^{(n)})(\IQbar) \subset G^{[2]}(\IQbar)$ that is not contained in any other algebraic subgroup of codimension $2$. Therefore any surjective homomorphism $\varphi\!: G \rightarrow G^\prime$, $\dim(G^\prime) = 2$, with $p \in \ker(\varphi)(\IQbar)$ factors through $\varphi^{(n)}$. However, Lemma \ref{lemma::semiabelian1} implies that $G^{(n)}$ and $G^{(m)}$ are not isogeneous if $n\neq m$. Indeed, all $\Gm$-extensions isogeneous to $G^{(n)}$ are represented by ``rational multiples'' of $n\eta_1 + \eta_2 \in E^\vee(\IQbar)$.} Our lemma is optimal in the general case.

By Lemma \ref{lemma::semiabelian1}, we may associate with each $\varphi \in \Hom(G, G^\prime)$ as in Lemma \ref{lemma::homomorphisms} a pair
\begin{equation*}
(\varphi_{\mathrm{tor}}, \varphi_{\mathrm{ab}})\in \Hom(\Gm^{t},\Gm^{t^\prime}) \times \Hom(A,A_j^\prime), \ t^\prime \in \{ 0, \dots, t\}, \ j \in \{ 1, \dots, j_0\}.
\end{equation*}
This allows us to concentrate on a finite number of fixed finite rank $\IZ$-modules
\begin{equation} \label{equation::Vij}
V^{(t^\prime,j)} = \Hom(\Gm^{t},\mathbb{G}_m^{t^\prime}) \times \Hom(A,A_j^\prime), \ t^\prime \in \{ 0, \dots, t\}, \ j \in \{ 1, \dots, j_0\}.
\end{equation}
instead of infinitely many different $\Hom(G,G^\prime)$. We study now one of these modules separately and drop the superscripts, writing $V$ instead of $V^{(t^\prime,j)}$. As $V$ is a free $\IZ$-module, it embeds into $V_\IQ = V \otimes_\IZ \IQ$ and $V_\IR = V \otimes_\IZ \IR$. Furthermore, a quasi-homomorphism $\phi \in \Hom_\IQ(G,G^\prime)$ determines a pair $(\phi_{\mathrm{tor}},\phi_{\mathrm{ab}}) \in V_\IQ$. However, the relation between elements $(\phi_{\mathrm{tor}}, \phi_{\mathrm{ab}})\in V_\IQ$ and actual quasi-homomorphisms $\phi: G \rightarrow_{\IQ} G^\prime$ of semiabelian varieties is quite intricate. The reader is referred to Section \ref{section::structuralproperties} for details. As witnessed by the results of Section \ref{section::distribution}, we have a special interest in pairs that are contained in
\begin{equation*}
V_\IR^\circ = \Hom_\IR^{\circ}(\Gm^{t},\Gm^{t^\prime}) \times \Hom_\IR^{\circ}(A,A_j^\prime) \subset V_\IR.
\end{equation*}
For this reason, we also define $V_\IQ^\circ = V_\IQ \cap V_\IR^\circ$. It is easy to see that a quasi-homomorphism $\phi_{\mathrm{tor}} \in \Hom_\IQ(\Gm^{t},\Gm^{t^\prime})$ (resp.\ $\phi_{\mathrm{ab}} \in \Hom_\IQ(A,A_j^\prime)$) is contained in $\Hom_\IR^\circ(\Gm^{t},\Gm^{t^\prime})$ (resp.\ $\Hom_\IR^\circ(A,A_j^\prime)$) if and only if it is surjective in the sense of Section \ref{section::homomorphisms}.

With these preparations, we can state our first approximation result. The proof is a simple reduction to the abelian and toric cases treated in \cite{Habegger2009a, Habegger2009}.

\begin{lemma} \label{lemma::approximation} There exists a compact subset $\mathcal{K}=\mathcal{K}_{\mathrm{tor}} \times \mathcal{K}_{\mathrm{ab}} \subset V_\IR^\circ$ such that the following assertion is true: Let $x \in G(\IQbar)$ be contained in the kernel of a surjective homomorphism $\varphi:G \rightarrow G^\prime$ of semiabelian varieties that is represented by some $(\varphi_{\mathrm{tor}}, \varphi_{\mathrm{ab}}) \in V$. Then, there exists a semiabelian variety $G^{\dprime}$ and a surjective quasi-homomorphism $\phi: G \rightarrow_{\IQ} G^\dprime$ such that $x \in \ker(\phi)+\Tors(G)$ and $\phi$ is represented by some $(\phi_{\mathrm{tor}},\phi_{\mathrm{ab}}) \in V_\IQ \cap \mathcal{K}$.
\end{lemma}

The reader may be reminded that $\dim(G^\prime)=\dim(G^\dprime)$ as well as the fact that $\Gm^{t^\prime}$ (resp.\ $A_j^\prime$) is the toric part (resp.\ the abelian quotient) of both $G^\prime$ and $G^\dprime$ is automatic.

\begin{proof}

Using again Lemma \ref{lemma::semiabelian1}, we obtain a commutative diagram
\begin{equation*}
\xymatrix{0 \ar[r] & \Gm^{t} \ar[r] \ar^{\varphi_{\mathrm{tor}}}[d] & G \ar[r] \ar[d]^{\varphi} & A \ar[r] \ar[d]^{\varphi_{\mathrm{ab}}} & 0  \\
		  0 \ar[r] & \Gm^{t^\prime} \ar[r] & G^\prime \ar[r] & A_j^\prime \ar[r] & 0\text{.}
}
\end{equation*}
By \cite[Lemma 2]{Habegger2009a}, there exists some compact subset $\mathcal{K}_{\mathrm{ab}} \subset \Hom_\IR^\circ(A,A_j^\prime)$ such that for every surjective $\psi \in \Hom_\IQ(A,A_j^\prime)$ there exists a surjective $\psi^\prime \in \Hom_\IQ(A_j^\prime,A_j^\prime)$ with $\psi^\prime \circ \psi \in \mathcal{K}_{\mathrm{ab}}$. As $\varphi$ is surjective, the same is true for its abelian component $\varphi_{\mathrm{ab}}$. Hence, we may apply the lemma with $\psi=\varphi_{\mathrm{ab}}$ and obtain a quasi-homomorphism $\psi_{\mathrm{ab}}^\prime: A_j^\prime \rightarrow_\IQ A_j^\prime$ such that $\psi_{\mathrm{ab}}^\prime \circ \varphi_{\mathrm{ab}}\in \mathcal{K}_{\mathrm{ab}}$. Similarly, we can extract from the proof of \cite[Lemma 4.2]{Habegger2009} that there exists a compact set $\mathcal{K}_{\mathrm{tor}} \subset \Hom_\IR^\circ(\Gm^{t},\Gm^{t^\prime})$ such that there always exists a surjective quasi-homomorphism $\psi_{\mathrm{tor}}^\prime: \Gm^{t^\prime} \rightarrow_\IQ \Gm^{t^\prime}$ with $\psi_{\mathrm{tor}}^\prime \circ \varphi_{\mathrm{tor}} \in \mathcal{K}_{\mathrm{tor}}$. We claim that $\mathcal{K} = \mathcal{K}_{\mathrm{tor}} \times \mathcal{K}_{\mathrm{ab}} \subset V_\IR^\circ$ satisfies the assertion of the lemma.

Let $n$ be a positive integer such that $n \cdot \psi_{\mathrm{ab}}^\prime \in \Hom(A_j^\prime,A_j^\prime)$ and $n \cdot \psi_{\mathrm{tor}}^\prime \in \Hom(\Gm^{t^\prime},\Gm^{t^\prime})$. 
By Lemma \ref{lemma::semiabelian3}, there exists a semiabelian variety $G^\dprime$ and a homomorphism $\varphi^\prime: G^\prime \rightarrow G^\dprime$ such that
\begin{equation*}
\xymatrix{0 \ar[r] & \Gm^{t^\prime} \ar[r] \ar^{n \cdot \psi_{\mathrm{tor}}^\prime}[d] & G^\prime \ar[r] \ar[d]^{\varphi^\prime} & A_j^\prime \ar[r] \ar[d]^{n \cdot \psi_{\mathrm{ab}}^\prime} & 0  \\
		  0 \ar[r] & \Gm^{t^\prime} \ar[r] & G^\dprime \ar[r] & A_j^\prime \ar[r]& 0
}
\end{equation*}
is a commutative diagram with exact rows. The homomorphism $\varphi^\prime \circ \varphi: G \rightarrow G^\dprime$ is represented by
\begin{equation*}
n \cdot (\psi_{\mathrm{tor}}^\prime \circ \varphi_{\mathrm{tor}}, \psi_{\mathrm{ab}}^\prime \circ \varphi_{\mathrm{ab}}) \in V \cap n \cdot \mathcal{K}.
\end{equation*}
Multiplying with $n^{-1}$, we get a quasi-homomorphism $\phi: G \rightarrow_\IQ G^\dprime$ that is represented by
\begin{equation*}
(\psi_{\mathrm{tor}}^\prime \circ \varphi_{\mathrm{tor}}, \psi_{\mathrm{ab}}^\prime \circ \varphi_{\mathrm{ab}}) \in V_\IQ \cap \mathcal{K}.
\end{equation*}
This is evidently the quasi-homomorphism we are searching for.
\end{proof}

For the next lemma, we endow $\Hom_\IR(\Gm^{t},\Gm^{t^\prime})$ and $\Hom_\IR(A,A_j^\prime)$ with linear norms. As all norms on a finite-dimensional $\IR$-vector space are equivalent, the precise choice is irrelevant for our purposes. Therefore, we just fix an arbitrary norm $\vert \cdot \vert$ on $\Hom_\IR(\Gm^{t},\Gm^{t^\prime})$ and $\Hom_\IR(A,A_j^\prime)$ for the sequel. We slightly abuse notation in denoting both norms by $|\cdot|$. For each real $r>0$, we denote by $B_r(\phi_{\mathrm{tor}})$ (resp.\ $B_{r^{1/2}}(\phi_{\mathrm{ab}})$) the open ball with radius $r$ (resp.\ $r^{1/2}$) around $\phi_{\mathrm{tor}} \in \Hom_\IR(\Gm^{t},\Gm^{t^\prime})$ (resp.\ $\phi_{\mathrm{ab}} \in \Hom_\IR(A,A_j^\prime)$). In addition, we set
\begin{equation*}
B_r(\phi_{\mathrm{tor}},\phi_{\mathrm{ab}}) =  B_r(\phi_{\mathrm{tor}}) \times B_{r^{1/2}}(\phi_{\mathrm{ab}}), \, (\phi_{\mathrm{tor}},\phi_{\mathrm{ab}})\in V_\IR.
\end{equation*} 

\begin{lemma} \label{lemma::finite_approximation} Let $\delta>0$ be arbitrary. Then, there exists an integer $n_\delta \geq 1$ and a finite set 
\begin{equation*}
\{ (\phi_{1,\mathrm{tor}},\phi_{1,\mathrm{ab}}),\dots, (\phi_{k_\delta,\mathrm{tor}},\phi_{k_\delta,\mathrm{ab}}) \} \subset n_\delta^{-1} V
\end{equation*}
such that for each $(\phi_{\mathrm{tor}},\phi_{\mathrm{ab}}) \in\mathcal{K}$ we have $(\phi_{\mathrm{tor}},\phi_{\mathrm{ab}}) \in B_\delta(\phi_{k,\mathrm{tor}},\phi_{k,\mathrm{ab}})$ for some $1 \leq k \leq k_\delta$.
\end{lemma}

\begin{proof} For sufficiently large $n_\delta$, the open sets
\begin{equation*}
B_\delta(\phi_{\mathrm{tor}},\phi_{\mathrm{ab}}) =  B_\delta(\phi_{\mathrm{tor}}) \times B_{\delta^{1/2}}(\phi_{\mathrm{ab}}), (\phi_{\mathrm{tor}},\phi_{\mathrm{ab}}) \in n_\delta^{-1}V,
\end{equation*}
cover all of $V_\IR$. By compactness, finitely many of these open sets suffice to cover all of $\mathcal{K}$.
\end{proof}

In both \cite{Habegger2009a} and \cite{Habegger2009}, a step analogous to Lemma \ref{lemma::finite_approximation} is performed quite explicitly with a quantitatively much better result, using diophantine approximation. The above weaker estimate is however sufficient for our proof.

\subsection{Height bounds}
\label{section::heightbound}

In this section, we derive two competing height bounds. The first one (Lemma \ref{lemma::height_approximation}) is valid for any $x \in C(G^{[s]},\widehat{h}_L,\varepsilon)$, whereas the second one (\ref{equation::secondheightinequality}) is valid for almost all $x \in (X \setminus X^{(s)})(\overline{\IQ})$. In combination, they imply the desired Proposition \ref{proposition::induction_step}.

Throughout this section, we keep fixed some sufficiently small $\delta$; the precise conditions on $\delta$ can be found in (\ref{equation::delta2}) and (\ref{equation::delta}). \textit{For the constants to be introduced in the sequel, we have to distinguish between those depending only on $G$ and $X$ and those that depend additionally on $\delta$. For this purpose, the former are written plainly $c_i$ whereas the latter are written $c_i(\delta)$. None of these constants depends on the point $x \in G(\IQbar)$ under consideration.}

We now consider a point $x \in (X \setminus X^{(s)})(\IQbar) \cap C(G^{[s]}(\IQbar),\widehat{h}_L,\varepsilon)$. Write $x=y+z$ with $y \in G^{[s]}(\IQbar)$ and $\widehat{h}_L(z)\leq \varepsilon \max \{ 1, \widehat{h}_L(y)\}$. Assuming $\varepsilon< 1/4$, we obtain
\begin{equation*}
\hhat_L(y)=\hhat_L(x-z)\leq 2 \hhat_L(x)+2\hhat_L(z) \leq 2 \hhat_L(x)+\frac{\hhat_L(y)}{2}+\frac{1}{2}
\end{equation*}
by using Lemma \ref{lemma::canonicalheight} (b) and Lemma \ref{lemma::heightsgrouplaw} for the first inequality. Hence, we have that 
\begin{equation} \label{equation::coneinequalities}
\hhat_L(y)\leq 4\hhat_L(x)+1 \text{ and } \hhat_L(z)\leq \varepsilon(4\hhat_L(x)+2).
\end{equation}

Denote by $A_1,\dots,A_{j_0}^\prime$ the abelian varieties afforded by Lemma \ref{lemma::homomorphisms}. We endow each $A_j^\prime$ with an ample symmetric line bundle $N_j$, $1 \leq j \leq j_0$. There exists a semiabelian variety $G^\prime$ with abelian quotient $A_{j}^\prime$ ($j \in \{ 1, \dots, j_0 \}$) and toric part $\Gm^{t^\prime}$ ($t^\prime \in \{ 0, \dots, t \}$) such that there is a surjective homomorphism $\varphi: G \rightarrow G^\prime$ satisfying $y \in \ker(\varphi)$. We emphasize that it is essential that there are only finitely many choices for $j^\prime$ and $t^\prime$ as $x$ varies; otherwise, we would not be able to choose all constants below independent of the point $x$.

Consider the $\IZ$-module $V=V^{(t^\prime,j)}$ defined in (\ref{equation::Vij}) and choose linear norms on $\Hom_\IR(\Gm^{t},\Gm^{t^\prime})$ and $\Hom_\IR(A, A_j^\prime)$, which we simply denote both by $|\cdot|$. Lemma \ref{lemma::approximation} yields a compact set 
\begin{equation*}
\mathcal{K} = \mathcal{K}_{\mathrm{tor}} \times \mathcal{K}_{\mathrm{ab}} \subset V_\IR^\circ = \Hom_\IR^{\circ}(\Gm^{t},\Gm^{t^\prime}) \times \Hom_\IR^{\circ}(A,A_j^\prime) \subset V_\IR
\end{equation*}
and a quasi-homomorphism $\phi_0: G \rightarrow_{\IQ} G_0$ represented by some $(\phi_{0,\mathrm{tor}}, \phi_{0,\mathrm{ab}}) \in \mathcal{K}$ such that $y \in \ker(\phi_0)+\Tors(G)$. We compactify $G_0$ by $\overline{G}_0$ as in Construction \ref{construction1} and endow $\overline{G}_0$ with the ample line bundle (cf.\ Lemma \ref{lemma::ampleness})
\begin{equation*}
L_0 = M_{\overline{G}_0} \otimes (\overline{\pi}_0)^\ast N_j
\end{equation*} 
where $\overline{\pi}_0: \overline{G}_0 \rightarrow A_j^\prime$ denotes the usual projection.

Since $\mathcal{K}_{\mathrm{tor}}$ (resp.\ $\mathcal{K}_{\mathrm{ab}}$) is compact and contained in the open subset $\Hom_\IR^\circ(\Gm^{t},\Gm^{t^\prime})$ (resp.\ $\Hom_\IR^\circ(A,A_j^\prime)$), the distance between $\mathcal{K}_{\mathrm{tor}}$ (resp.\ $\mathcal{K}_{\mathrm{ab}}$) and the complement 
\begin{equation*}
\mathcal{C}_{\mathrm{tor}}=\Hom_\IR(\Gm^{t},\Gm^{t^\prime}) \setminus \Hom_\IR^\circ(\Gm^{t},\Gm^{t^\prime}) \
\text{(resp.\ }
\mathcal{C}_{\mathrm{ab}}=\Hom_\IR(A,A_j^\prime)\setminus \Hom_\IR^\circ(A,A_j^\prime))
\end{equation*}
is strictly positive. We assume that
\begin{equation} \label{equation::delta2}
\delta < \min \{\mathrm{dist}(\mathcal{K}_{\mathrm{tor}},\mathcal{C}_\mathrm{tor}),\mathrm{dist}(\mathcal{K}_{\mathrm{ab}},\mathcal{C}_{\mathrm{ab}})^{2} \}.
\end{equation}
By the triangle inequality, this implies that the distance between $\mathcal{K}_\delta=\mathcal{K}+B_\delta(0,0)$ and $V_\IR \setminus V_\IR^\circ$ is strictly positive. Consequently, $\mathcal{K}_\delta$ is a relatively compact subset of $V_\IR^\circ$. We choose pairs
\begin{equation} \label{equation::pairs}
(\phi_{k,\mathrm{tor}},\phi_{k,\mathrm{ab}})\in n_\delta^{-1} V,\, 1 \leq k \leq k_\delta,
\end{equation}
such that the conclusion of Lemma \ref{lemma::finite_approximation} is true. Discarding pairs if necessary, we may assume that $(\phi_{k,\mathrm{tor}},\phi_{k,\mathrm{ab}}) \in \mathcal{K}_\delta$ and hence that $(\phi_{k,\mathrm{tor}},\phi_{k,\mathrm{ab}}) \in V_\IQ^\circ$.  Our choice of the pairs (\ref{equation::pairs}) allows us to pick a pair $(\phi_{k,\mathrm{tor}},\phi_{k,\mathrm{ab}})$, $k \in \{ 1, \dots, k_\delta\}$, with $(\phi_{0,\mathrm{tor}},\phi_{0,\mathrm{ab}}) \in B_\delta(\phi_{k,\mathrm{tor}},\phi_{k,\mathrm{ab}})$. Renumbering if necessary, we can even impose that $$(\phi_{0,\mathrm{tor}},\phi_{0,\mathrm{ab}}) \in B_\delta(\phi_{1,\mathrm{tor}},\phi_{1,\mathrm{ab}})$$ in order to simplify our notation. Again, let us emphasize that it is important that we only have to choose among finitely many pairs \eqref{equation::pairs} so that all constants in the sequel can be taken independent of $k$ and hence of the point $x$.

Set $\overline{\Gamma}_{1}=\overline{\Gamma(n_\delta \cdot \phi_{1,\mathrm{tor}})} \subset (\IP^1)^{t} \times (\IP^1)^{t^\prime}$. From Construction \ref{construction3}, we obtain a compactification $G_{\overline{\Gamma}_{1}}$ endowed with a line bundle $M_{\overline{\Gamma}_{1}}$. Denoting by $\pi_{\overline{\Gamma}_{1}}: G_{\overline{\Gamma}_{1}}\rightarrow A$ the projection to the abelian quotient, the line bundle
\begin{equation*}
L_{1}=M_{\overline{\Gamma}_{1}}^{\otimes n_\delta} \otimes (\pi_{\overline{\Gamma}_{1}})^\ast (n_\delta \cdot \phi_{1,\mathrm{ab}})^\ast N_j
\end{equation*}
is nef by Lemma \ref{lemma::ampleness}.

Let $n$ be a denominator of $\phi_0$ (i.e., $n$ is an integer such that $\psi_0 = n \cdot \phi_0 \in \Hom(G,G_0)$). We also write $(\psi_{1,\mathrm{tor}},\psi_{1,\mathrm{ab}})$ for $n_\delta \cdot (\phi_{1,\mathrm{tor}},\phi_{1,\mathrm{ab}}) \in V$.
As $\mathcal{K}_\delta \subset V_\IR^\circ$ is relatively compact, there exists a constant $c_{10}>1$ such that
\begin{equation*}
c_{10}^{-1} \leq \min \{ |\phi_{\mathrm{tor}}|, |\phi_{\mathrm{ab}}| \} \leq \max \{ |\phi_{\mathrm{tor}}| , |\phi_{\mathrm{ab}}| \} \leq c_{10}
\end{equation*}
for any $(\phi_{\mathrm{tor}},\phi_{\mathrm{ab}}) \in \mathcal{K}_\delta$. Since $n_\delta^{-1}(\psi_{1,\mathrm{tor}},\psi_{1,\mathrm{ab}})=(\phi_{1,\mathrm{tor}},\phi_{1,\mathrm{ab}})\in \mathcal{K}_\delta$, we infer
\begin{equation} \label{equation::ndelta}
c_{10}^{-1} n_\delta \leq \min \{ |\psi_{1,\mathrm{tor}}|,|\psi_{1,\mathrm{ab}}| \} \leq \max \{ |\psi_{1,\mathrm{tor}}|,|\psi_{1,\mathrm{ab}}| \} \leq c_{10} n_\delta.
\end{equation}

We can now demonstrate the first of the two announced height bounds.



\begin{lemma} \label{lemma::height_approximation} There is some constant $c_{11}>0$ such that
\begin{equation} \label{equation::heightbound}
\hhat_{L_1}(x)\leq c_{11}n_\delta^2(\delta+\varepsilon)(\hhat_L(x)+1).
\end{equation}
\end{lemma}

\begin{proof} From Lemma \ref{lemma::heightsgrouplaw}, we know the estimate
\begin{equation} \label{equation::heightbound3}
\hhat_{L_1}(x)\leq 2\widehat{h}_{L_1}(y)+2\widehat{h}_{L_1}(z).
\end{equation}
We may hence bound $\widehat{h}_{L_1}(y)$ and $\widehat{h}_{L_1}(z)$ separately. Recall that
\begin{equation*}
\widehat{h}_{L}(y)=
\widehat{h}_{M_{\overline{G}}}(y)+\widehat{h}_{\overline{\pi}^\ast N}(y)
\end{equation*}
and note that
\begin{equation*}
\widehat{h}_{L_1}(y)=
n_\delta \widehat{h}_{M_1}(y)+\widehat{h}_{\overline{\pi}_1^\ast \psi_{1,\mathrm{ab}}^\ast N_j}(y)
\end{equation*}
by Lemma \ref{lemma::canonicalheight} (b,\ c). Let $c_1$ and $c_2$ be the constants of Lemma \ref{lemma::heightsclosehomomorphisms} if applied to $G=G$, $N_0=N$, $t=t^\prime$, $A_1=A_j^\prime$, $N_1=N_j$ and our fixed linear norms $|\cdot|$ on $\Hom_\IR(\Gm^{t},\Gm^{t^\prime})$ and $\Hom_\IR(A, A_j^\prime)$. Comparing Constructions \ref{construction3} and \ref{construction2}, we infer $\psi_0^\ast M_{\overline{G}_0}\approx M_{\overline{\Gamma(\psi_{0,\mathrm{tor}})}}$. From Construction \ref{construction3}, we also know the homogeneities 
\begin{equation*}
M_1^{\otimes n} = M^{\otimes n}_{\overline{\Gamma(\psi_{1,\mathrm{tor}})}} \approx
\vartheta_{\psi_{1,\mathrm{tor}},n}^\ast
 M_{\overline{\Gamma(n\cdot \psi_{1,\mathrm{tor}})}}
\text{ and }
\psi_{0}^\ast M_{\overline{G}_0}^{\otimes n_\delta}
\approx
M_{\overline{\Gamma(\psi_{0,\mathrm{tor}})}}^{\otimes n_\delta} 
\approx 
\vartheta_{\psi_{0,\mathrm{tor}},n_\delta}^\ast
M_{\overline{\Gamma(n_\delta\cdot \psi_{0,\mathrm{tor}})}},
\end{equation*}
implying 
\begin{equation*}
n \hhat_{M_1}(y) = \hhat_{M_{\overline{\Gamma(n\cdot \psi_{1,\mathrm{tor}})}}}(y)
\text{ and }
n_\delta \hhat_{\psi_0^\ast M_{\overline{G}_0}}(y) = \hhat_{M_{\overline{\Gamma(n_\delta\cdot \psi_{0,\mathrm{tor}})}}}(y).
\end{equation*}
Invoking Lemma \ref{lemma::heightsclosehomomorphisms} for $n \cdot (\psi_{1,\mathrm{tor}},\psi_{1,\mathrm{ab}})$ and $n_\delta \cdot (\psi_{0,\mathrm{tor}},\psi_{0,\mathrm{ab}})$ yields
\begin{equation} \label{equation::difference1}
|n \hhat_{M_1}(y)-n_\delta \hhat_{\psi_0^\ast M_{\overline{G}_0}}(y)|< c_1 |n \cdot \psi_{1,\mathrm{tor}}-n_\delta\cdot \psi_{0,\mathrm{tor}}| \hhat_{M_{\overline{G}}}(y)
\end{equation}
and
\begin{equation} \label{equation::difference2}
|n^2\hhat_{\overline{\pi}^\ast \psi_{1,\mathrm{ab}}^\ast N_j}(y)-n_\delta^2\hhat_{\psi_0^\ast(\overline{\pi}_0)^\ast N_j}(y)|<c_2 | n \cdot \psi_{1,\mathrm{ab}}- n_\delta \cdot \psi_{0,\mathrm{ab}} |^{2} \hhat_{\overline{\pi}^\ast N}(y).
\end{equation}
As $y \in \ker(\psi_0)+\mathrm{Tor}(G)$, we have $\psi_0(y) \in \mathrm{Tor}(G_0)$ and consequently
\begin{equation} \label{equation::zero}
\hhat_{L_0}(\psi_0(y))= \hhat_{M_{\overline{G}_0}}(\psi_0(y)) + \hhat_{(\overline{\pi}_0)^\ast N_j}(\psi_0(y)) = 0.
\end{equation}
With Lemma \ref{lemma::nonnegativeheight}, we obtain $\hhat_{\psi_{0}^\ast M_{\overline{G}_0}}(y) = \hhat_{M_{\overline{G}_0}}(\psi_0(y)) = 0$ and $\hhat_{\psi_0^\ast(\overline{\pi}_0)^\ast N_j}(y) =\hhat_{(\overline{\pi}_0)^\ast N_j}(\psi_0(y)) = 0$ from (\ref{equation::zero}). Since 
\begin{equation}
n\cdot (\psi_{1,\mathrm{tor}},\psi_{1,\mathrm{ab}}) - n_\delta \cdot (\psi_{0,\mathrm{tor}},\psi_{0,\mathrm{ab}})=n n_\delta \cdot \left((\phi_{1,\mathrm{tor}},\phi_{1,\mathrm{ab}})-(\phi_{0,\mathrm{tor}},\phi_{0,\mathrm{ab}})\right),
\end{equation}
we may cancel $n$ (resp.\ $n^2$) in (\ref{equation::difference1}) (resp.\ (\ref{equation::difference2})) and obtain
\begin{equation} \label{equation::heightbound1_1}
\widehat{h}_{L_1}(y)< c_{12} \delta n_\delta^2 \widehat{h}_{L}(y) \leq c_{12} \delta n_\delta^2 (4\hhat_L(x)+1)
\end{equation}
for some constant $c_{12}>0$.
Applying Lemma \ref{lemma::heightsclosehomomorphisms} to $(\psi_{1,\mathrm{tor}},\psi_{1,\mathrm{ab}})$ and $(0,0) \in V$, we obtain similarly
\begin{equation*}
\hhat_{L_1}(z) < c_{13} (n_\delta |\psi_{1,\mathrm{tor}}|+|\psi_{1,\mathrm{ab}}|^2) \widehat{h}_{L}(z)
\end{equation*}
for some constant $c_{13}>0$. Using \eqref{equation::ndelta}, we deduce from this the estimate
\begin{equation} \label{equation::heightbound2_2}
\hhat_{L_1}(z) < 2c_{13} c_{10}^2 n_\delta^2 \widehat{h}_{L}(z)\leq 4 \varepsilon c_{13} c_{10}^2 n_\delta^2(2\hhat_L(x)+1).
\end{equation}
Finally, (\ref{equation::heightbound}) follows from combining (\ref{equation::heightbound3}), (\ref{equation::heightbound1_1}) and (\ref{equation::heightbound2_2}).
\end{proof}

Our second height bound is a consequence of Siu's numerical bigness criterion (\cite[Corollary 1.2]{Siu1993}). Recall from (\ref{equation::graphdiagram2}) the maps $\iota=\iota_{\overline{\Gamma}_{1}}: G \rightarrow G_{\overline{\Gamma}_1}$ and $q = q_{\overline{\Gamma}_{1}}: G_{\overline{\Gamma}_{1}} \rightarrow \overline{G}$. The idea is to compare the line bundles $L_1$ and $q^\ast L$ on the Zariski closure $\overline{X}$ of $\iota(X) \subset G_{\overline{\Gamma}_{1}}$. Set $r=\dim(X) \geq 1$ and
\begin{equation} \label{equation::alpha}
\alpha =\frac{\deg(c_1(L_1)^r \cap \left[ \overline{X} \right])}{\max \{ 1, 2r\deg(c_1(L_1)^{r-1} c_1(q^\ast L) \cap \left[ \overline{X} \right])\}}.
\end{equation}
We note that both $L_1$ and $q^\ast L$ are nef.

\begin{lemma} \label{lemma::siu} There exists a non-empty Zariski open subset $U_\delta \subseteq X$ and a constant $c_{14}(\delta)$, both depending on $\delta$, such that
\begin{equation} \label{equation::siuheights}
\alpha \widehat{h}_{L}(x) - c_{14}(\delta) \leq \widehat{h}_{L_1}(x)
\end{equation}
if $x\in U_\delta(\IQbar)$.
\end{lemma}

\begin{proof} If $\deg(c_1(L_1)^r \cap [\overline{X}]) = 0$, then $\alpha = 0$ and there is nothing left to prove because $\widehat{h}_{L_1}(x)$ is non-negative by Lemma \ref{lemma::nonnegativeheight}. Hence, we may and do assume $\deg(c_1(L_1)^r \cap [\overline{X}]) \neq 0$. As $L_1$ is nef, this actually means $\deg(c_1(L_1)^r \cap [\overline{X}]) >0$. Set 
\begin{equation*}
u = \deg(c_1(L_1)^r \cap [\overline{X}]) \text{ and } v = \max \{ 1, 2r\deg(c_1(L_1)^{r-1} c_1(q^\ast L) \cap [\overline{X}]) \}.
\end{equation*}
This is arranged so that 
\begin{align*}
\deg(c_1(L_1^{\otimes v})^r \cap [\overline{X}]) = v^r \deg(c_1(L_1)^r \cap [\overline{X}]) &\geq 2rv^{r-1}u \deg(c_1(L_1)^{r-1} c_1(q^\ast L) \cap [\overline{X}]) \\ &= 2r \deg(c_1(L_1^{\otimes v})^{r-1} c_1(q^\ast L^{\otimes u}) \cap [\overline{X}]).
\end{align*}
Thus, Siu's criterion as stated in \cite[Theorem 2.2.15]{Lazarsfeld2004} implies that $L_2=(L_1^{\otimes v} \otimes q^\ast L^{\otimes (-u)})|_{\overline{X}}$ is big. In particular, some power of $L_2$ is effective. By \cite[Theorem B.3.6]{Hindry2000}, there exists a non-empty Zariski-open set $U_\delta \subseteq X$ and a constant $c_{15}(L_2)>0$ such that 
\begin{equation*}
- c_{15}(L_2) \leq h_{L_1^{\otimes v} \otimes q^\ast L^{\otimes (-u)}}(x)
\end{equation*}
for all $x \in U_\delta(\IQbar)$. For a fixed $\delta>0$, we wind up here with finitely many choices for $\overline{X} \subset G_{\overline{\Gamma}_{1}}$ and the line bundles $L_1$ and $q^\ast L$ on $G_{\overline{\Gamma}_{1}}$. As $L_2$ can be determined from this data, we can hence replace $ c_{15}(L_2^{\otimes w})$ by some constant depending only on $\delta$. Combining this fact with Lemma \ref{lemma::canonicalheight} (a), we conclude the existence of some constant $c_{16}(\delta)>0$ such that
\begin{equation*}
- c_{16}(\delta) \leq \widehat{h}_{L_1^{\otimes v} \otimes q^\ast L^{\otimes (-u)}}(x) 
\end{equation*}
whenever $x\in U_\delta(\IQbar)$. Inequality (\ref{equation::siuheights}) follows immediately by using Lemma \ref{lemma::canonicalheight} (c), Lemma \ref{lemma::functoriality}, and $\alpha = u/v$.
\end{proof}

It remains to bound the quantity $\alpha$ from below.

\begin{lemma} \label{lemma::alpha} There exists a constant $c_{17}>0$ such that $\alpha \geq c_{17} n_\delta^2$.
\end{lemma}

\begin{proof} We first define auxiliary functions $\beta_i$, $0 \leq i \leq r$, and $\gamma_i$, $0 \leq i \leq r-1$, on $V_\IQ$. Once again, we use that Construction \ref{construction3} gives for each $\varphi_{\mathrm{tor}} \in \Hom(\Gm^{t},\Gm^{t^\prime})$ a compactification $G_{\overline{\Gamma(\varphi_{\mathrm{tor}})}}$ of $G$ with abelian quotient $\pi_{\overline{\Gamma(\varphi_{\mathrm{tor}})}}: G_{\overline{\Gamma(\varphi_{\mathrm{tor}})}} \rightarrow A$ and a line bundle $M_{\overline{\Gamma(\varphi_{\mathrm{tor}})}}$ on $G_{\overline{\Gamma(\varphi_{\mathrm{tor}})}}$. We use the notations introduced in (\ref{equation::graphdiagram2}). In addition, we let $X_{\overline{\Gamma(\varphi_{\mathrm{tor}})}}$ be the Zariski closure of $\iota_{\overline{\Gamma(\varphi_{\mathrm{tor}})}}(X)$ in $G_{\overline{\Gamma(\varphi_{\mathrm{tor}})}}$. For any $(\varphi_{\mathrm{tor}},\varphi_{\mathrm{ab}}) \in V$, we can now define
\begin{equation*}
\beta_i(\varphi_{\mathrm{tor}},\varphi_{\mathrm{ab}})= (\vert \varphi_{\mathrm{tor}} \vert + \vert \varphi_{\mathrm{ab}} \vert)^i \deg \! \left( c_1(M_{\overline{\Gamma(\varphi_{\mathrm{tor}})}})^{i} c_1((\varphi_{\mathrm{ab}} \circ \pi_{\overline{\Gamma(\varphi_{\mathrm{tor}})}})^\ast N_j)^{r-i} \cap [X_{\overline{\Gamma(\varphi_{\mathrm{tor}})}}] \right)
\end{equation*}
and
\begin{equation*}
\gamma_i(\varphi_{\mathrm{tor}},\varphi_{\mathrm{ab}})=(\vert \varphi_{\mathrm{tor}} \vert + \vert \varphi_{\mathrm{ab}} \vert)^i \deg \! \left( c_1(M_{\overline{\Gamma(\varphi_{\mathrm{tor}})}})^{i}  c_1((\varphi_{\mathrm{ab}} \circ \pi_{\overline{\Gamma(\varphi_{\mathrm{tor}})}})^\ast N_j)^{r-1-i}  c_1(q^\ast_{\overline{\Gamma(\varphi_{\mathrm{tor}})}} L) \cap [X_{\overline{\Gamma(\varphi_{\mathrm{tor}})}}] \right).
\end{equation*}
This defines a $\IZ$-homogeneous function $\beta_i$ (resp.\ $\gamma_i$) of degree $2r$ (resp. $2r-2$) on $V$. To prove this, we recall that Construction \ref{construction3} provides a finite birational morphism $\vartheta_{n,\varphi_{\mathrm{tor}}}: G_{\overline{\Gamma(\varphi_{\mathrm{tor}})}} \rightarrow G_{\overline{\Gamma(n \cdot \varphi_{\mathrm{tor}})}}$ such that $\vartheta_{\varphi_{\mathrm{tor}},n}^\ast M_{\overline{\Gamma(n\cdot \varphi_{\mathrm{tor}})}} 
\approx M_{\overline{\Gamma(\varphi_{\mathrm{tor}})}}^{\otimes n}$. The homogeneity relation follows from the projection formula 
(cf.\ \cite[Proposition 2.5 (c)]{Fulton1998}) by using the straightforward relations $(\vartheta_{\varphi_{\mathrm{tor}},n})_\ast [X_{\overline{\Gamma(\varphi_{\mathrm{tor}})}}]=[X_{\overline{\Gamma(n\cdot \varphi_{\mathrm{tor}})}}]$, $\vartheta_{\varphi_{\mathrm{tor}},n}^\ast ((n \cdot \varphi_{\mathrm{ab}}) \circ \pi_{\overline{\Gamma(n \cdot \varphi_{\mathrm{tor}})}})^\ast N_j = (\varphi_{\mathrm{ab}} \circ \pi_{\overline{\Gamma(\varphi_{\mathrm{tor}})}})^\ast N_j^{\otimes n^2}$ and $\vartheta_{\varphi_{\mathrm{tor}},n}^\ast (q^\ast_{\overline{\Gamma(n \cdot \varphi_{\mathrm{tor}})}} L) = q^\ast_{\overline{\Gamma(\varphi_{\mathrm{tor}})}} L$. Therefore, we may and do extend both $\beta_i$ and $\gamma_i$ to unique $\IQ$-homogeneous functions on $V_\IQ$. We denote these extensions also by $\beta_i$ and $\gamma_i$. By \cite[Theorem III.2.1]{Kleiman1966}, the nefness of $M_{\overline{\Gamma(\varphi_{\mathrm{tor}})}}$, $N_j$ and $L$ implies that all $\beta_i$ and $\gamma_i$ are non-negative.

Recall that $\overline{X}$ is the Zariski closure of $\iota(X)$ in $G_{\overline{\Gamma}_{1}}$.
The reason for introducing the functions $\beta_i$ and $\gamma_i$ are the relations
\begin{equation*}
\deg (c_1(L_1)^r \cap [ \overline{X} ]) = \sum_{i=0}^{r} \binom{r}{i} \frac{n_\delta^i}{(\vert \psi_{1,\mathrm{tor}} \vert + \vert \psi_{1,\mathrm{ab}} \vert)^i} \beta_i(\psi_{1,\mathrm{tor}},\psi_{1,\mathrm{ab}})
\end{equation*}
and
\begin{equation*}
\deg (c_1(L_1)^{r-1} c_1(q^\ast L) \cap [ \overline{X} ]) = \sum_{i=0}^{r-1} \binom {r-1} {i} \frac{n_\delta^i}{(\vert \psi_{1,\mathrm{tor}} \vert + \vert \psi_{1,\mathrm{ab}} \vert)^i} \gamma_i(\psi_{1,\mathrm{tor}},\psi_{1,\mathrm{ab}}).
\end{equation*}
Each $n_\delta^i/(|\psi_{1,\mathrm{tor}}| + |\psi_{1,\mathrm{ab}}|)^i$, $0\leq i \leq r$, is bounded both from above and below by virtue of (\ref{equation::ndelta}). Therefore, the assertion of the lemma follows by homogeneity from the existence of constants $c_{18}, c_{19} > 0$ such that
\begin{equation*}
\max_{0 \leq i \leq r} \{ \beta_i(\phi_{\mathrm{tor}},\phi_{\mathrm{ab}}) \} \geq c_{18}
\end{equation*}
and
\begin{equation*}
\max_{0 \leq i \leq r-1} \{ \gamma_i(\phi_{\mathrm{tor}},\phi_{\mathrm{ab}}) \}  \leq c_{19}
\end{equation*}
for every $(\phi_{\mathrm{tor}},\phi_{\mathrm{ab}}) \in \mathcal{K}_\delta \cap V_\IQ$. The former bound is stated as Lemma \ref{lemma::intersection1} and the latter as Lemma \ref{lemma::intersection2} below.
\end{proof}

Lemma \ref{lemma::alpha} allows us to make (\ref{equation::siuheights}) precise: There exists a non-empty Zariski open $U_{\delta} \subset X$ such that
\begin{equation} \label{equation::secondheightinequality}
c_{17} n_\delta^2 \widehat{h}_{L}(x) - c_{14}(\delta) \leq \widehat{h}_{L_1}(x)
\end{equation}
whenever $x \in U_{\delta}(\IQbar)$. Combining this with (\ref{equation::heightbound}), we obtain \begin{equation*}
c_{17} n_\delta^2 \widehat{h}_{L}(x) - c_{14}(\delta) \leq c_{11}n_\delta^2(\delta+\varepsilon)(\hhat_L(x)+1).
\end{equation*}
Canceling $n_\delta^2$, this can be rewritten as
\begin{equation*}
(c_{17}-c_{11}(\delta+\varepsilon))\hhat_L(x)\leq n_{\delta}^{-2}c_{14}(\delta) + c_{11}(\delta+\varepsilon).
\end{equation*}
This inequality gives the desired upper bound on $\hhat_L(x)$ if
\begin{equation} \label{equation::delta}
\max \{ \delta, \varepsilon \} < \frac{1}{2}c_{11}^{-1}c_{17}.
\end{equation}
Consequently, Proposition \ref{proposition::induction_step} is proven up to Lemmas \ref{lemma::intersection1} and \ref{lemma::intersection2}, whose proofs are provided next in Section \ref{section::intersections}.

\subsection{Bounds on intersection numbers}
\label{section::intersections}

The reader may profitably compare our derivation of Lemma \ref{lemma::intersection1} with the lengthy one of \cite[Proposition 4]{Habegger2009} to appreciate the technical advantage provided by using Chern forms. In fact, our argument is particularly simple if $G$ is an abelian variety because most of Section \ref{section::chernforms} is not needed in this case and only the functions $\beta_0$ and $\gamma_0$ are non-zero.

\begin{lemma} \label{lemma::intersection1} 
Assume $X^{(s)} \neq X$. There exists a constant $c_{18}>0$ such that
\begin{equation*}
\max_{0 \leq i \leq r} \{ \beta_i(\phi_{\mathrm{tor}},\phi_{\mathrm{ab}}) \} \geq c_{18}
\end{equation*}
for all $(\phi_{\mathrm{tor}},\phi_{\mathrm{ab}}) \in \mathcal{K}_\delta \cap V_\IQ$.
\end{lemma}

Before starting the proof, let us recall a compatibility between algebraic Chern classes and analytic Chern forms on proper complex algebraic varieties. Let $Z$ be a proper complex algebraic variety and let $L_1,\dots,L_n$ be line bundles on $Z$. If $\Vert \cdot \Vert_i$ ($1 \leq i \leq n$) are smooth Hermitian metrics on $L_i$, then
\begin{equation*}
c_1(L_1) c_1(L_2) \dots c_1(L_n) \cap [Z] = \int_{Z(\IC)} c_1(L_1, \Vert \cdot \Vert_1)c_1(L_2, \Vert \cdot \Vert_2)\dots c_1(L_n, \Vert \cdot \Vert_n).
\end{equation*}
In case $Z$ is smooth, this follows from the fact that the topological Chern class of a line bundle is given by its Chern form (see e.g.\ \cite[Proposition on p.\ 141]{Griffiths1994}) and the compatibility between algebraic Chern classes and their topological counterparts acting on singular homology \cite[Proposition 19.1.2]{Fulton1998}. For general $Z$, one can reduce to this case via Hironaka's desingularization theorem \cite{Hironaka1964} (see also \cite{Kollar2007}).

\begin{proof} Since $\mathcal{K}_\delta$ is a relatively compact subset of $V_\IR^\circ$, it suffices to prove the following claim: For each $(\phi_{\mathrm{tor}}^\prime,\phi_{\mathrm{ab}}^\prime) \in V^\circ_\IR$, there exists a euclidean neighborhood $U \subset V_\IR^\circ$ of $(\phi_{\mathrm{tor}}^\prime,\phi_{\mathrm{ab}}^\prime)$ and a constant $c_{20}(\phi_{\mathrm{tor}}^\prime,\phi_{\mathrm{ab}}^\prime)>0$ such that
\begin{equation}
\max_{0 \leq i \leq r} \{ \beta_i(\phi_{\mathrm{tor}},\phi_{\mathrm{ab}}) \} \geq c_{20}(\phi_{\mathrm{tor}}^\prime,\phi_{\mathrm{ab}}^\prime)
\end{equation}
for all $(\phi_{\mathrm{tor}},\phi_{\mathrm{ab}}) \in U \cap V_\IQ$.

In order to prove this claim, let $(\phi_{\mathrm{tor}},\phi_{\mathrm{ab}}) \in V_\IQ$ and let $n$ denote a denominator for $(\phi_{\mathrm{tor}},\phi_{\mathrm{ab}})$. In Section \ref{section::chernforms}, the line bundle $M_{\overline{\Gamma(n\cdot \phi_{\mathrm{tor}})}}$ is endowed with a hermitian metric such that $c_1(\overline{M}_{\overline{\Gamma(n\cdot \phi_{\mathrm{tor}})}})= \omega(n\cdot \phi_{\mathrm{tor}})$. Similarly, the line bundle $N_j$ is endowed with a hermitian metric such that $c_1(\overline{N}_j)=\omega(N_j; n\cdot \phi_{\mathrm{ab}})$. These hermitian line bundles can be used to express $\beta_l(\phi_{\mathrm{tor}},\phi_{\mathrm{ab}})$ analytically; to wit, $\beta_l(\phi_{\mathrm{tor}},\phi_{\mathrm{ab}})= n^{-2r}\beta_l(n\cdot \phi_{\mathrm{tor}},n\cdot \phi_{\mathrm{ab}})$ and 
\begin{align*}
\beta_l(n\cdot \phi_{\mathrm{tor}},n\cdot \phi_{\mathrm{ab}})&= (|n\cdot \phi_{\mathrm{tor}}| + |n \cdot \phi_{\mathrm{ab}}|)^l \int_{X_{\overline{\Gamma(n\cdot \phi_{\mathrm{tor}})}}(\IC)}{\omega(n\cdot \phi_{\mathrm{tor}})^{\wedge l} \wedge (\overline{\pi}_{\Gamma(n \cdot \phi_{\mathrm{tor}})})^\ast \omega(N_j; n \cdot \phi_{\mathrm{ab}})^{\wedge r-l}}.
\end{align*}
Since each $\beta_l$ is a non-negative function, it suffices to prove that there exists a positive constant $c_{21}(\phi_{\mathrm{tor}}^\prime,\phi_{\mathrm{ab}}^\prime)$ and a neighborhood $U$ of $(\phi_{\mathrm{tor}}^\prime,\phi_{\mathrm{ab}}^\prime)$ such that
\begin{multline} \label{equation::integral}
n^{2r} \sum_{l=0}^{r} \binom{r}{l}\beta_l(\phi_{\mathrm{tor}}, \phi_{\mathrm{ab}}) \\
= \int_{X_{\overline{\Gamma(n\cdot \phi_{\mathrm{tor}})}}(\IC)}{\left((|n\cdot \phi_{\mathrm{tor}}| + |n \cdot \phi_{\mathrm{ab}}|) \omega(n\cdot \phi_{\mathrm{tor}}) + (\overline{\pi}_{\Gamma(n \cdot \phi_{\mathrm{tor}})})^\ast \omega(M_j; n \cdot \phi_{\mathrm{ab}})\right)^{\wedge r}}
\end{multline}
exceeds $n^{2r}c_{21}(\phi_{\mathrm{tor}}^\prime,\phi_{\mathrm{ab}}^\prime)$ for any $(\phi_{\mathrm{tor}},\phi_{\mathrm{ab}}) \in U \cap V_\IQ$ with denominator $n$. As the boundary $X_{\overline{\Gamma(n\cdot \phi_{\mathrm{tor}})}}(\IC) \setminus \iota_{\overline{\Gamma(n\cdot \phi_{\mathrm{tor}})}}(X)(\IC)$ has measure zero, the integral in (\ref{equation::integral}) equals
\begin{equation*}
\int_{X(\IC)}{\left((|n\cdot \phi_{\mathrm{tor}}| + |n \cdot \phi_{\mathrm{ab}}|) \omega(n\cdot \phi_{\mathrm{tor}}) + (\overline{\pi}_{\Gamma(n \cdot \phi_{\mathrm{tor}})})^\ast \omega(N_j; n \cdot \phi_{\mathrm{ab}})\right)^{\wedge r}},
\end{equation*}
which by Lemma \ref{lemma::chernforms_quasihomomorphisms} and (\ref{equation::abelianhomogenity}) simplifies to
\begin{equation*}
n^{2r} \int_{X(\IC)}{\left((|\phi_{\mathrm{tor}}| + | \phi_{\mathrm{ab}}|) \omega(\phi_{\mathrm{tor}}) + (\overline{\pi}_{\Gamma(\phi_{\mathrm{tor}})})^\ast \omega(N_j; \phi_{\mathrm{ab}})\right)^{\wedge r}}.
\end{equation*}
It remains to show that the integral
\begin{equation} \label{equation::integraltobound}
\int_{X(\IC)}{((|\phi_{\mathrm{tor}}|+|\phi_{\mathrm{ab}}|)
\omega(\phi_{\mathrm{tor}})+\pi^\ast\omega(N_j;\phi_{\mathrm{ab}}))^{\wedge r}}
\end{equation}
is bounded from below by a positive constant $c_{21}(\phi_{\mathrm{tor}}^\prime,\phi_{\mathrm{ab}}^\prime)$ for all $(\phi_{\mathrm{tor}},\phi_{\mathrm{ab}}) \in V_{\IQ}$ in a neighborhood $U$ of $(\phi_{\mathrm{tor}}^\prime,\phi_{\mathrm{ab}}^\prime)$. In the sequel, we write
\begin{equation}
\label{equation::definition_omega}
\omega(\phi_{\mathrm{tor}},\phi_{\mathrm{ab}}) = (|\phi_{\mathrm{tor}}|+|\phi_{\mathrm{ab}}|)
\omega(\phi_{\mathrm{tor}})+\pi^\ast\omega(N_j;\phi_{\mathrm{ab}})
\end{equation}
for any $(\phi_{\mathrm{tor}},\phi_{\mathrm{ab}})\in V_{\IR}$. From Section \ref{section::chernforms}, we know that each $
\omega(\phi_{\mathrm{tor}},\phi_{\mathrm{ab}})$ is a semipositive $(1,1)$-form of real type. Furthermore, our assumption $X\neq X^{(s)}$ implies $(\omega(\phi_{\mathrm{tor}}^\prime,\phi_{\mathrm{ab}}^\prime)|_X)^{\wedge \dim(X)} \neq 0$ by Lemma \ref{lemma::ax} (with $c=|\phi_{\mathrm{tor}}^\prime|+|\phi_{\mathrm{ab}}^\prime|$). We infer from this the existence of a non-empty relatively compact open subset $K$ such that $(\omega(\phi_{\mathrm{tor}}^\prime,\phi_{\mathrm{ab}}^\prime)|_{X,y})^{\wedge \dim(X)}$ is a positive volume form for each $y \in K$. By continuity of $\omega(\phi_{\mathrm{tor}},\phi_{\mathrm{ab}})$ with respect to $(\phi_{\mathrm{tor}},\phi_{\mathrm{ab}})$ and compactness, there exists an open neighborhood $U \subset V_\IR$ such that
\begin{equation*}
\omega(\phi_{\mathrm{tor}},\phi_{\mathrm{ab}}) ^{\wedge \dim(X)}-\frac{1}{2}\omega(\phi_{\mathrm{tor}}^\prime,\phi_{\mathrm{ab}}^\prime) ^{\wedge \dim(X)}
\end{equation*}
restricts to a positive volume form on each $T_{\IR,y} X^{\mathrm{sm}}(\IC)$, $y \in K$. Using the semipositivity of $\omega(\phi_{\mathrm{tor}},\phi_{\mathrm{ab}})$, we obtain that (\ref{equation::integraltobound}) is bounded from below by
\begin{equation*}
\int_{K}{\omega(\phi_{\mathrm{tor}},\phi_{\mathrm{ab}})^{\wedge \dim(X)}} \geq \frac{1}{2}\int_{K} \omega(\phi_{\mathrm{tor}}^\prime,\phi_{\mathrm{ab}}^\prime) ^{\wedge \dim(X)} = c_{21}(\phi_{\mathrm{tor}}^\prime,\phi_{\mathrm{ab}}^\prime)> 0.
\end{equation*}
This proves our claim.
\end{proof}

\begin{lemma} \label{lemma::intersection2} There exists a constant $c_{22}>0$ such that
\begin{equation*}
\max_{0 \leq i \leq r-1} \{ \gamma_i(\phi_{\mathrm{tor}},\phi_{\mathrm{ab}}) \} \leq c_{22}
\end{equation*}
for all $(\phi_{\mathrm{tor}},\phi_{\mathrm{ab}}) \in \mathcal{K}_\delta \cap V_\IQ$.
\end{lemma}

It is tempting to provide a proof resembling the one of Lemma \ref{lemma::intersection1}. In fact, we can reduce the statement of the lemma to bounds on certain integrals of volume forms on $X(\IC)$ that vary continuously with $(\phi_{\mathrm{tor}},\phi_{\mathrm{ab}})$. If $X(\IC)$ were compact (e.g.\ because $G=A$ is an abelian variety), the above lemma could be immediately inferred from this continuity. However, non-compactness of $X(\IC)$ precludes such a direct argument in the general case. We circumvent these problems by using algebraic intersection theory \cite{Fulton1998} instead. This resembles the proof of \cite[Lemma 3.3]{Habegger2009} by a multiprojective version of B\'ezout's Theorem. We use the standard notation from \cite{Fulton1998} freely.

\begin{proof} Consider a fixed $(\phi_{\mathrm{tor}},\phi_{\mathrm{ab}})\in \mathcal{K}_\delta \cap V_\IQ$ with denominator $n$. By compactness, $(|\phi_{\mathrm{tor}}| + |\phi_{\mathrm{ab}}|)$ is bounded on $\mathcal{K}_\delta$. It suffices to bound
\begin{equation*}
\deg(c_1(M_{\overline{\Gamma(n\cdot \phi_{\mathrm{tor}})}})^{i} c_1(((n\cdot \phi_{\mathrm{ab}}) \circ \overline{\pi}_{\overline{\Gamma(n\cdot \phi_{\mathrm{tor}})}})^\ast N_j)^{r-1-i} c_1(q^\ast_{\overline{\Gamma(n \cdot \phi_{\mathrm{tor}})}} L) \cap [X_{\overline{\Gamma(n \cdot \phi_{\mathrm{tor}})}}])
\end{equation*}
by $n^{2r-2}c_{22}$ because $\gamma_i(\phi_{\mathrm{tor}},\phi_{\mathrm{ab}})$ is homogeneous of degree $2r-2$. As in the proof of Lemma \ref{lemma::intersection1}, it is enough to demonstrate that
\begin{equation} \label{equation::intersectionnumber3}
\deg(c_1(M_{\overline{\Gamma(n\cdot \phi_{\mathrm{tor}})}} \otimes ((n\cdot \phi_{\mathrm{ab}}) \circ \overline{\pi}_{\overline{\Gamma(n\cdot \phi_{\mathrm{tor}})}})^\ast N_j)^{r-1}  c_1(q^\ast_{\overline{\Gamma(n \cdot \phi_{\mathrm{tor}})}} L) \cap [X_{\overline{\Gamma(n \cdot \phi_{\mathrm{tor}})}}])
\end{equation}
is bounded by $n^{2r-2}c_{23}$. 

Let $G^\prime$ be the semiabelian variety described by $\eta_{G^\prime} = (n \cdot \phi_{\mathrm{tor}})_\ast \eta_{G} \in \Ext^1_{\IQbar}(A,\Gm^{t^\prime})$. From Construction \ref{construction1}, we recall the compactification $\overline{G}$ (resp.\ $\overline{G}^\prime$) of $G$ (resp.\ $G^\prime$) with its abelian quotient $\overline{\pi}: \overline{G} \rightarrow A$ (resp.\ $\overline{\pi}^\prime: \overline{G}^\prime \rightarrow A$) and the line bundle $M_{\overline{G}}$ (resp.\ $M_{\overline{G}^\prime}$) on $\overline{G}$ (resp.\ $\overline{G}^\prime$). The Zariski closure of $X$ in $\overline{G}$ is denoted $\overline{X}$. Then, $L = M_{\overline{G}} \otimes \overline{\pi}^\ast N$ and we also set
\begin{equation*}
L^\prime = M_{\overline{G}^\prime} \otimes ((n\cdot \phi_{\mathrm{ab}}) \circ \overline{\pi}^\prime)^\ast N_j.
\end{equation*}
The homomorphism $(\id_{\Gm^{t}},n\cdot \phi_{\mathrm{tor}}): \Gm^{t} \rightarrow \Gm^{t} \times \Gm^{t^\prime}$ extends to a $(\id_{\Gm^{t}},n\cdot \phi_{\mathrm{tor}})$-equivariant map $\overline{\Gamma(n\cdot \phi_{\mathrm{tor}})} \rightarrow (\IP^1)^{t}\times (\IP^1)^{t^\prime}$, yielding a closed immersion $\iota: G_{\overline{\Gamma(n\cdot \phi_{\mathrm{tor}})}} \rightarrow \overline{G} \times \overline{G}^\prime$ by means of the constructions in Section \ref{section::compactification}. Furthermore, the line bundle $q^\ast_{\overline{\Gamma(n \cdot \phi_{\mathrm{tor}})}} L$ (resp.\ $M_{\overline{\Gamma(n\cdot \phi_{\mathrm{tor}})}}$) on $G_{\overline{\Gamma(n \cdot \phi_{\mathrm{tor}})}}$ coincides with the pullback $\iota^\ast \pr_1^\ast L$ (resp.\ $\iota^\ast \pr_2^\ast M_{\overline{G}^\prime}$). Using the projection formula (\cite[Proposition 2.5.c]{Fulton1998}), we infer that (\ref{equation::intersectionnumber3}) equals the degree of
\begin{equation} \label{equation::intersectionnumber6}
c_1(\pr_1^\ast L)c_1(\pr_2^\ast L^\prime)^{r-1} \cap [\iota(X_{\overline{\Gamma(n \cdot \phi_{\mathrm{tor}})}})] \in A_0(\overline{G} \times \overline{G}^\prime).
\end{equation}
To estimate this degree, we use suitable projective embeddings $\overline{G} \hookrightarrow \IP^{r_1}$ and $\overline{G}^\prime \hookrightarrow \IP^{r_2}$. By Lemma \ref{lemma::ampleness}, the line bundles $L$ and $L^\prime_0= L^\prime \otimes (\overline{\pi}^\prime)^\ast N$ are ample. Consequently, there exists an integer $l_1$ such that $L^{\otimes l_1}$ is very ample. Since $L$ is independent of $(\phi_{\mathrm{tor}},\phi_{\mathrm{ab}})$, we can choose $l_1$ less than some constant $c_{24}$ that only depends on $G$ and $X$. The line bundle $(L^\prime_0)^{\otimes l_2}$ is very ample if $l_2$ is sufficiently large; in contrast to $l_1$, there is an implicit dependence on $(\phi_{\mathrm{tor}},\phi_{\mathrm{ab}})$ here. These very ample line bundles determine projective embeddings $\iota_1: \overline{G} \hookrightarrow \IP^{r_1}$ and $\iota_2: \overline{G}^\prime \hookrightarrow \IP^{r_2}$ such that $\iota_1^\ast \mathcal{O}_{\IP^{r_1}}(1) \approx L^{\otimes l_1}$ and $\iota_2^\ast \mathcal{O}_{\IP^{r_2}}(1) \approx (L^\prime_{0})^{\otimes l_2}$. Setting $\kappa=(\iota_1 \times \iota_2) \circ \iota$, we continue by estimating the degree of
\begin{equation} \label{equation::intersectionnumber1}
c_1(\pr_1^\ast \mathcal{O}_{\IP^{r_1}}(1))c_1(\pr_2^\ast \mathcal{O}_{\IP^{r_2}}(1))^{r-1}\cap [\kappa(X_{\overline{\Gamma(n\cdot \phi_{\mathrm{tor}})}})]\in A_0(\IP^{r_1} \times \IP^{r_2}).
\end{equation}
If it is shown that the degree of  (\ref{equation::intersectionnumber1}) is less than $l_1l_2^{r-1} n^{2r-2} c_{23}$, the desired degree bound on (\ref{equation::intersectionnumber6}) follows immediately. In fact, the degree of (\ref{equation::intersectionnumber1}) equals the degree of
\begin{equation*}
l_1l_2^{r-1}c_1(\pr_1^\ast L)(c_1(\pr_2^\ast L^\prime)+c_1(\pr_2^\ast(\overline{\pi}^\prime)^\ast N))^{r-1} \cap [\iota(X_{\overline{\Gamma(n\cdot \phi_{\mathrm{tor}})}})] \in A_0(\overline{G} \times \overline{G}^\prime)
\end{equation*}
by the projection formula. By Lemma \ref{lemma::ampleness}, the line bundles $\pr_1^\ast L$, $\pr_2^\ast L^\prime$ and $\pr_2^\ast (\overline{\pi}^\prime)^\ast N$ are nef so that this can be expanded into a sum of $r$ zero-cycle classes with non-negative degrees (see \cite[Theorem III.2.1]{Kleiman1966}). Since one of the summands is a $(l_1l_2^{r-1})$-multiple of (\ref{equation::intersectionnumber6}), the reduction is clear. (Note that both $l_1$ and $l_2$ cancel out in this way, and hence the dependence of $l_2$ on $(\varphi_{\mathrm{tor}},\varphi_{\mathrm{ab}})$ is not an issue. Of course, we have to make sure that $c_{23}$ depends only on $G$ and $X$, as it should be by our convention.)

The variety $\kappa(X_{\overline{\Gamma(n \cdot \phi_{\mathrm{tor}})}})$ is an irreducible component of $\kappa(G_{\overline{\Gamma(n\cdot \phi_{\mathrm{tor}})}}) \cap (\iota_1(\overline{X}) \times \IP^{r_2}) \subset \IP^{r_1} \times \IP^{r_2}$. In fact, both are subvarieties of $\iota_1(\overline{G}) \times \iota_2(\overline{G}^\prime)$ whose restrictions to the open dense subset $\iota_1(G) \times \iota_2(G^\prime)$ coincide with $\kappa(X)$. Choose hypersurfaces $S_1, S_2, \dots, S_{k} \subset \IP^{r_1}$ such that $\iota_1(\overline{X})=S_1 \cap S_2 \cap \cdots \cap S_k$ as varieties (i.e., set-theoretically). As $X$ is irreducible, we can select a subset $\{ S_{k_1},\dots,S_{k_{\dim(G)-r}} \}$ of these hypersurfaces such that $\kappa(X_{\overline{\Gamma(n\cdot \phi_{\mathrm{tor}})}})$ is an irreducible component of 
\begin{equation*}
\kappa(G_{\overline{\Gamma(n \cdot \phi_{\mathrm{tor}})}}) \cap (S_{k_1} \times \IP^{r_2}) \cap \cdots \cap (S_{k_{\dim(G)-r}}\times \IP^{r_2}).
\end{equation*}
For reasons of dimension (cf.\ \cite[Lemma 7.1 (a)]{Fulton1998} and \cite[Example 8.2.1]{Fulton1998}), we have
\begin{equation}
\label{equation::intersectionproductpositivity}
\iota(\kappa(X_{\overline{\Gamma(n\cdot \phi_{\mathrm{tor}})}}), \kappa(G_{\overline{\Gamma(n\cdot \phi_{\mathrm{tor}})}}) \cdot (S_{k_1} \times \IP^{r_2}) \cdots (S_{k_{\dim(G)-r}}\times \IP^{r_2}); \IP^{r_1} \times \IP^{r_2}) \geq 1.
\end{equation}
It is well-known (compare \cite[Section 12.3]{Fulton1998}) that the tangent vector bundle $T (\IP^{r_1} \times \IP^{r_2}) = \pr_1^\ast (T \IP^{r_1}) \oplus \pr_2^\ast (T \IP^{r_2})$ is ample and hence globally generated. By \cite[Corollary 12.2 (a)]{Fulton1998}, every distinguished subvariety contributes a non-negative cycle to the intersection product in (\ref{equation::intersectionproductpositivity}). The degree of the $0$-cycle class (\ref{equation::intersectionnumber1}) is hence majorized by the degree of the $0$-cycle class
\begin{equation} \label{equation::intersectionnumber4}
c_1(\pr_1^\ast \mathcal{O}_{\IP^{r_1}}(1))c_1(\pr_2^\ast \mathcal{O}_{\IP^{r_2}}(1))^{r-1}
\\
\cap \kappa(G_{\overline{\Gamma(n\cdot \phi_{\mathrm{tor}})}}) \cdot (S_{k_1} \times \IP^{r_2}) \cdots (S_{k_{\dim(G)-r}}\times \IP^{r_2})
\end{equation}
on $\IP^{r_1} \times \IP^{r_2}$. The Chow ring $A_{\ast}(\IP^{r_1} \times \IP^{r_2})$ is of the form
\begin{equation*}
\IZ[H_1]/([H_1]^{r_1+1}) \otimes \IZ[H_2]/([H_2]^{r_2+1})
\end{equation*}
for any two hyperplanes $H_1 \subset \IP^{r_1}$ and $H_2 \subset \IP^{r_2}$ (see \cite[Example 8.3.7]{Fulton1998}). Thus, we may write $[S_i \times \IP^{r_2}] = d_i [H_1 \times \IP^{r_2}]$ and
\begin{equation*}
[\kappa(G_{\overline{\Gamma(n\cdot \phi_{\mathrm{tor}})}})]=\sum_{i_1+i_2 = r_1+r_2-\dim(G)} e_{i_1,i_2}[H_1^{i_1} \times H_2^{i_2}].
\end{equation*}
Furthermore, the definition of the first Chern class immediately implies that 
\begin{equation*}
c_1(\pr_i^\ast \mathcal{O}_{\IP^{r_i}}(1)) \cap [H_1^{i_1} \times H_2^{i_2}] =
\begin{cases}
 [H_1^{i_1+1} \times H_2^{i_2}]  & \text{if $i=1$} \\
 [H_1^{i_1} \times H_2^{i_2+1}]  & \text{if $i=2$}.
\end{cases}
\end{equation*}
With these notations, the degree of (\ref{equation::intersectionnumber4}) is
\begin{equation*}
d_{k_1}\cdots d_{k_{\dim(G)-r}}e_{r_1+r-\dim(G)-1, r_2-r+1}\leq \max_{\substack{K\subset \{ 1,\dots, k\} \\ |K|={\dim(G)-r}}} \left\lbrace \prod_{k \in K} d_k \right\rbrace \cdot e_{r_1+r-\dim(G)-1, r_2-r+1}.
\end{equation*}
Additionally, we have
\begin{equation*}
e_{r_1+r-\dim(G)-1, r_2-r+1} = \deg (c_1(\pr_1^\ast \mathcal{O}_{\IP^{r_1}}(1))^{\dim(G)+1-r}c_1(\pr_2^\ast \mathcal{O}_{\IP^{r_2}}(1))^{r-1} \cap [\kappa(G_{\overline{\Gamma(n\cdot \phi_{\mathrm{tor}})}})]),
\end{equation*}
which is less than
\begin{equation} \label{equation::intersectionnumber5}
c_{24}^{\dim(G)+1-r}l_2^{r-1} \deg ( c_1(\pr_1^\ast L)^{\dim(G)+1-r}c_1(\pr_2^\ast L_{0}^\prime)^{r-1} \cap [\iota(G_{\overline{\Gamma(n\cdot \phi_{\mathrm{tor}})}})])
\end{equation}
by the projection formula. To ease our exposition notationally, we write $\alpha_1=c_1(\pr_1^\ast M_{\overline{G}})$, $\alpha_2=c_1(\pr_2^\ast M_{\overline{G}^\prime})$, $\beta_1=c_1(\pr_1^\ast \overline{\pi}^\ast N)$, $\beta_2=c_1(\pr_2^\ast ((n\cdot \phi_{\mathrm{ab}}) \circ \overline{\pi}^\prime)^\ast N_j)$, $\beta_3=c_1(\pr_2^\ast (\overline{\pi}^\prime)^\ast N)$ and $r^\prime=\dim(G)+1-r$ in the following computations. Then,
\begin{align*}
c_1(\pr_1^\ast L)^{r^\prime}c_1(\pr_2^\ast L_{0}^\prime)^{r-1}
&= (\alpha_1 + \beta_1)^{r^\prime} (\alpha_2 + \beta_2 + \beta_3)^{r-1} \\
&= 
\left( \sum_{s_1=0}^{r^\prime} \binom{r^\prime}{s_1} \alpha_1^{s_1
} \beta_1^{r^\prime-s_1} \right)
\left( \sum_{s_2=0}^{r-1} \binom{r-1}{s_2} \alpha_2^{s_2}(\beta_2 + \beta_3)^{r-1-s_2} \right).
\end{align*}
For each positive integer $n$, the isogeny $[n]_G: G \rightarrow G$ of degree $n^{t+2\dim(A)}$ extends to a proper map $[n]_{G_{\overline{\Gamma(n\cdot \phi_{\mathrm{tor}})}}}: G_{\overline{\Gamma(n\cdot \phi_{\mathrm{tor}})}} \rightarrow G_{\overline{\Gamma(n\cdot \phi_{\mathrm{tor}})}}$ such that $([n]_{G_{\overline{\Gamma(n\cdot \phi_{\mathrm{tor}})}}})_\ast [G_{\overline{\Gamma(n\cdot \phi_{\mathrm{tor}})}}] = n^{t+2\dim(A)} [G_{\overline{\Gamma(n\cdot \phi_{\mathrm{tor}})}}]$. Furthermore, pulling back the line bundles $\pr_1^\ast M_{\overline{G}}$ and $\pr_2^\ast M_{\overline{G}^\prime}$ (resp.\ $\pr_1^\ast \overline{\pi}^\ast N$, $\pr_2^\ast((n\cdot \phi_{\mathrm{ab}}) \circ \overline{\pi}^\prime)^\ast N_j$ and $\pr_2^\ast(\overline{\pi}^\prime)^\ast N$) along $[n]_{G_{\overline{\Gamma(n\cdot \phi_{\mathrm{tor}})}}}$ amounts to rising them to the $n$-th (resp.\ $n^2$-th) power. Therefore, the projection formula (applied to $[n]_{G_{\overline{\Gamma(n\cdot \phi_{\mathrm{tor}})}}}$) yields that
\begin{equation} \label{equation::intersectionnumber2}
n^{2\dim(G)-s_1 - s_2} \deg(\iota^\ast (\alpha_1^{s_1}\alpha_2^{s_2}\beta_1^{r^\prime - s_1}(\beta_2+\beta_3)^{r-1-s_2}) \cap [G_{\overline{\Gamma(n\cdot \phi_{\mathrm{tor}})}}])
\end{equation}
is the same as
\begin{multline*}
\deg(\iota^\ast (\alpha_1^{s_1}\alpha_2^{s_2}\beta_1^{r^\prime - s_1}(\beta_2+\beta_3)^{r-1-s_2}) \cap ([n]_{G_{\overline{\Gamma(n\cdot \phi_{\mathrm{tor}})}}})_\ast [G_{\overline{\Gamma(n\cdot \phi_{\mathrm{tor}})}}]) \\
= n^{t+2\dim(A)} \deg(\iota^\ast (\alpha_1^{s_1}\alpha_2^{s_2}\beta_1^{r^\prime - s_1}(\beta_2+\beta_3)^{r-1-s_2}) \cap ([G_{\overline{\Gamma(n\cdot \phi_{\mathrm{tor}})}}])
\end{multline*}
It follows that (\ref{equation::intersectionnumber2}) is zero whenever $s_1+s_2 \neq t$. Hence, the quantity (\ref{equation::intersectionnumber5}) can be rewritten as
\begin{equation*}
c_{24}^{\dim(G)+1-r}l_2^{r-1}\sum_{s=\max\{0,t-r+1\}}^{\min \{ r^\prime, t\}} \binom{r^\prime}{s} \binom{r-1}{t-s} \deg( \alpha_1^{s} \alpha_2^{t-s} \beta_1^{r^\prime-s}(\beta_2+\beta_3)^{\dim(A)-(r^\prime-s)} \cap [G_{\overline{\Gamma(n\cdot \phi_{\mathrm{tor}})}}]).
\end{equation*}  
(Note that $(r^\prime-s) + (r-1-t+s)=\dim(G)-t=\dim(A)$.) Taking into account our previous reductions, it is sufficient to show that each
\begin{equation}
\label{equation::intersectionnumber8}
\deg( \alpha_1^{s} \alpha_2^{t-s} \beta_1^{r^\prime-s}(\beta_2+\beta_3)^{\dim(A)-(r^\prime-s)} \cap [G_{\overline{\Gamma(n\cdot \phi_{\mathrm{tor}})}}])
\end{equation}
is bounded from above by $c_{25}n^{2r-2}$ for some constant $c_{25}$. 

As $\pi_{\overline{\Gamma(\phi_{\mathrm{tor}})}}: G_{\overline{\Gamma(n\cdot \phi_{\mathrm{tor}})}} \rightarrow A$ exhibits $G_{\overline{\Gamma(n\cdot \phi_{\mathrm{tor}})}}$ as a $\overline{\Gamma(n \cdot \phi_{\mathrm{tor}})}$-bundle over $A$, it is flat of relative dimension $t$. We can therefore pull back cycle classes on $A$ to cycle classes on $G_{\overline{\Gamma(n\cdot \phi_{\mathrm{tor}})}}$. In particular, we have $\pi_{\overline{\Gamma(n\cdot \phi_{\mathrm{tor}})}}^\ast ([A]) =[G_{\overline{\Gamma(n\cdot \phi_{\mathrm{tor}})}}]$ and $\pi_{\overline{\Gamma(n\cdot \phi_{\mathrm{tor}})}}^\ast ([p])=[\pi_{\overline{\Gamma(n\cdot \phi_{\mathrm{tor}})}}^{-1}(p)]$ for any point $p \in A$. Setting
\begin{equation*}
\sigma_s= c_1(N)^{r^\prime-s}(c_1((n\cdot \phi_{\mathrm{ab}})^\ast N_j)+c_1(N))^{\dim(A)-(r^\prime-s)} \cap [A] \in A_0(A),
\end{equation*}
we know from \cite[Proposition 2.5 (d)]{Fulton1998} that there exist points $p_1,\dots,p_{\deg(\sigma_s)+m},q_1,\dots,q_{m} \in A$ such that
\begin{equation*}
\beta_1^{r^\prime-s}(\beta_2+\beta_3)^{\dim(A)-(r^\prime-s)} \cap [G_{\overline{\Gamma(n\cdot \phi_{\mathrm{tor}})}}] = \sum_{l=1}^{\deg(\sigma_s)+m} [\pi_{\overline{\Gamma(n\cdot \phi_{\mathrm{tor}})}}^{-1}(p_l)]-\sum_{l=1}^{m} [\pi_{\overline{\Gamma(n\cdot \phi_{\mathrm{tor}})}}^{-1}(q_l)].
\end{equation*}
By construction, there exists a non-canonical isomorphism between each fiber $\pi_{\overline{\Gamma(n\cdot \phi_{\mathrm{tor}})}}^{-1}(x)$, $x \in A$, and $\overline{\Gamma(n\cdot \phi_{\mathrm{tor}})} \subset (\IP^{1})^t \times (\IP^1)^{t^\prime}$ such that the restrictions of $\pr_1^\ast M_{\overline{G}}$ and $\pr_2^\ast M_{\overline{G}^\prime}$ to $\iota(\pi_{\overline{\Gamma(n\cdot \phi_{\mathrm{tor}})}}^{-1}(x))$ correspond to the line bundles $\pr_1^\ast M_{t}|_{\overline{\Gamma(n\cdot \phi_{\mathrm{tor}})}}$ and $\pr_2^\ast M_{t^\prime}|_{\overline{\Gamma(n\cdot \phi_{\mathrm{tor}})}}$. Once again, we apply the projection formula to obtain
\begin{equation} \label{equation::intersectionnumber7}
\deg(\alpha_1^{s} \alpha_2^{t-s} \cap [\pi_{\overline{\Gamma(n\cdot \phi_{\mathrm{tor}})}}^{-1}(p_l)]) = \deg(c_1(\pr_1^\ast M_{t})^{s} c_1(\pr_2^\ast M_{t^\prime})^{t-s} \cap [\overline{\Gamma(n \cdot\phi_{\mathrm{tor}})}]).
\end{equation}
In particular, (\ref{equation::intersectionnumber8}) is bounded by
\begin{equation*}
\deg(\sigma_s) \deg(\alpha_1^{s} \alpha_2^{t-s} \cap [\pi_{\overline{\Gamma(n\cdot \phi_{\mathrm{tor}})}}^{-1}(e_A)]).
\end{equation*}
We first show that $\alpha_1^{s} \alpha_2^{t-s} \cap [\pi_{\overline{\Gamma(n\cdot \phi_{\mathrm{tor}})}}^{-1}(e_A)]$ has degree less than $c_{26}n^{t-s}$ for some constant $c_{26}$. Using standard coordinates $X_1,\dots, X_{t}$ (resp.\ $Y_1,\dots, Y_{t^\prime}$) on $\Gm^{t}$ (resp.\ $\Gm^{t^\prime}$), let us write 
\begin{equation*}
(n\cdot \phi_{\mathrm{tor}})^\ast(Y_v)= X_1^{a_{1v}}\cdots X_{t}^{a_{tv}}
\end{equation*}
with integers $a_{uv}$ ($1 \leq u \leq t, 1 \leq v \leq t^\prime$). By dimension, we have again
\begin{equation*}
\iota(\overline{\Gamma(n \cdot \phi_{\mathrm{tor}})}, (Y_1 = X_1^{a_{11}} \cdots X_{t}^{a_{1t}}) \cdots (Y_{t^\prime} = X_{1}^{a_{1t^\prime}} \cdots X_{t}^{a_{tt^\prime}}); (\IP^1)^{t} \times (\IP^1)^{t^\prime})\geq 1.
\end{equation*}
We determine next the intersection product
\begin{equation} \label{equation::intersectionproduct}
(Y_1 = X_1^{a_{11}} \cdots X_{t}^{a_{1t}}) \cdots (Y_{t^\prime} = X_{1}^{a_{1t^\prime}} \cdots X_{t}^{a_{tt^\prime}}) \in A_{t}((\IP^1)^{t} \times (\IP^1)^{t^\prime}). 
\end{equation}
From \cite[Example 8.3.7]{Fulton1998}, we deduce an identification
\begin{equation*}
A_{\ast}((\IP^{1})^{t} \times (\IP^{1})^{t^\prime})=\IZ[\varepsilon_1,\dots, \varepsilon_{t}, \varepsilon_{1}^\prime, \dots, \varepsilon_{t^\prime}^\prime]/((\varepsilon_1)^2,\dots, (\varepsilon_{t})^2,(\varepsilon^\prime_1)^2,\dots, (\varepsilon^\prime_{t^\prime})^2)
\end{equation*}
such that $\varepsilon_i$ (resp.\ $\varepsilon_i^\prime$) corresponds to the flat pullback of the cycle class associated with an arbitrary point in the $i$-th factor of $(\IP^{1})^{t}$ (resp.\ $(\IP^1)^{t^\prime}$). Considering appropriate intersections, it is easy to verify
\begin{equation*}
[Y_v=X_1^{a_{1v}}\cdots X_{t}^{a_{tv}}]= |a_{1v}|\varepsilon_1 + |a_{2v}|\varepsilon_2 + \cdots + |a_{tv}|\varepsilon_{t} +  \varepsilon^\prime_v.
\end{equation*}
Thus, (\ref{equation::intersectionproduct}) is simply the product
\begin{equation*}
\prod_{1\leq v \leq t^\prime} ( |a_{1v}|\varepsilon_1 + |a_{2v}|\varepsilon_2 + \cdots + |a_{tv}|\varepsilon_{t} +  \varepsilon^\prime_v).
\end{equation*}
Inspecting the definition of $M_{t}$ (resp.\ $M_{t^\prime}$) in Construction \ref{construction0}, we note that intersecting a cycle class on $(\IP^{1})^{t} \times (\IP^{1})^{t^\prime}$ with $c_1(\pr_1^\ast M_{t})$ (resp.\ $c_1(\pr_2^\ast M_{t^\prime})$) amounts to multiplication with $2(\varepsilon_1+\cdots+\varepsilon_{t})$ (resp.\ $2(\varepsilon_1^\prime+\cdots+\varepsilon_{t^\prime}^\prime)$) in the Chow ring. We infer that the degree of (\ref{equation::intersectionnumber7}) is majorized by the degree of
\begin{equation*}
2^{t}(\varepsilon_1+\cdots+\varepsilon_{t})^{s}(\varepsilon_1^\prime+\cdots+\varepsilon_{t^\prime}^\prime)^{t-s}
\prod_{1\leq v \leq t^\prime} ( |a_{1v}|\varepsilon_1+\dots + |a_{tv}|\varepsilon_{t} + \varepsilon_v^\prime).
\end{equation*}
Exploiting cancellations, this can be simplified to 
\begin{equation*}
2^{t} s! (t-s)! \cdot \sum_{\substack{ 1 \leq u_1,\dots,u_{t-s} \leq t \text{ distinct} \\ 1 \leq v_1 < \dots < v_{t-s} \leq t^\prime }}{|a_{u_1v_1}a_{u_2v_2}\cdots a_{u_{t-s}v_{t-s}}|} (\varepsilon_1\cdots \varepsilon_{t}\varepsilon^{\prime}_1\cdots \varepsilon^{\prime}_{t^\prime})
\end{equation*}
Since $(\phi_{\mathrm{tor}},\phi_{\mathrm{ab}}) \in \mathcal{K}_\delta$, (\ref{equation::intersectionnumber7}) can be consequently bounded from above by $c_{26}n^{t-s}$ as claimed.

We finally demonstrate that $\deg(\sigma_s)$ is bounded from above by $c_{27}n^{2(\dim(A)-(r^\prime-s))}$ for some constant $c_{27}$. For this, it suffices to note that $\Hom(A,A_j^\prime)$ is a finitely generated $\IZ$-module and that 
\begin{equation*}
\Hom(A,A_j^\prime) \longrightarrow \Pic(A), \varphi_{\mathrm{ab}} \longmapsto \varphi_{\mathrm{ab}}^\ast N_j,
\end{equation*}
is quadratic by the Theorem of the Cube (\cite[Corollary II.6.2]{Mumford1970}) because $N_j$ is symmetric (see \cite[p.\ 417]{Habegger2009a} for details). Combining this with the previous estimate, we immediately obtain the bound
\begin{equation*}
c_{26} c_{27} n^{t-s+2(\dim(A)-(r^\prime-s))}\leq c_{26} c_{27} n^{2(t-s+\dim(A)-r^\prime+s)} = c_{26} c_{27} n^{2r-2}
\end{equation*}
on (\ref{equation::intersectionnumber8}). Taking our previous reductions into account, this completes the proof of the lemma.
\end{proof}

\section{Quotients of Semiabelian Varieties}
\label{section::structuralproperties}

In this section, we elucidate the set of quotients belonging to a fixed semiabelian variety. Let $G$ be a semiabelian variety over $\IQbar$ with split toric part $\Gm^{t}$ and abelian quotient $\pi: G \rightarrow A$. For a fixed torus $\Gm^{t^\prime}$ and a fixed abelian variety $A^\prime$, we ask which elements $(\phi_{\mathrm{tor}}, \phi_{\mathrm{ab}})$ of
\begin{equation*}
V_\IQ = \Hom_\IQ(\Gm^{t},\Gm^{t^\prime}) \times \Hom_\IQ(A,A^\prime)
\end{equation*}
are such that there exists a quasi-homomorphism $\phi: G \rightarrow G^\prime$ represented by $(\phi_{\mathrm{tor}}, \phi_{\mathrm{ab}})$ in the sense of Section \ref{section::homomorphisms}. Let $Z(\IQ) \subset V_\IQ$ denote the subset consisting of these elements. For a fixed semiabelian variety $G^\prime$ with toric part $\Gm^{t^\prime}$ and abelian quotient $A^\prime$, we know from Lemma \ref{lemma::semiabelian1} that the surjective quasi-homomorphisms $\phi: G \rightarrow G^\prime$ are parameterized by a linear subspace of $V_\IQ$. The set $Z(\IQ)$ is the union of all these linear subspaces for varying $G^\prime$. It is, however, not a union of finitely many linear subspaces in general. Nevertheless, we can interpret $V_\IQ$ as the $\IQ$-points of an additive algebraic group, which we abusively denote also by $V_\IQ$, and ask whether there is an algebraic subvariety $Z\subset V_\IQ$ with $Z(\IQ)$ as its set of $\IQ$-points. This would also motivate our notation $Z(\IQ)$ retroactively. In the next theorem, a cone $Z \subset V_{\IQ}$ is a (not necessarily closed) algebraic subvariety of $V_{\IQ}$ such that $[n]_{V_\IQ} (Z) \subseteq Z$ for any non-zero integer $n$.

\begin{theorem} \label{theorem::realizable} There exists a cone $Z \subset V_\IQ$ such that its $\IQ$-points are precisely the pairs $(\phi_{\mathrm{tor}},\phi_{\mathrm{ab}}) \in V_\IQ$ representing quasi-homomorphisms.
\end{theorem}

In the following, pairs $(\phi_{\mathrm{tor}},\phi_{\mathrm{ab}}) \in V_\IQ$ representing quasi-homomorphisms are called realizable.

\begin{proof} Write $\eta_{G}=(\eta_{G}^{(1)},\dots,\eta_{G}^{(t)})\in A^\vee(\IQbar)^{t}$. By Lemma \ref{lemma::semiabelian1}, a pair $(\phi_{\mathrm{tor}},\phi_{\mathrm{ab}}) \in V_\IQ$ is realizable if and only if for one of its multiples $n \cdot (\phi_{\mathrm{tor}}, \phi_{\mathrm{ab}}) \in V$ there exists some $\mu= (\mu_1,\dots,\mu_{t^\prime}) \in (A^\prime)^\vee(\IQbar)^{t^\prime}$ such that
\begin{equation} \label{equation::eta_equation}
(n \cdot \phi_{\mathrm{tor}})_\ast(\eta_{G}^{(1)},\dots,\eta_{G}^{(t)}) = (n \cdot \phi_{\mathrm{tor}})_\ast \eta_{G} = (n \cdot\phi_{\mathrm{ab}})^\ast \mu = ((n \cdot\phi_{\mathrm{ab}})^\vee \mu_1, \dots, (n \cdot\phi_{\mathrm{ab}})^\vee \mu_{t^\prime})
\end{equation}
in $A^\vee(\IQbar)^{t^\prime}$. Write
\begin{equation*}
\begin{pmatrix}
 a_{11}   & \cdots  & a_{1t^\prime}   \\
 \vdots   &   & \vdots \\
 a_{t1} & \cdots  & a_{tt^\prime}  
\end{pmatrix}\!, \, a_{uv} \in n^{-1}\IZ,
\end{equation*}
for the matrix representing $\phi_{\mathrm{tor}} \in \Hom_\IQ(\Gm^t, \Gm^{t^\prime})$ and let $\phi_{v,\mathrm{tor}}: \Gm^{t} \rightarrow_\IQ \Gm$, $1 \leq v \leq t^\prime$, be the quasi-homomorphism described by the column vector $(a_{1v},\dots, a_{tv})^t$.
Then, (\ref{equation::eta_equation}) is equivalent to the equations
\begin{equation*}
(n \cdot \phi_{v,\mathrm{tor}})_\ast \eta_{G}= n a_{1v}\cdot \eta_{G}^{(1)}+\dots+n a_{tv}\cdot \eta_{G}^{(t)}= (n \cdot \phi_{\mathrm{ab}})^\vee \mu_v, \, 1 \leq v \leq t^\prime,
\end{equation*}
having solutions $\mu_v \in A^\vee(\IQbar)$. Hence, the pair $(\phi_{\mathrm{tor}},\phi_{\mathrm{ab}})$ is realizable if and only if each $(\phi_{v,\mathrm{tor}},\phi_{\mathrm{ab}})$, $1 \leq v \leq t^\prime$, represents a quasi-homomorphism $G \rightarrow G_v^{\prime}$. Assume that there are cones $Z_v \subset \Hom(\Gm^{t},\Gm) \times \Hom(A,A^\prime) = V_{v,\IQ}$, $1 \leq v \leq t^\prime$, with $Z_v(\IQ)$ consisting of the pairs in $V_{v,\IQ}$ representing quasi-homomorphisms. Denoting by $p_v: V_\IQ \rightarrow V_{v,\IQ}$ the standard projection, the cone $Z = \bigcap_{v=1}^{t^\prime} p_v^{-1}(Z_v) \subset V_\IQ$ is as wanted. In conclusion, it suffices to show the assertion for $t^\prime = 1$.

Choose pairwise non-isogeneous simple abelian varieties $B_1,\dots,B_{k}$ such that there exist isogenies
\begin{equation*}
\chi: A \longrightarrow B_1^{r_1} \times \cdots \times B_k^{r_k} \text{ and } \chi^\prime: A^\prime \longrightarrow B_1^{r_1^\prime} \times \cdots B_k^{r_k^\prime}
\end{equation*}
and set 
\begin{equation*}
W = \Hom(\Gm^{t},\Gm) \times \Hom(B_1^{r_1} \times \cdots \times B_k^{r_k},B_1^{r_1^\prime} \times \cdots B_k^{r_k^\prime}).
\end{equation*}
Let further $\psi$ (resp.\ $\psi^\prime$) be isogenies such that $\psi \circ \chi = \chi \circ \psi = [m]_{A}$ (resp. $\psi^\prime \circ \chi^\prime = \chi^\prime \circ \psi^\prime = [m]_{A^\prime}$) for some integer $m \geq 1$. We have $\IQ$-linear maps 
\begin{equation*}
f: V_\IQ \longrightarrow W_\IQ, \, (\phi_{\mathrm{tor}}, \phi_{\mathrm{ab}}) \longmapsto (\phi_{\mathrm{tor}}, \chi^\prime \circ \phi_{\mathrm{ab}} \circ \psi)
\end{equation*}
and
\begin{equation*}
g: W_\IQ \longrightarrow V_\IQ, \, (\phi_{\mathrm{tor}}, \phi_{\mathrm{ab}}) \longmapsto (\phi_{\mathrm{tor}}, \psi^\prime \circ \phi_{\mathrm{ab}} \circ \chi)
\end{equation*}
such that both $g \circ f: V_\IQ \rightarrow V_\IQ$ and $f \circ g: W_\IQ \rightarrow W_\IQ$ send $(\phi_{\mathrm{tor}},\phi_{\mathrm{ab}})$ to $(\phi_{\mathrm{tor}}, m^2 \cdot \phi_{\mathrm{ab}})$. Hence, $f$ and $g$ are bijections between $V_\IQ$ and $W_\IQ$. Using Lemma \ref{lemma::semiabelian3}, we additionally deduce that both $f$ and $g$ preserve realizable pairs. Consequently, we may assume that $A=B_1^{r_1} \times \cdots \times B_k^{r_k}$ and $A^\prime=B_1^{r_1^\prime} \times \cdots \times B_k^{r_k^\prime}$ in proving the theorem. In this case, we can also identify
\begin{equation*}
V_\IQ = \Hom_\IQ(\Gm^{t},\Gm) \times \prod_{i=1}^{k} \Hom_\IQ(B_i^{r_i}, B_i^{r_i^\prime}).
\end{equation*}
By Lemma \ref{lemma::semiabelian1}, an element $(\phi_{\mathrm{tor}},\phi_{\mathrm{ab}}^{(1)},\dots, \phi_{\mathrm{ab}}^{(k)}) \in V_\IQ$, $\phi_{\mathrm{ab}}^{(i)} \in \Hom_\IQ(B_i^{r_i},B_i^{r_i^\prime})$, is realized by a quasi-homomorphism if and only if for one of its multiples $n \cdot (\phi_{\mathrm{tor}}, \phi_{\mathrm{ab}}^{(1)},\dots, \phi_{\mathrm{ab}}^{(k)}) \in V$ there exists some tuple $(\eta^{(1)},\dots,\eta^{(k)}) \in B_1^\vee(\IQbar)^{r_1^\prime} \times \dots \times B_k^\vee(\IQbar)^{r_k^\prime}$ such that
\begin{equation*}
(n \cdot \phi_{\mathrm{tor}})_\ast \eta_{G} = (n \cdot \phi_{\mathrm{ab}}^{(1)},\dots, n \cdot \phi_{\mathrm{ab}}^{(k)})^\ast (\eta^{(1)},\dots, \eta^{(k)}) = ((n \cdot \phi_{\mathrm{ab}}^{(1)})^\vee (\eta^{(1)}),\dots, (n \cdot \phi_{\mathrm{ab}}^{(k)})^\vee(\eta^{(k)})).
\end{equation*}
Arguing as above, we deduce that it suffices to prove the theorem under the additional assumption that $k=1$ (i.e., $A=B^r$ and $A^\prime=B^{r^\prime}$ with a simple abelian variety $B$).

Let us write $\eta_G = (\underline{\eta}_1,\dots,\underline{\eta}_{t})\in (B^r)^\vee(\IQbar)^{t}= \Ext_{\IQbar}^1(B^r,\Gm^{t})$ and $\underline{\eta}_j=(\eta_{1j},\dots, \eta_{rj})^t \in (B^\vee)(\IQbar)^r=(B^r)^\vee(\IQbar)$. Again, $(\phi_{\mathrm{tor}},\phi_{\mathrm{ab}}) \in V_\IQ$ is realizable if and only if there exists some multiple $n \cdot (\phi_{\mathrm{tor}}, \phi_{\mathrm{ab}}) \in V$ such that 
\begin{equation}
\label{equation::etacondition}
(n\cdot \phi_{\mathrm{tor}})_\ast \eta_{G} = (n \cdot \phi_{\mathrm{ab}})^\ast \mu
\end{equation}
has a solution $\mu = (\mu_1, \dots, \mu_{r^\prime}) \in B^\vee(\IQbar)^{r^\prime} = \Ext^1_{\IQbar}(A^\prime, \Gm)$. This condition can be translated into linear algebra over the $\IQ$-division algebra $D=\End(B^\vee)_\IQ$ (cf.\ \cite[Corollary 2 on p.\ 174]{Mumford1970}). For this, we denote by $\Gamma$ the left $\End(B^\vee)$-submodule of $B^\vee(\IQbar)$ generated by 
\begin{equation*}
\eta_{ij}, \, 1\leq i \leq r, \, 1 \leq j \leq t.
\end{equation*}
The tensor product $\Gamma_\IQ = \Gamma \otimes_\IZ \IQ$ is a left $D$-submodule of $B^\vee(\IQbar) \otimes_\IZ \IQ$. For any $\gamma \in B^\vee(\IQbar)$, we let $[\gamma]$ denote $\gamma \otimes 1 \in B^\vee(\IQbar) \otimes_\IZ \IQ$. As $D$ is a division ring, $\Gamma_\IQ$ is a free left $D$-module so that we may choose $\End(B^\vee)$-linearly independent elements $\gamma_1,\dots,\gamma_l \in B^\vee(\IQbar)$ satisfying
\begin{equation}
\label{equation::Gammadecomposition}
\Gamma = \End(B^\vee) \cdot \gamma_1 \oplus \cdots \oplus \End(B^\vee) \cdot \gamma_l \oplus \Tors(\Gamma);
\end{equation}
here $\Tors(\Gamma)$ denotes the $\IZ$-torsion elements of $\Gamma$. If (\ref{equation::etacondition}) has a solution $\mu = (\mu_1,\dots,\mu_{r^\prime}) \in B^\vee(\IQbar)^{r^\prime}$ for some $n$, then it also has a solution $\mu \in \Gamma^{r^\prime} \subset B^\vee(\IQbar)^{r^\prime}$ for a possibly larger $n$. In fact, one may take any image under a $D$-linear projection from $\Gamma_\IQ + D \cdot [\mu_1] + \dots + D \cdot [\mu_{r^\prime}]$ to $\Gamma_\IQ$. Since we can always arrange for $n$ to annihilate the finite group $\Tors(\Gamma)$, we infer that $(\phi_{\mathrm{tor}},\phi_{\mathrm{ab}})$ is realizable if and only if, in the notation from Section \ref{section::homomorphisms},
\begin{equation*}
(\phi_{\mathrm{tor}})_{\ast,\IQ} (\eta_G) = (\phi_{\mathrm{ab}})^{\ast,\IQ} (\mu)
\end{equation*}
has a solution $\mu \in \Gamma_\IQ^{r^\prime} \subset (B^\vee(\IQbar) \otimes_\IZ \IQ)^{r^\prime}$. Both $\phi_{\mathrm{tor}} \in \Hom_\IQ(\Gm^t,\Gm)$ and $\phi_{\mathrm{ab}}^\vee \in \Hom_\IQ((A^\prime)^\vee, A^\vee)$ can be represented by matrices
\begin{equation}
\label{equation::acoefficients}
(a_1, a_2, \cdots, a_{t})^t, \, a_i \in \IQ,
\end{equation}
and 
\begin{equation}
\label{equation::bcoefficients}
\begin{pmatrix}
b_{11} & \cdots & b_{1r^\prime} \\
\vdots & & \vdots \\
b_{r1} & \cdots & b_{rr^\prime}
\end{pmatrix}
, \, b_{ij} \in D.
\end{equation}
Using this notation, we are searching for (\ref{equation::acoefficients}) and (\ref{equation::bcoefficients}) such that
\begin{equation}
\label{equation::etaconditionQ}
a_1 \cdot
\begin{pmatrix}
[\eta_{11}] \\
\vdots \\
[\eta_{r1}]
\end{pmatrix}
+ \dots +
a_t \cdot
\begin{pmatrix}
[\eta_{1t}] \\
\vdots \\
[\eta_{rt}]
\end{pmatrix}
=
\begin{pmatrix}
b_{11} & \cdots & b_{1r^\prime} \\
\vdots & & \vdots \\
b_{r1} & \cdots & b_{rr^\prime}
\end{pmatrix}
\begin{pmatrix}
[\mu_1] \\
\vdots \\
[\mu_{r^\prime}]
\end{pmatrix}
\end{equation}
has a solution $([\mu_1],\dots, [\mu_{r^\prime}])\in \Gamma_\IQ$. Using the decomposition (\ref{equation::Gammadecomposition}), we expand
\begin{equation*}
[\eta_{ij}]= c_{ij}^{(1)} [\gamma_1] + \dots + c_{(ij)}^{(l)} [\gamma_l], \, c_{ij}^{(\cdot)} \in D.
\end{equation*}
Then, (\ref{equation::etaconditionQ}) has a solution $([\mu_1], \dots, [\mu_{r^\prime}])$ if and only if each of the $l$ linear equations
\begin{equation}
\label{equation::etaconditionQ2}
a_1 \cdot
\begin{pmatrix}
c_{11}^{(\cdot)} \\
\vdots \\
c_{r1}^{(\cdot)}
\end{pmatrix}
+
\cdots
+
a_t \cdot
\begin{pmatrix}
c_{1t}^{(\cdot)} \\
\vdots \\
c_{rt}^{(\cdot)}
\end{pmatrix}
=
\begin{pmatrix}
b_{11} & \cdots & b_{1r^\prime} \\
\vdots & & \vdots \\
b_{r1} & \cdots & b_{rr^\prime}
\end{pmatrix}
\begin{pmatrix}
\delta_1^{(\cdot)} \\
\vdots \\
\delta_{r^\prime}^{(\cdot)}
\end{pmatrix}
\end{equation}
has a solution $(\delta_1^{(\cdot)},\dots,\delta_{r^\prime}^{(\cdot)})\in D^r$. By Lemma \ref{lemma::linearequation} below, the corresponding condition on (83) and (84) is described by a subcone of $\IQ^t \times D^{r \times r^\prime}$. The intersection $Z^\vee$ of these $l$ cones is almost what we are searching for. In fact, a pair $(\phi_t, \phi_a) \in V_\IQ$ is realizable if and only if $(\phi_t,\phi_a^\vee) \in \IQ^t \times D^{r \times r^\prime}$ is in $Z^\vee(\IQ)$. The theorem follows now from the $\IQ$-linearity (cf.\ \cite[(ii) on p.\ 75]{Mumford1970}) of 
\begin{equation*}
\Hom_\IQ(A,A^\prime) \longrightarrow \Hom_\IQ((A^\vee)^\prime,A^\vee), f \longmapsto f^\vee.
\end{equation*}
\end{proof}

The following lemma is certainly standard (for $t=1$ and $a_1=1$ at least) but I have found no trace of it in the literature so that a complete proof is given.

\begin{lemma} 
\label{lemma::linearequation}
Let $D$ be a finite-dimensional $\IQ$-algebra and $\underline{y}_1, \dots, \underline{y}_t \in D^{r}$ column vectors. Then, the pairs $(\underline{a}, M) \in \IQ^t \times D^{r \times r^\prime}$, $\underline{a}=(a_1,\dots,a_t)$, such that 
\begin{equation}
\label{equation::lemmalinearequation}
a_1\underline{y}_1 + \dots + a_t\underline{y}_t = M \cdot \underline{x}
\end{equation}
has a solution $\underline{x} \in D^{r^\prime}$, are the $\IQ$-points of a cone $Z \subset \IQ^t \times D^{r \times r^\prime}$.
\end{lemma}

Here, $\IQ^t \times D^{r \times r^\prime}$ is given its canonical structure as an affine linear space over $\IQ$. We also remark that $Z$ is generally not a closed subvariety.

\begin{proof} Choosing a $\IQ$-linear isomorphism $\varphi: D \rightarrow \IQ^n$, we obtain a map $l: D \hookrightarrow \IQ^{n\times n}$ such that $l(d_1) \varphi(d_2) = \varphi(d_1 d_2)$. This realizes $D$ as a $n$-dimensional subspace $l(D)$ of $\IQ^{n\times n}$. With these identifications, the equation (\ref{equation::lemmalinearequation}) can be written as 
\begin{equation*}
a_1\underline{y}_1^\prime + \dots + a_t\underline{y}_t^\prime =
M^\prime \cdot \underline{x}^\prime
\end{equation*}
with $M^\prime \in \IQ^{nr \times nr^\prime}$, $\underline{x}^\prime\in \IQ^{nr^\prime}$ and $\underline{y}^\prime_1,\dots,\underline{y}^\prime_t \in \IQ^{nr}$. We then search for $M^\prime \in \IQ^{nr \times nr^\prime}$ such that a solution $\underline{x}^\prime$ exists under the additional restraint that $M^\prime$ comes from a matrix $M \in D^{r \times r^\prime}$ by applying $l$ to each entry. Since this restraint can be evidently expressed as $M^\prime$ being contained in a $\IQ$-subcone of $\IQ^{nr \times nr^\prime}$, we can restrict to the case $D=\IQ$.

To deal with this special case, we make the following elementary observation: Write $M=(\underline{m}_1 \dots \underline{m}_{r^\prime})$ with column vectors $\underline{m}_i \in \IQ^r$. For any $\underline{y} \neq 0$, we have $\underline{y} \in \im (M)$ if and only if there exists a subset $I \subset \{ 1, \dots, r^\prime \}$ such that $\bigwedge_{i \in I} \underline{m}_i \neq 0 \in \bigwedge^{|I|} \IQ^r$ and $\bigwedge_{i \in I} \underline{m}_i \wedge \underline{y} = 0 \in \bigwedge^{|I|+1} \IQ^r$. From this, we straightforwardly obtain equations for the sought-after $\IQ$-cone $Z \subset \IQ^t \times \IQ^{r \times r^\prime}$.
\end{proof}

The proof of Theorem \ref{theorem::realizable} gives evidently a procedure to determine $Z$ via linear algebra so that one may hope that its rational points $Z(\IQ)$ are equally easy to describe. However, $Z(\IQ)$ can be rather complicated if $G$ is neither an abelian variety nor a torus. For example, $Z$ is not even rational in general, although it is in these two special cases. With respect to the proof of Theorem \ref{theorem::main}, this means in particular that (Dirichlet) approximation arguments as in \cite[Section 4]{Habegger2009a} and \cite[Section 4]{Habegger2009} break down if one insists on the use of surjective quasi-homomorphisms $G \rightarrow G^\prime$. This makes it necessary to work with explicit line bundles on $G$ as we do in this article.
 
\begin{example} \label{example::nonrational}
Let $E$ be an elliptic curve without complex multiplication (i.e., $\End(E)=\IZ$). Furthermore, let $\gamma_1, \gamma_2, \gamma_3 \in E^\vee(\IQbar)$ be such that $\Gamma = \sum_{i=1}^3 \IZ \cdot \gamma_i$ is a free $\IZ$-module of rank $3$. For an arbitrary tuple $(n_1,n_2,n_3) \in \IZ^3$, we define
\begin{equation*}
\underline{\eta}_1 = \begin{pmatrix} n_1 \cdot \gamma_1 \\ \gamma_3 \\ \gamma_2 \end{pmatrix}, \underline{\eta}_2 = \begin{pmatrix} n_2 \cdot \gamma_2 \\ \gamma_1 \\ \gamma_3 \end{pmatrix}, \underline{\eta}_3 = \begin{pmatrix} n_3 \cdot \gamma_3 \\ \gamma_2 \\ \gamma_1 \end{pmatrix},
\end{equation*}
considering these column vectors as elements of $(E^3)^\vee(\IQbar)$. Let $G$ be the semiabelian variety determined by 
\begin{equation*}
(\underline{\eta}_1,\underline{\eta}_2,\underline{\eta}_3) \in ((E^3)^\vee)^3 = \Ext^1(E^3,\Gm)^3= \Ext^1(E^3,\Gm^3).
\end{equation*}
From Theorem \ref{theorem::realizable}, we know that the realizable pairs in
\begin{equation*}
V_\IQ = \Hom(\Gm^{3},\Gm)_\IQ \times \Hom(E^3,E^2)_\IQ
\end{equation*}
are the $\IQ$-rational points of an algebraic subvariety $Z \subset V_\IQ$. Consider the projection $\pi: V_\IQ \rightarrow \Hom(\Gm^3,\Gm)_\IQ$. An inspection of the three linear equations given by (\ref{equation::etaconditionQ2}) tells us that the image $\pi(Z)$ is described by
\begin{equation*}
\det 
\begin{pmatrix}
n_1 a_{1} & n_2 a_{2} & n_3 a_{3} \\
a_{2} & a_{3} & a_{1} \\
a_{3} & a_{1} & a_{2}
\end{pmatrix}
= (n_1+n_2+n_3)a_1a_2a_3 - n_1a_1^3 - n_2a_2^3 - n_3a_3^3.
\end{equation*}
It is easy to check (cf.\ \cite[Chapter 10]{Mordell1969} or \cite[Section 3.1]{Dolgachev2012}) that
\begin{equation*}
n_1X^3 + n_2Y^3 + n_3Z^3 - (n_1+n_2+n_3) XYZ = 0
\end{equation*}
is the projective equation of an elliptic curve $E_{n_1,n_2,n_3}^\prime$ for generic tuples $(n_1,n_2,n_3) \in \IZ^3$. In these cases, $\pi(Z)$ is birationally equivalent to $\IP^1 \times E_{n_1,n_2,n_3}^\prime$. The existence of a global non-zero one-form (i.e., the pull-back of the invariant differential form on $E_{n_1,n_2,n_3}^\prime$) precludes unirationality of $\IP^1 \times E_{n_1,n_2,n_3}^\prime$ (cf.\ \cite[Theorem 1.52]{Kollar2004}). Therefore, $Z$ itself cannot be a rational variety. In addition, the set $Z(\IQ)$ surjects onto the Mordell-Weil group of the $\IQ$-elliptic curve $E_{n_1,n_2,n_3}^\prime$. Given that no known algorithm
produces the Mordell-Weil rank, this should demonstrate that the ``mixed structure'' of a semiabelian variety can lead to an intricate set of quotients and subgroups.
\end{example}

\textbf{Acknowledgements:} The author deeply thanks Philipp Habegger for suggesting this problem quite some years ago and for providing a constant motivation to continue working on it. The present article also profited much from discussions with and advice of Daniel Bertrand, Pietro Corvaja, Marc Hindry, Rafael von Känel, Friedrich Knop, Harry Schmidt, and Gisbert W\"ustholz. He also thanks the organizers of the workshop ``Diophantische Approximationen" (ID: 1615) at the MFO for allowing him to present the results of this article there for the first time in a preliminary fashion. In addition, the author acknowledges the supportive hospitality of both the Max Planck Institute for Mathematics and the Fields Institute, where this article has been written. Finally, he thanks the referee for their ample suggestions, which helped to improve the expository quality of this article.

\nocite{Chambert-Loir1999}

\bibliographystyle{plain}
\bibliography{../Bibliography/references}

\end{document}